\documentclass[reqno]{amsart}
\usepackage{graphicx, amsmath,amsfonts,amsthm,amssymb, paralist} 

\newtheorem{theorem}{Theorem}[section]
\newtheorem{corollary}[theorem]{Corollary}
\newtheorem{lemma}[theorem]{Lemma}
\newtheorem{proposition}[theorem]{Proposition}
\theoremstyle{definition}
\newtheorem{definition}[theorem]{Definition}
\theoremstyle{remark}
\newtheorem{remark}[theorem]{Remark}
\numberwithin{equation}{section}

\newcommand{\g}{\geqslant}

\newcommand{\RR}{\mathbb{R}}
\newcommand{\ZZ}{\mathbb{Z}}
\newcommand{\CC}{\mathbb{C}}

\newcommand{\sph}{\mathbb{S}}
\newcommand{\NN}{\mathbb{N}}
\newcommand{\p}{\partial}

\newcommand{\les}{\leqslant}
\newcommand{\lesa}{\lesssim}

\newcommand{\mc}[1]{\mathcal{#1}}
\newcommand{\mb}[1]{\mathbf{#1}}
\newcommand{\mf}[1]{\mathfrak{#1}}
\newcommand{\eref}[1]{(\ref{#1})}
\newcommand{\lr}[1]{ \langle #1 \rangle}
\newcommand{\ind}{\mathbbold{1}}

\DeclareSymbolFont{bbold}{U}{bbold}{m}{n}
\DeclareSymbolFontAlphabet{\mathbbold}{bbold}

\DeclareMathOperator*{\sgn}{sgn}
\DeclareMathOperator*{\supp}{supp}
\DeclareMathOperator*{\diam}{diam}
\DeclareMathOperator*{\diag}{diag}
\DeclareMathOperator*{\dist}{dist}

\setdefaultenum{(i)}{(a)}{1}{A}

\begin{document}

\title{On the Division Problem for the Wave Maps Equation}%
\author{Timothy Candy}%
\address{Universit\"at Bielefeld, Fakult\"at f\"ur Mathematik, Postfach 10 01 31, 33501 Bielefeld, Germany
}
\email{tcandy@math.uni-bielefeld.de}%

\author{Sebastian Herr}%
\address{Universit\"at Bielefeld, Fakult\"at f\"ur Mathematik, Postfach 10 01 31, 33501 Bielefeld, Germany}
\email{herr@math.uni-bielefeld.de}%
\subjclass{35L15,35L52}%
\keywords{wave maps, division problem, bilinear Fourier restriction, atomic spaces}%

\begin{abstract}
We consider Wave Maps into the sphere and give a new proof of small data global well-posedness and scattering in the critical Besov space, in any space dimension $n \g 2$. We use an adapted version of the atomic space $U^2$ as the single building block for the iteration space. Our approach to the so-called division problem is modular as it systematically uses two ingredients: atomic bilinear (adjoint) Fourier restriction estimates and an algebra property of the iteration space, both of which can be adapted to other phase functions.
\end{abstract}

\maketitle

\section{Introduction}\label{sec:intro}
Let $(\RR^{1+n},\eta)$ be the Minkowski space-time with metric $(\eta_{\alpha \beta})=\diag(-1,1,\ldots,1)$ and $M$ be a smooth manifold with Riemannian metric $g$. Formally, a wave map is map  $\phi:\RR^{1+n}\to M$ which is a critical point of the
Langrangian
\[
\mathcal{L}(\phi)=\frac{1}{2}\int_{\RR^{1+n}} \langle \partial^\alpha \phi,\partial_\alpha \phi \rangle_{g(\phi)} dtdx .
\]
Space-time coordinates are denoted by $(t,x)$, we use the standard summation convention,  raise indices according to $\partial^\alpha=\eta^{\alpha\beta}\partial_\beta $, and write $\Box=-\partial^\alpha\partial_\alpha=\partial_t^2-\Delta$ for the d'Alembertian.
In the extrinsic formulation, assuming that $M$ is a submanifold of some Euclidean space $\RR^m$, a wave map is a solution  $\phi:\RR^{1+n}\to M\subset \RR^m$ to
\begin{equation}\label{eq:wm-extr}
\Box \phi = -S(\phi)(\partial^\alpha \phi,\partial_\alpha\phi),
\end{equation}
where $S(p):T_pM \times T_pM\to (T_pM)^\perp$ is the second fundamental form at $p\in M$. For the purposes of this paper, the important point is that the Wave Maps equation \eqref{eq:wm-extr} takes the form of a nonlinear wave equation with null structure, more specifically
\begin{equation}\label{eq:wm-sphere}
\Box \phi = \phi (|\nabla \phi|^2-|\partial_t \phi|^2)
\end{equation}
in the case of the target manifold $M=\mathbb{S}^2\subset \RR^3$. We remark that for classical (smooth) solutions to the equation \eqref{eq:wm-sphere} one can drop the target constraint because if the initial data $(\phi(0),\partial_t \phi (0)):\RR^n \to \RR^3$ satisfy $|\phi(0)|=1 $ and $\partial_t \phi (0)\cdot \phi (0)=1$, one can prove $|\phi(t)|=1$ for all $t$.

Solutions can be rescaled according to $\phi(t,x)\to \phi(\lambda t,\lambda x)$. Therefore $\dot{H}^{\frac{n}{2}}(\RR^n)$ is the critical Sobolev regularity for global well-posedness, which barely fails to control the $L^\infty$-norm. Wave Maps conserve the energy
\[
E(\phi)=\frac{1}{2}\int_{\RR^{n}} |\partial_t \phi |^2+|\nabla \phi |^2 dx,
\]
therefore the space dimension $n=2$ is the energy-critical dimension.
It turned out that, even in the case of small initial data, the Cauchy problem is challenging to solve in the critical Sobolev space, in particular in low space dimensions $n=2,3$. For instance, the problem cannot be solved iteratively in Fourier restriction norms only \cite{Klainerman2002}. In the (smaller) critical Besov space $\dot{B}^{\frac{n}{2}}_{2,1}(\RR^n)$ the breakthrough global well-posedness and scattering result in dimension $n=2,3$ was obtained by Tataru \cite{Tataru2001}. The small data problem in the critical Sobolev space is more subtle. The example of Nirenberg \cite[p. 45]{Klainerman1980} shows that the  scalar model problem $\Box u = \partial^\alpha u \partial_\alpha u$ is ill-posed in $\dot{H}^{\frac{n}{2}}(\RR^n)$. In the critical Sobolev space the Wave Maps problem exhibits a quasilinear behaviour and a renormalization is necessary. In the case $M=\mathbb{S}^2$ this was solved by Tao \cite{Tao2001}, later by Krieger \cite{Krieger2004} for the hyperbolic plane target, and for more general targets by Klainerman-Rodnianski \cite{Klainerman2001} and Tataru \cite{Tataru2005}.

Building on the small data results, Sterbenz-Tataru \cite{Sterbenz2010a,Sterbenz2010b} and Krieger-Schlag \cite{Krieger2012} solved the global regularity problem in the energy-critical dimension $n=2$ for initial data below the threshold given by  nontrivial harmonic maps of lowest energy (if any).
We refer the reader to \cite{Shatah1998,Tao2006,Tataru2004,Koch2014,Geba2017} for more comprehensive introductions to various aspects of the theory of Wave Maps and further references.

In this paper, we revisit the problem of iteratively solving the Cauchy problem associated with \eqref{eq:wm-sphere} in the critical Besov space for small data, which was solved first in \cite{Tataru2001}.
This problem is also known as the \emph{division problem}. Citing \cite[p. 195]{Tataru2004} (see also \cite{Krieger2003}), the name stems ``from the fact that in
Fourier space the parametrix $\Box^{-1}$ for the wave equation is essentially the division by'' the symbol $|\xi|^2-\tau^2$ of the d'Alembertian, which fails to be locally integrable. In essence, it consists in constructing a function space $S^{\frac{n}{2}}$ with the properties
\begin{equation}\label{eq:mapping}
S^{\frac{n}{2}}\cdot S^{\frac{n}{2}} \to S^{\frac{n}{2}} \text{ and } \Box^{-1} (S^{\frac{n}{2}} \cdot \Box S^{\frac{n}{2}}) \to S^{\frac{n}{2}},
\end{equation}
see Theorem \ref{thm:div-prob} below for a precise statement. The division problem arises on the level of the Littlewood-Paley pieces and we do not address the \emph{summation problem}, which is the second ingredient for a proof of global well-posedness in the critical Sobolev space and was solved first in \cite{Tao2001}. We emphasize that a solution of the division problem is crucial for all later developments on Wave Maps mentioned above and the original construction \cite{Tataru2001} has been successfully used, adapted and refined in related problems, such as \cite{Bournaveas2016,Bejenaru2015,Bejenaru2016,Krieger2017,Oh2016}, among others.
Further, the division problem is universal in the sense that it crucially arises
in many other nonlinear dispersive evolution equations at the critical regularity or for global-in-time problems,
such as for Schr\"odinger maps \cite{Bejenaru2011}.

One of the key difficulties in the solution of the division problem originates in the fact that, even for solutions $u$ on the unit frequency scale, it is impossible to obtain global-in-time control $\Box u$ in $L^1_tL^2_x$. Instead, in \cite{Tataru2001} Tataru introduces characteristic (or null) coordinate frames $(t_\Theta,x_\Theta)$, for unit vectors $\Theta$ on the cone, and $t_{\Theta}=\Theta \cdot (t,x)$. If $u$ is Fourier localized in a transversal direction, it is possible to control $\Box u$ in $L^1_{t_\Theta} L^2_{x_\Theta}$. Then, by an involved construction of an atomic function space in addition to the standard Fourier restriction space he succeeds in proving the requires estimates alluded to above, which rest on certain bilinear estimates in $L^2_{t,x}$. Due to its complexity we do not describe the solution to the division problem of \cite{Tataru2001} in more detail here but refer to \cite{Tataru2001,Tao2001,Krieger2004} instead. In a first version of \cite{Tataru2001}  Tataru took a different route, based on the space of functions of bounded $p-$variation $V^p$ \cite{Wiener1924} and its predual $U^q$ \cite{Pisier1987}. In this construction, the space for solutions is an atomic space, where the atoms are normalized step-functions and each step solves the homogeneous wave equation. However, as pointed out by Nakanishi, there was a serious problem in the proof of the crucial bilinear $L^2_{t,x}$ estimates. Instead, Tataru abandonded the approach via $U^p$ and $V^p$ and developed the null frame spaces instead. The null frame space construction is custom-made for the application to the Wave Maps problem, for which it proved very successful, but the functional analysis is delicate and adaptations to closely related problems require new ideas \cite{Bejenaru2011,Bejenaru2016}.

Around the same time as \cite{Tataru2001} there have been significant advances on the Fourier restriction problem for the cone.
The key fact is that by passing to the bilinear setting it is possible to use both curvature and transversality properties of the cone.
Indeed, in dimension $n\g 2$, Wolff \cite{Wolff2001} proved that for every $p>p_n:=\frac{n+3}{n+1}$,
the bilinear (adjoint) Fourier restriction estimate
$$ \big\| e^{ - it |\nabla|} f e^{ - i t |\nabla|} g \big\|_{L^p_{t, x}(\RR^{1+n})} \lesa \| f \|_{L^2_x} \| g \|_{L^2_x},$$
holds true, provided that the Fourier-supports of $f$ and $g$ are angularly
separated and contained in the unit annulus. Shortly after, Tao \cite{Tao2001b} proved this estimate in the endpoint case $p=p_n$.

In the present paper, we prove Tataru's original conjecture to be true: it is possible to use $U^2$ as the only building block in the solution of division problem by using recent advances on the bilinear (adjoint) Fourier restriction estimates as the  key new ingredient.
Specifically, for $u:\RR^{1+n}\to \RR^3$, we consider the model problem
		\begin{equation}\label{eqn:wm model}
			\begin{split}
				\Box u &= u (|\nabla u|^2-|\partial_t u|^2)\\
					(u, \p_t u)(0) &= (f,g)
			\end{split}
		\end{equation}
for small initial data in $(f,g) \in \dot{B}^\frac{n}{2}_{2,1}(\RR^n)\times \dot{B}^{\frac{n}{2}}_{2,1}(\RR^n)$, for $n \g 2$, and provide a new proof of global well-posedness and scattering.

While we present our approach with a focus on this specific problem, we emphasize that it is modular. It is straight-forward to adapt the function space $U^2$ to any linear propagator. By the standard resonance analysis of \eqref{eqn:wm model}, the mapping properties
\eqref{eq:mapping} can be reduced to two independent building blocks, both of which are new. Firstly, if one of the factors has Fourier support far from the characteristic set, then the required estimates are consequences of bilinear Fourier multiplier versions of the following toy estimates:
\[
\|fg\|_{V^2}\lesa\|f\|_{L^{\infty}_{t,x}} \|g\|_{V^2}, \qquad \|fg\|_{U^2}\lesa\|f\|_{L^{\infty}_{t,x}} \|g\|_{U^2}
\]
for $g$ with high temporal frequency, see Subsection \ref{subsec:alg}. Secondly, if all factors have Fourier support close to  the characteristic set, we exploit atomic versions of bilinear (adjoint) Fourier restriction estimates, which are available for general phases under tranversality and curvature conditions, see Section \ref{sec:adapted bilinear restriction} and \cite{Candy2017b}. To put things into perspective, let us mention that since the spaces $U^p$ and $V^p$ have been introduced in the PDE context in \cite{Koch2005}, the theory of these spaces has been developed in \cite{Hadac2009,Koch2014,Koch2016}, among others. In parallel, since the seminal works of Wolff \cite{Wolff2001} and Tao \cite{Tao2001b}, there have been advances in the theory of bilinear Fourier restriction estimates by Bejenaru \cite{Bejenaru2017}, Lee-Vargas \cite{Lee2010}, the first named author \cite{Candy2017b}, among others. Here, we are able to connect these two lines of research and provide a systematic and modular solution of the division problem.

The paper is organized as follows. After introducing some notation, we describe the solution 
 to the division problem in Section \ref{sec:sol-div-prob}, see Theorem \ref{thm:div-prob}. Also, we introduce the function spaces and provide a proof of small data global well-posedness and scattering for the Wave Maps equation. In Section \ref{sec:multi}
we establish properties of the iteration space, prove product estimates in far cone regions
and the bilinear $L^2$ estimates, and give a proof of Theorem \ref{thm:div-prob}. Section \ref{sec:Up and Vp} is devoted to the basic properties of the critical function spaces $U^p$ and $V^p$, such as embedding properties, almost orthogonality and duality statements. In Section \ref{sec:characterisation} we provide characterisations of $U^p$ which are crucial for applications to PDE. In Section \ref{sec:bound} we establish results concerning convolution and multplication in $U^p$ and $V^p$ and introduce the adapted function spaces.
Finally, in Section \ref{sec:adapted bilinear restriction} we prove the bilinear restriction estimate in adapted function spaces which is used in Section \ref{sec:multi}.

\subsection{Notation}\label{subsec:not}
Let $\mathcal{S}(\RR^n)$ the Schwartz space and $\mathcal{S}'(\RR^n)$ the space of tempered distributions. Given a function $f \in L^2(\RR^n)$, we let $\widehat{f}(\xi)$ denote the Fourier transform, and if $u \in L^2_{t,x}(\RR^{1+n})$, we let $\widetilde{u}(\tau, \xi)$ denote the space-time Fourier transform. 

Let $P_\lambda$ denote a smooth (spatial) cutoff to the Fourier region $|\xi| \approx \lambda$. Similarly, we take $P^{(t)}_d$ denote the (temporal) Fourier projection to the set $|\tau| \approx d$. The Fourier multipliers $P_{\les \lambda}$ and $P^{(t)}_{\les d}$ are defined similarly to restrict to spatial frequencies $|\xi|\lesa \lambda $, and temporal frequencies $|\tau| \lesa d$ respectively. We often use the short hand $P_\lambda u = u_\lambda$. Let $C_d$ and $C^\pm_d$ restrict to the space-time Fourier regions $\big||\tau| - |\xi| \big| \approx d$ and $\big| \tau \pm |\xi| \big|\approx d$ respectively. Note that we may write $C^\pm_d = e^{\mp i |\nabla| t} P^{(t)}_d e^{\pm i |\nabla| t}$.

Define $\mc{C}_\alpha$ to be a finitely overlapping collection of caps $\kappa \subset \sph^{n-1}$ of radius $\alpha$ which cover the sphere, and take $\angle(\xi,\eta)$ denote the angle between $\xi, \eta \in \RR^n\setminus \{0\}$. Let $\mc{Q}_\mu$ be a collection of finitely overlapping cubes of diameter $\mu$ which form a cover of $\RR^n$. We denote the corresponding Fourier cutoffs to caps $\kappa \in \mc{C}_\alpha$ and cubes $q \in \mc{Q}_\mu$ as $R_\kappa$ and $P_q$ respectively.

For $s\in \RR$, $1\leq p,q \leq \infty$ and $f\in \mathcal{S}'(\RR^n)$, let
\[
\|f\|_{\dot{B}^{s}_{p,q}}=\Big(\sum_{\lambda \in 2^\ZZ}\lambda^{sq} \|f_\lambda \|_{L^p(\RR^n)}^q\Big)^{\frac{1}{q}},
\]
with the obvious modification for $q=\infty$. Further, let $\dot{B}^{s}_{p,q}$
denote the space of all $f\in \mathcal{S}'(\RR^n)$ satisfying \[f=\sum_{\lambda \in 2^\ZZ} f_\lambda  \text{ in }\mathcal{S}'(\RR^n) \text{ and }\|f\|_{\dot{B}^{\frac{n}{2}}_{p,1}}<+\infty.\] It is well-known (see \cite{Bahouri2011}) that, if either $s<n/p$ or $s=n/p$ and $q=1$, then $\dot{B}^{s}_{p,q}$ is a Banach space. We have the continuous embedding $\dot{B}^{\frac{n}{2}}_{2,1}\subset C_0 (\RR^n)$.

\section{A solution to the division problem}\label{sec:sol-div-prob}

\subsection{Definition of the spaces $U^p$ and $V^p$}\label{subsec:upvp}

Let \[\mb{P}= \{ \tau=(t_j)_{j=1}^{N} \mid N \in \NN, \,\, t_j \in \RR, \,\, t_j< t_{j+1} \}\] be the set of partitions, i.e.\ finite increasing sequences. For a partition $\tau=(t_j)_{j=1, \ldots, N}$, let \[\mc{I}_\tau=\{[t_1,t_{2}), \ldots, [t_{N-1},t_{N}), [t_N,\infty)\},\]
i.e.\ left-closed disjoint intervals associated with $\tau$. Let $1\leq p<\infty$. We say that $u$ is a \emph{$U^p$-atom} if there exists $\tau \in \mb{P}$ such that $u(t) = \sum_{I \in \mc{I}_\tau} \ind_{I}(t) f_I$ is a step function satisfying
    $$ \Big( \sum_{I \in \mc{I}_\tau} \| f_I \|_{L^2}^p \Big)^{\frac{1}{p}} = 1.$$
The atomic space $U^p$ is then defined to be
    $$ U^p = \Big\{ \sum_{j\in \NN} c_j u_j \, \Big| \,\, (c_j) \in \ell^1, \text{ $u_j$ is a $U^p$ atom } \Big\},$$
with the induced norm
    $$ \| u \|_{U^p} = \inf_{ u = \sum_{j\in \NN} c_j u_j } \sum_{j\in \NN} |c_j|.$$
Functions in $u\in U^p$ are bounded, have one-sided limits everywhere and are right-continuous with $\lim_{t\to -\infty}u(t)=0$ in $L^2(\RR^n)$.

Closely related to the atomic spaces $U^p$, are the $V^p$ spaces of finite $p$-variation. Given a function $v: \RR \rightarrow L^2(\RR^n)$, we define
        $$ |v |_{V^p} = \sup_{(t_j)_{j=1}^{N} \in \mb{P}} \Big( \sum_{j=1}^{N-1}\| v(t_{j+1}) - v(t_j) \|_{L^2}^p \Big)^\frac{1}{p}.$$
and $\| v \|_{V^p} =\|v\|_{L^\infty_t L^2_x}+|v |_{V^p}$. If $ |v |_{V^p} < \infty$, then $v$ has one-sided limits everywhere including $\pm \infty$. Let $V^p$ be the space of all right-continuous functions $v$ such that $| v |_{V^p} < \infty$ and $\lim_{t\to-\infty} v(t)=0$ in $L^2(\RR^n)$. Then, $\| v \|_{V^p}\leq 2|v|_{V^p}$ for all $v\in V^p$.

The spaces $V^p$ were introduced by Wiener \cite{Wiener1924} while a discrete version of the atomic $U^p$ spaces appeared in work of Pisier-Xu \cite{Pisier1987}. In the context of PDE, the spaces $U^p$ and $V^p$ were used in unpublished work of Tataru as a replacement for the endpoint $X^{s,\frac12}$ type spaces, and developed in more detail by Koch-Tataru \cite{Koch2005}. We leave a more complete discussion of the spaces $U^p$ and $V^p$ till Section \ref{sec:Up and Vp} below, further properties can also be found in \cite{Hadac2009, Koch2005,  Koch2016}.

\subsection{The solution space}\label{subsec:s}
The solution space $S^{\frac{n}{2}}$ for the Wave Maps equation is based on an adapted version of the atomic space $U^2$. More precisely, we define $U^2_\pm = e^{\mp i t |\nabla| } U^2$ with the obvious norm $$ \| u \|_{U^2_\pm} = \| e^{\pm i t|\nabla|} u \|_{U^2}.$$
Elements of $U^2_\pm$ should be thought of as being close to solutions to the linear half-wave equation. In fact, atoms in $U^2_\pm$ are piecewise solutions to $(-i \p_t \pm |\nabla|) u = 0$.
Let $S$ be the collection of all $u \in C_b( \RR; L^2_x)$ with $|\nabla|^{-1} \p_t u \in C_b(\RR; L^2_x)$ such that $  u \pm i |\nabla|^{-1} \p_t u \in U^2_\pm$ with the norm
		$$ \| u \|_S =  \big\| u  + i |\nabla|^{-1} \p_t u \big\|_{U^2_+} + \big\| u - i |\nabla|^{-1} \p_t u \big\|_{U^2_-}.$$
The space $S$ will contain the frequency localised pieces of the wave map $u$. To define the full solution space $S^{\frac{n}{2}}$ we take
	$$S^\frac{n}{2} = \{ u \in C_b(\RR, \dot{B}^\frac{n}{2}_{2,1}) \mid  P_\lambda u \in S \}$$
with the norm
	$$ \| u \|_{S^\frac{n}{2}} =  \sum_{\lambda \in 2^\ZZ} \lambda^\frac{n}{2} \| P_\lambda u \|_S.$$
The space $S^\frac{n}{2}$ is a Banach space, this is an immediate consequence of the fact that the subspace of continuous functions in $U^2_\pm$ is closed. Since we may write
        \begin{equation}\label{eq:u-pu}\begin{split} u &= \tfrac{1}{2}\big(u + i |\nabla|^{-1} \p_t u\big) + \tfrac{1}{2}\big(u - i |\nabla|^{-1} \p_t u\big),\\
 \qquad  \p_t u &= \tfrac{1}{2i}|\nabla| \big(u + i |\nabla|^{-1} \p_t u\big) - \tfrac{1}{2i}|\nabla|\big(u  - i |\nabla|^{-1} \p_t u\big)
\end{split}
\end{equation}
it is clear that we have bounds
\[\| u_\lambda \|_{L^\infty_t L^2_x}  + \| \p_t u_\lambda \|_{L^\infty_t \dot{H}^{-1}} \lesa \| u_\lambda \|_{S}\]
and
\[\| u \|_{L^\infty_t \dot{B}^{\frac{n}{2}}_{2, 1}}  + \| \p_t u \|_{L^\infty_t \dot{B}^{\frac{n}{2}-1}_{2, 1}} \lesa \| u \|_{S^\frac{n}{2}}.\]

 Let
\[V(t)(f,g)= \cos(t|\nabla|) f + \frac{\sin(t|\nabla|)}{|\nabla|}g\]
be the propagator for the homogeneous wave equation. Further, let $\chi(t) \in C^\infty$ with $\chi(t) = 1$ for $t\g 0$, and $\chi(t) = 0$ for $t<-1$.

\begin{lemma}\label{lem:lin-est}
For  all $f_\lambda ,g_\lambda \in L^2$, we have $\chi(t|\nabla|)V(t)(f_\lambda,g_\lambda) \in S$ and
\begin{equation}\label{eq:lin-est}
\|\chi(t|\nabla|)V(t)(f_\lambda ,g_\lambda ) \|_{S}\lesa \|f_\lambda\|_{L^2}+\lambda^{-1}\|g_\lambda\|_{L^2}.
\end{equation}
\end{lemma}
\begin{proof}
Let $v=V(t)(f_\lambda ,g_\lambda )$ und $w=\chi(t|\nabla|)V(t)(f_\lambda ,g_\lambda )$. Then,
\[
v\pm i |\nabla|^{-1}\partial_t v=e^{\mp i t |\nabla|} (f_\lambda \pm i |\nabla|^{-1} g_\lambda ),
\]
and therefore
\[
w\pm i |\nabla|^{-1}\partial_t w= \chi(t|\nabla|) e^{\mp i t |\nabla|} (f_\lambda \pm i |\nabla|^{-1} g_\lambda )\pm i \chi'(t|\nabla|)v.
\]
Due to \eqref{eq:w11} we have
\begin{align*}
\|\chi(t|\nabla|) e^{\mp i t |\nabla|} (f_\lambda \pm i |\nabla|^{-1} g_\lambda)\|_{U^2_\pm}={}&\|\chi(t|\nabla|) (f_\lambda\pm i |\nabla|^{-1} g_\lambda)\|_{U^2}\\
\lesa{}& \|\chi'(t|\xi|)|\xi| (\widehat{f_\lambda}\pm i |\xi|^{-1}\widehat{g_\lambda})\|_{L^1_t L^2_\xi}\\
\lesa{}&\|f_\lambda\|_{L^2}+\lambda^{-1}\|g_\lambda\|_{L^2},
\end{align*}
similarly,
\begin{align*}
\|\chi'(t|\nabla|) v_\lambda \|_{U^2_\pm}={}&\|\chi'(t|\nabla|)  e^{\pm i t |\nabla|} v_\lambda\|_{U^2}\\
\lesa{}& \big\|\partial_t \big(\chi'(t|\xi|) e^{\pm i t |\xi|}\widehat{v_\lambda}\big)\big\|_{L^1_t L^2_\xi}\\
\lesa{}&\big\|\partial_t \big(\chi'(t|\xi|) e^{\pm i t |\xi|}\big(\cos(t|\xi|) \widehat{f_\lambda}(\xi) + \frac{\sin(t|\xi|)}{|\xi|}\widehat{g_\lambda}(\xi) \big)\big\|_{L^1_t L^2_\xi}\\
\lesa{}&\|f_\lambda\|_{L^2}+\lambda^{-1}\|g_\lambda\|_{L^2},
\end{align*}
and therefore $w\in S$ with the required bound.
\end{proof}

We prove that the space $S^{\frac{n}{2}}$ is a solution to the division problem.
More precisely, if we define
        $$ \Box^{-1} F (t)= \ind_{[0,\infty)}(t) \int_0^t \frac{\sin( (t-s)|\nabla|)}{|\nabla|} F(s) ds $$
to be the solution to the wave equation $\Box u = F$ on $[0, \infty)\times \RR^n$ with vanishing data at $t=0$, then we have the following.

\begin{theorem}\label{thm:div-prob}
Let $\lambda_0, \lambda_1, \lambda_2 \in 2^\ZZ$. If $u_{\lambda_1}, v_{\lambda_2}  \in S$,  then $ P_{\lambda_0}(u_{\lambda_1} v_{\lambda_2}) \in S$ and
    \begin{equation}\label{eqn:thm div-prob:S-alg} \lambda_0^\frac{n}{2} \| P_{\lambda_0} (u_{\lambda_1} v_{\lambda_2}) \|_{S} \lesa  (\lambda_1 \lambda_2)^\frac{n}{2}  \| u \|_{S} \| v \|_{S}.
    \end{equation}
Moreover, if in addition we have $\Box v_{\lambda_2} \in L^1_{t,loc} L^2_x$, then $\Box^{-1}P_{\lambda_0}( u_{\lambda_1} \Box v_{\lambda_2}) \in S$ and
     \begin{equation}\label{eqn:thm div-prob:S-nonlin} \lambda_0^\frac{n}{2} \| \Box^{-1}P_{\lambda_0}( u_{\lambda_1} \Box v_{\lambda_2}) \|_{S} \lesa   (\lambda_1 \lambda_2)^\frac{n}{2}  \| u\|_{S} \| v\|_{S}.
     \end{equation}
\end{theorem}

\begin{remark}\label{rmk:summ-prob}
The estimates can be easily summed up. However,
a summation problem arises when one aims at solving the Wave Maps equation in the critical Sobolev space $\dot{H}^{\frac{n}{2}}\times \dot{H}^{\frac{n}{2}-1}$. The $\ell^1$ summation over frequencies has to be replaced with an $\ell^2$ sum, which would require obtaining the stronger bounds and a renormalization. Such estimates are known in the null frame based solution space, see \cite[equations (21), (27)]{Tao2001} and \cite[equation (4.3)]{Tataru2005} (the proof can be found directly following (129) in \cite{Tao2001}). We do not pursue this issue here.
\end{remark}

\subsection{Small data GWP for the Wave Maps equation}\label{subsec:gwp}
Let \[Q_0(u,v)=-\partial^\alpha u \cdot \partial_\alpha v=\partial_t u \cdot \partial_t v-\partial_{x_j} u \cdot \partial_{x_j} v.\]
It is well-known that this is a null form, i.e. it satisfies
the identity
 \begin{equation}
\label{eq:nf}
2Q_0(u,v) = \Box(uv) - \Box u v - u \Box v.
\end{equation}
Then, \eqref{eq:wm-sphere} can be written as
\[
\Box u =uQ_0(u,u).
\]
Theorem \ref{thm:div-prob} together with \eqref{eq:nf}
and a standard fixed point argument can be used to construct a solution to the Wave Maps equation \eqref{eq:wm-sphere} in the space $S^{\frac{n}{2}}$ provided that the initial data $(f,g)$ are sufficiently small in $\dot{B}^{\frac{n}{2}}_{2,1}(\RR^n)\times \dot{B}^{\frac{n}{2}-1}_{2,1}(\RR^n)$. In more detail,  set
		$$S_0^{\frac{n}{2}} = \big\{ u \in S^\frac{n}{2} \,\big| \,\forall \lambda\in 2^{\ZZ} : \; \Box u_\lambda \in L^1_{t,loc}L^2_x(\RR^{1+n}) \, \}.$$
The purpose of this subset of $S^{\frac{n}{2}}$ is to ensure that all nonlinear expressions are a-priori well-defined.
Define the map $\mc{T}:S_0^\frac{n}{2} \to S_0^{\frac{n}{2}}$ by
	$$ \mc{T}[u](t) = \chi(t|\nabla|) V(t)(f,g) + \Box^{-1}\big( u Q_0(u,u) \big).$$
Theorem \ref{thm:div-prob} and summation implies that
\begin{align*}
&\| \Box^{-1}\big( u Q_0(v,w) \big)\|_{S^{\frac{n}{2}}}\\
\leq{}&  \|\Box^{-1}\big(u \Box(vw)\big)\|_{S^{\frac{n}{2}}}+\|\Box^{-1}\big(u (\Box v) w\big)\|_{S^{\frac{n}{2}}} +\| \Box^{-1}\big(u v \Box w\big) \|_{S^{\frac{n}{2}}}\\
\lesa {}& \|u\|_{S^{\frac{n}{2}}} \|vw\|_{S^{\frac{n}{2}}}+\|uw \|_{S^{\frac{n}{2}}}\|\Box v\|_{S^{\frac{n}{2}}}
+\|u v\|_{S^{\frac{n}{2}}} \| w \|_{S^{\frac{n}{2}}}\\
\lesa {}& \|u\|_{S^{\frac{n}{2}}}\|v\|_{S^{\frac{n}{2}}} \| w \|_{S^{\frac{n}{2}}},
\end{align*}
where we have used \eqref{eq:nf} in the first, \eqref{eqn:thm div-prob:S-nonlin} in the second and \eqref{eqn:thm div-prob:S-alg} in the third inequality.
By Lemma \ref{lem:lin-est} and summation we obtain
\begin{align*}
\|\mc{T}[u]\|_{S^{\frac{n}{2}}}\lesa \|(f,g)\|_{\dot{B}^{\frac{n}{2}}_{2,1}\times \dot{B}^{\frac{n}{2}-1}_{2,1}}+\|u\|_{S^{\frac{n}{2}}}^3,\\
\|\mc{T}[u]-\mc{T}[v]\|_{S^{\frac{n}{2}}}\lesa \big(\|u\|_{S^{\frac{n}{2}}}^2+\|v\|_{S^{\frac{n}{2}}}^2\big)\|u-v\|_{S^{\frac{n}{2}}}.
\end{align*}
Therefore, $\mc{T}$ is a contraction in  a small ball in $S_0^{\frac{n}{2}}$. This implies the existence of a fixed point in $S^{\frac{n}{2}}$ as the latter space is complete. Then, the restriction of the fixed point to the interval $[0,\infty)$ is a generalized solution of \eqref{eq:wm-sphere}. Clearly, if in addition the initial data are $C^\infty$, we obtain a classical solution to \eqref{eqn:wm model}. Scattering is an immediate consequence of \eqref{eq:u-pu} and the existence of one-sided limits in $U^2$. Indeed, it implies the existence of
\[\lim_{t\to \infty} e^{\pm i t |\nabla|} \big(u(t) \pm i| \nabla|^{-1} \partial_t u(t)\big)=:f_\pm \in \dot{B}^{\frac{n}{2}}_{2,1}(\RR^n).\]
Now,
\[
f_\infty:= \tfrac{1}{2}\big(f_++f_-\big) \in \dot{B}^{\frac{n}{2}}_{2,1}(\RR^n)\text{ and }g_\infty:=\tfrac{i}{2}|\nabla|\big(f_--f_+\big)\in \dot{B}^{\frac{n}{2}-1}_{2,1}(\RR^n)
\]
satisfy
\[
\lim_{t\to \infty}\|u(t)-V(t)(f_\infty,g_\infty)\|_{L^\infty_t \dot{B}^{\frac{n}{2}}_{2, 1}} +\lim_{t\to \infty}\|\partial_t u(t)-\partial_t V(t)(f_\infty,g_\infty)\|_{L^\infty_t \dot{B}^{\frac{n}{2}-1}_{2, 1}}=0.
\]
In fact, the result is slightly stronger. For instance, the embedding $U^2\subset V^2$ implies that the quadratic variation is finite. The analogous results on $(-\infty,0]$ follow from time reversability.

\section{Multilinear estimates}\label{sec:multi}

\subsection{Properties of the iteration space $S$} \label{subsec:prop-S}

In this subsection we describe the main properties of the space $S$ which are needed for the solution of the division problem. The results collected here are for the most part consequences of  properties of the $U^p$ and $V^p$ spaces. To aid the reader, we state (and for the most part give a proof of) these properties in Sections \ref{sec:Up and Vp}, \ref{sec:characterisation} and \ref{sec:bound}.

 We start by defining a weak version of the space $S$, which is based on $V^2$ rather than $U^2$. Let $S_w$ denote the collection of all right continuous functions $v$ such that 	
        $$ \| v \|_{S_w} = \| v\|_{V^2_+ + V^2_-}= \inf_{ v = v^+ + v^-\atop v^{\pm}\in V^2_{\pm}} \Big(\| v^+ \|_{V^2_+} + \| v^- \|_{V^2_-}\Big)<\infty$$
 where we define $\| \phi \|_{V^2_\pm} = \| e^{ \pm i t|\nabla|} \phi \|_{V^2}$. It is clear that
	\begin{equation}\label{eqn:S_w norm weaker} \| u \|_{S_w} +\| |\nabla|^{-1}\p_t u \|_{S_w} \les  \|u + i |\nabla|^{-1} \p_t u\|_{V^2_+} + \|u - i |\nabla|^{-1} \p_t u\|_{V^2_-}\lesa \| u \|_S,
	\end{equation}
and
    \begin{equation}\label{eqn:S contr U2}
    \| u \|_{U^2_+ + U^2_-} + \| |\nabla|^{-1} \p_t u \|_{U^2_++ U^2_-} \les \| u \|_S.
    \end{equation}
Thus the norm on  $S_w$ is weaker than that on $S$ and $S$ is slightly stronger than just have taking $u, |\nabla|^{-1} \p_t u \in U^2_+ + U^2_-$. However, for space-time frequencies localised to a fixed dyadic distance from the cone, these spaces are all closely related, see Part (ii) of Lemma \ref{lem:S properties} below. In particular, in some sense the differences in these spaces only appear when considering frequency regions of the form $\{ ||\tau| - |\xi|| \lesa \mu \}$.

The space $S$ satisfies the following properties.

\begin{lemma}\label{lem:S properties}\leavevmode
\begin{enumerate}
  \item \label{itm:lem S prop:dp}\emph{(Dual pairing)}  Let $\psi \in S$ and $\phi \in S_w$ with $|\nabla|^{-1} \Box \phi,  \p_t \phi, |\nabla|\phi   \in L^1_t L^2_x$. Then
  		$$ \Big| \int_\RR \lr{|\nabla|^{-1} \Box \phi, \psi}_{L^2} dt \Big| \lesa \| \phi \|_{S_w} \| \psi \|_S.$$

  \item \label{itm:lem S prop:Xsb}\emph{($\dot X^{0, \frac12,\infty}$ control)} Let $d, \lambda \in 2^\ZZ$. Then
            $$ \| C_d u_\lambda \|_{L^2_{t,x}} \approx  d^{-\frac{1}{2}}\frac{\lambda }{d+\lambda} \|  C_d u_\lambda  \|_{S} \approx d^{-\frac{1}{2}} \|  C_d v_\lambda  \|_{S_w}. $$
  \item \label{itm:lem S prop:sq sum}\emph{(Square sum bounds)}  Let $\epsilon>0$. For any $0<\alpha \les 1$ and $\lambda \g \mu >0$ we have
        \[\Big( \sum_{ q \in \mc{Q}_\mu} \| P_q u_\lambda \|_S^2 \Big)^\frac{1}{2}\lesa \| u_\lambda \|_{S},  \qquad  \Big( \sum_{\kappa \in \mc{C}_\alpha} \| R_\kappa u_\lambda \|_S^2 \Big)^\frac{1}{2} \lesa \| u_\lambda \|_{S}.\]
      \item \label{itm:lem S prop:dis} \emph{(Uniform disposability)} For any $d \in 2^\ZZ$ we have
\begin{align*}
\|C_{d} u \|_{S}+\|C_{\les d} u \|_{S}\lesa \|u\|_{S}, \qquad
\|C_{d} v \|_{S_w}+\|C_{\les d} v \|_{S_w}\lesa \|v\|_{S_w}.
\end{align*}
\end{enumerate}
\end{lemma}
\begin{proof}
We start with the proof of \eref{itm:lem S prop:dp}, which is a consequence of an approximation argument, together with the $U^2$ and $V^2$ version
    \begin{equation}\label{eqn:lem S prop:U2 dual pairing}
        \Big| \int_\RR \lr{ \p_t v, u}_{L^2} dt \Big| \les \| v \|_{V^2} \| u \|_{U^2}
    \end{equation}
which holds provided that $\p_t v \in L^1_t L^2_x$, see Theorem \ref{thm:dual pairing} below. The definition of the space $S_w$, implies that we can write $\phi = \phi_+ + \phi_-$ with $\phi_\pm \in V^2_\pm$. However we have a slight difficulty as the functions $\phi_\pm$ do not necessarily inherit the smoothness or integrability properties of $\phi$. To address this problem, we define the phase space localisation operator $\mc{P}_{\les d} \phi = P^{(t)}_{\les d} P_{\les d} ( \rho_d \phi)$ where $\rho \in C^\infty_0$ with $\rho(t) = 1$ on $|t|\les 1$, and we take $\rho_d(t) = \rho(\frac{t}{d})$. The assumptions on $\phi$ imply that $\| |\nabla|^{-1} \Box (1-\mc{P}_{\les d})\phi \|_{L^1_t L^2_x} \to 0$ as $d \to \infty$. Thus it suffices to bound the dual pairing with $\phi$ replaced with $\mc{P}_{\les d} \phi$. To this end, as $\mc{P}_{\les d} \phi_\pm$ is now smooth and integrable, we have
    \begin{align*}
      \Big| \int_\RR \lr{ |\nabla|^{-1} \Box \mc{P}_{\les d} \phi_\pm, \psi }_{L^2} dt \Big|
                &= \Big| \int_\RR \lr{ ( -i \p_t \pm |\nabla|) \mc{P}_{\les d} \phi_\pm, \psi \pm i |\nabla|^{-1} \p_t \psi }_{L^2} dt \\
                &= \Big| \int_\RR \lr{ \p_t ( e^{\pm it|\nabla|} \mc{P}_{\les d} \phi_\pm), e^{\pm it|\nabla|} (\psi \pm i |\nabla|^{-1} \p_t \psi) }_{L^2} dt \\
                &\les \| \mc{P}_{\les d} \phi_\pm \|_{V^2_\pm} \| \psi \pm i |\nabla|^{-1} \p_t \psi \|_{U^2_\pm}
    \end{align*}
where we applied \eref{eqn:lem S prop:U2 dual pairing}. If we now observe that
    $$  e^{\pm i t|\nabla|} P^{(t)}_{\les d} \phi(t) = \int_\RR d^{-1} \chi( \tfrac{s}{d}) e^{ \pm i s |\nabla|}  e^{ \pm i (t-s) |\nabla|} \phi(t-s) ds $$
for some $\chi \in L^1$ with $\| \chi\|_{L^1(\RR)}\les 1$,  a short computation gives the bound
\[\| \mc{P}_{\les d} \phi \|_{V^2_\pm} \lesa \| \phi \|_{V^2_\pm}.\]
 Consequently \eref{itm:lem S prop:dp} follows.

To prove \eref{itm:lem S prop:Xsb}, we first observe that $\big| |\tau| - |\xi| \big| \les \big| \tau \pm |\xi| \big|$.
Now, if $C_d v_\lambda=v_++v_-$, from Theorem \ref{thm:besov embedding} and \eqref{eqn:adapted Up vs Xsb}
\[
\| C_d v_\lambda \|_{L^2_{t,x}}\lesa \| C^+_{\gtrsim d} v_+ \|_{L^2_{t,x}} + \| C^{-}_{\gtrsim d} v_- \|_{L^2_{t,x}}\lesa d^{-\frac12}\big(\|v_+\|_{V^2_+}+\|v_-\|_{V^2_-}\big),
\]
proving $\| C_d v_\lambda \|_{L^2_{t,x}} \lesa d^{-\frac{1}{2}} \|  C_d v_\lambda  \|_{S_w} $.  If $d\lesa \lambda$, this also implies
$$ \| C_d u_\lambda \|_{L^2_{t,x}} \lesa  d^{-\frac{1}{2}}\frac{\lambda }{d+\lambda} \|  C_d u_\lambda  \|_{S}.$$
On the other hand, to obtain the high modulation gain in the region $d\gg \lambda$, we use the above together with the fact that the $S$ norm controls the time derivative, to see that

\begin{align*}
\| C_d u_\lambda \|_{L^2_{t,x}} \approx{}& \frac{\lambda}{d} \| C_d \p_t |\nabla|^{-1} u_\lambda \|_{L^2_{t,x}}\lesa d^{-\frac{1}{2}}\frac{\lambda }{d+\lambda} \| \p_t |\nabla|^{-1} C_d u_\lambda  \|_{S_w}\\
\lesa{}& d^{-\frac{1}{2}}\frac{\lambda }{d+\lambda}\|  C_d u_\lambda  \|_{S}
\end{align*}
where we used \eqref{eqn:S_w norm weaker}. This completes the proof of all $\lesa$ inequalities in (ii). For the converse inequalities, let $P^{(t)}_\pm$ denote the temporal Fourier multiplier with symbol $\ind_{\{ \pm \tau \g 0\}}(\tau)$. Then the identity
$C_d u_\lambda = C^+_{\approx d} P^{(t)}_- C_d u_\lambda + C^+_{\approx d+ \lambda} P^{(t)}_+ C_d u_\lambda$
implies that
	\begin{align*}
		\| & C_d (u_\lambda + i |\nabla|^{-1} \p_t u_{\lambda}) \|_{U^2_+}\\
				&\les \big\| C^+_d P^{(t)}_- C_d\big( u_\lambda + i |\nabla|^{-1} \p_t  u_{\lambda}\big) \big\|_{U^2_+} + \big\| C^+_{d+\lambda} P^{(t)}_+ C_d \big(u_\lambda + i |\nabla|^{-1} \p_t u_{\lambda} \big)\big\|_{U^2_+} \\
				&\lesa d^{\frac{1}{2}} \big\| C^+_d P^{(t)}_- C_d \big( u_\lambda + i |\nabla|^{-1} \p_t  u_{\lambda}\big) \big\|_{L^2_{t,x}} \\
				&\qquad \qquad+ (d+\lambda)^{\frac{1}{2}} \big\| C^+_{d+\lambda} P^{(t)}_+ C_d\big(u_\lambda + i |\nabla|^{-1} \p_t  u_{\lambda} \big)\big\|_{L^2_{t,x}} \\
				&\lesa d^{\frac{1}{2}} \frac{d+\lambda}{\lambda} \big\| C_d u_\lambda \big\|_{L^2_{t,x}} + (d+\lambda)^{\frac{1}{2}} \frac{d}{\lambda} \big\|  C_d u_{\lambda} \big\|_{L^2_{t,x}} \approx d^\frac{1}{2} \frac{d+\lambda}{\lambda} \|  C_d u_\lambda \|_{L^2_{t,x}}.
	\end{align*}
Since an identical argument gives the $U^2_-$ version, we conclude that
		$$ \| C_d u_\lambda \|_{L^2_{t,x}} \gtrsim d^{-\frac{1}{2}} \frac{\lambda}{d+\lambda} \| C_d u_\lambda\|_{S}.$$
For the $S_w$ version, we observe that by definition
	\begin{align*} \| C_d v_\lambda \|_{S_w} &\les \| C_d P^{(t)}_- v_\lambda \|_{V^2_+} + \| C_d P^{(t)}_+ v_\lambda \|_{V^2_-} \\
	&\lesa \| C^+_{\approx d} C_d P^{(t)}_- v_\lambda \|_{V^2_+} + \| C^{-}_{\approx d} C_d P^{(t)}_+ v_\lambda \|_{V^2_-}
	\lesa d^{\frac{1}{2}} \| C_d v_\lambda \|_{L^2_{t,x}}
	\end{align*}
as required.

To prove \eref{itm:lem S prop:sq sum}, the square sum bounds, we observe that the $S$ case follows immediately from the square sum control in $U^2$, namely Proposition \ref{prop:orthog} below.

For \eref{itm:lem S prop:dis}, it suffices to show that
	$$ \| C_d \phi_\lambda \|_{U^2_+} + \| C_{\les d} \phi_\lambda \|_{U^2_+} \lesa \| \phi_\lambda \|_{U^2_+}, \qquad  \| C_d \phi_\lambda \|_{V^2_+} + \| C_{\les d} \phi_\lambda \|_{V^2_+} \lesa \| \phi_\lambda \|_{V^2_+}. $$
After writing $ C_{\les d}^+ = e^{-i t |\nabla|} P^{(t)}_{\les d} e^{i t|\nabla|} $ and $ C_{ d}^+ = e^{-i t |\nabla|} P^{(t)}_{d} e^{i t|\nabla|} $ where $P^{(t)}_{\les d}$ and $P^{(t)}_d$ are smooth temporal cutoffs to $|\tau| \les d$ and $|\tau| \approx d$ respectively, the fact that convolution with $L^1_t$ kernels is bounded in $U^2$ and $V^2$ (see Lemma \ref{lem:conv oper on Up} below), implies that
	\begin{equation}\label{eqn:lem S prop:Cplus dis} \| C_d^+ \phi_\lambda \|_{U^2_+} + \| C_{\les d}^+ \phi_\lambda \|_{U^2_+} \lesa \| \phi_\lambda \|_{U^2_+}, \qquad  \| C_d^+ \phi_\lambda \|_{V^2_+} + \| C_{\les d}^+ \phi_\lambda \|_{V^2_+} \lesa \| \phi_\lambda \|_{V^2_+}.
	\end{equation}
Thus our goal is to replace $C_d^+$ in \eref{eqn:lem S prop:Cplus dis} with $C_d$. We only prove the $U^2_+$ case, as the $V^2_+$ case is similar. For the $C_d$ multipliers, we observe that after decomposing
	$$ C_d \phi_\lambda = C_{\approx d}^+ C_d \phi_\lambda +  ( 1- C_{\approx d}^+ ) C_d\phi_\lambda  = C_{\approx d}^+ C_d\phi_\lambda  +  C^+_{\approx \lambda} ( 1- C_{\approx d}^+ ) C_d \phi_\lambda$$
the standard Besov embedding in Theorem \ref{thm:besov embedding} gives
	\begin{align*} \| C_d \phi_\lambda \|_{U^2_+} &\lesa d^{\frac{1}{2}} \|C_{\approx d}^+ C_d \phi_\lambda \|_{L^2_{t,x}} + \lambda^\frac{1}{2} \|C^+_{\approx \lambda} ( 1- C_{\approx d}^+ ) C_d \phi_\lambda\|_{L^2_{t,x}} \\
	&\lesa \sup_{d'} (d')^\frac{1}{2} \| C_{d'}^+ \phi_\lambda \|_{L^2_{t,x}} \lesa \| \phi_\lambda\|_{U^2_+}.
	\end{align*}
On the other hand, for the $C_{\les d}$ multipliers, we first write $C_{\les d} = C_{\ll d}^+ + C_{\gtrsim d}^+ C_{\les d}$. The first term is immediate by \eqref{eqn:lem S prop:Cplus dis}. For the second, we note that
	$$ C_{\gtrsim d}^+ C_{\les d} \phi_\lambda = C_{\approx d}^+ C_{\les d} \phi_\lambda + C_{\gg d}^+ C_{\les d}\phi_\lambda =  C_{\approx d}^+ C_{\les d} \phi_\lambda + C^+_{\approx \lambda} C_{\gg d}^+ C_{\les d}\phi_\lambda $$
and apply the reasoning used in the $C_d$ case.
\end{proof}

For later use, we note that \eref{itm:lem S prop:Xsb} in Lemma \ref{lem:S properties} implies that for any $d \lesa \lambda$ we have the bounds
    \begin{equation}\label{eqn:trivial box bound}
        \big\| \Box C_d \psi_\lambda \big\|_{L^2_{t,x}} +  \big\| \Box C_{\les d} \psi_\lambda \big\|_{L^2_{t,x}} \lesa d^\frac{1}{2} \lambda \|\psi_\lambda \|_S.
    \end{equation}
The norm $\| \cdot \|_S$ also controls the Strichartz type spaces $L^q_t L^r_x$. In fact, as the $S$ norm is based on $U^2$, essentially any estimate for the free wave equation which involves $L^p_t$ with $p \g 2$ implies a corresponding bound for functions in $S$. However, somewhat surprisingly, we make no use of Strichartz estimates in the proof of Theorem \ref{thm:div-prob}, and instead rely on bilinear $L^2_{t,x}$ estimates together with the $X^{s,b}$ type bound \eref{eqn:trivial box bound}.

The space $S$ is constructed using the atomic space $U^2_\pm$. Although this definition is convenient for proving properties of functions in $S$, it is a challenging problem to determine precisely when a general function $u \in C_b(\RR, L^2_x)$ belongs to $S$. This problem is closely related to the difficult question of characterising elements of $U^p$, which has been addressed in \cite{Koch2016}. We also give a slightly different statement of the $U^p$ characterisation result in Theorem \ref{thm:characteristion of Up}, as well as a self-contained proof closely following that given by Koch-Tataru \cite{Koch2016}. Restated in terms of the solution space $S$, the conclusion is the following.

\begin{theorem}[Characterisation of $S$] \label{thm:chara of S}
Let $(\psi,\p_t \psi)\in C_b(\RR ;  L^2_x \times \dot{H}^{-1}_x)$ with $\| (\psi(t), \p_t \psi(t))\|_{L^2 \times \dot{H}^{-1}} \to 0$ as $t \to -\infty$. If
                $$ \sup_{\substack{\phi \in C^\infty_0 \\ \| \phi \|_{S_w} \les 1}} \Big| \int_\RR \lr{ \Box \phi, |\nabla|^{-1}  \psi}_{L^2_x} dt \Big| <  \infty $$
then $\psi \in S$ and
                $$ \| \psi \|_S \lesa \sup_{\substack{\phi \in C^\infty_0 \\ \| \phi \|_{S_w} \les 1}} \Big| \int_\RR \lr{ \Box \phi, |\nabla|^{-1}  \psi}_{L^2_x} dt \Big|. $$

\end{theorem}
\begin{proof}
Theorem \ref{thm:characteristion of Up} implies that if $u \in L^\infty_t L^2_x$ with $\| u(t) \|_{L^2_x} \to 0 $ as $t \to -\infty$, and
        $$ \sup_{\substack{ v \in C^\infty_0(\RR;L^2) \\ \| v \|_{V^2}\les 1}} \Big| \int_\RR \lr{\p_t v, u }_{L^2_x} dt \Big| < \infty, $$
then $u \in U^2$. To translate this statement to $S$, we first observe that as in the proof of \eref{itm:lem S prop:dp} in Lemma \ref{lem:S properties}, we have
    \begin{align*} \sup_{\substack{ v \in C^\infty_0 \\ \| v \|_{V^2} \les 1}} \Big| \int_\RR \lr{ \p_t v, e^{\mp it |\nabla|}( \psi \pm i |\nabla| \p_t \psi) }_{L^2} dt\Big| &=  \sup_{\substack{ v \in C^\infty_0 \\ \| v \|_{V^2} \les 1}} \Big| \int_\RR \lr{ \Box ( e^{\mp it |\nabla|} v), |\nabla|^{-1} \psi }_{L^2} dt \Big| \\
                          &\les \sup_{\substack{  \phi \in C^\infty_0 \\ \| \phi \|_{S_w} \les 1}} \Big| \int_\RR \lr{ \Box \phi, |\nabla|^{-1} \psi }_{L^2} dt\Big|.
    \end{align*}
Consequently, since $ \|\psi \pm i |\nabla|^{-1} \p_t \psi \|_{L^2_x} \to 0$ as $ t \to -\infty$, we conclude from the characterisation of $U^2$ that $\psi \pm i |\nabla|^{-1} \p_t \psi \in U^2_\pm$ and moreover that the claimed bound holds.
\end{proof}

Recall that we have defined the inhomogeneous solution operator $\Box^{-1} F$ as
    $$ \Box^{-1} F = \ind_{[0, \infty)}(t) \int_0^t |\nabla|^{-1} \sin\big( ( t-s)|\nabla|\big) F(s) ds. $$
Applying Theorem \ref{thm:chara of S} to the special case $\psi = \Box^{-1} F$ gives the following.

\begin{corollary}[Energy Inequality]\label{cor:energy ineq}
Let $F \in L^1_{t,loc} \dot{H}^{-1}_x$ with
            $$ \sup_{\substack{\phi \in C^\infty_0 \\ \| \phi \|_{S_w} \les 1}} \Big| \int_0^\infty \lr{\phi, |\nabla|^{-1} F }_{L^2_x} dt \Big| < \infty. $$
Then $\Box^{-1} F \in S$ and
            $$ \| \Box^{-1} F \|_S \lesa \sup_{\substack{\phi \in C^\infty_0 \\ \| \phi \|_{S_w} \les 1}} \Big| \int_0^\infty \lr{\phi, |\nabla|^{-1} F }_{L^2_x} dt \Big|. $$
\end{corollary}
\begin{proof}
We would like to apply Theorem \ref{thm:chara of S} to $\Box^{-1} F$. We start by observing that the definition of $\Box^{-1} F$ and fact that $F \in L^1_{t, loc} \dot{H}^{-1}$ implies that for any $\phi \in C^\infty_0$ we have
    $$ \Big| \int_\RR \lr{ \Box \phi, |\nabla|^{-1} \Box^{-1} F}_{L^2} dt \Big| = \Big| \int_0^\infty  \lr{ \phi, |\nabla|^{-1} F }_{L^2} dt \Big|$$
as well as $\Box^{-1} F, |\nabla|^{-1} \p_t ( \Box^{-1} F) \in C(\RR, L^2)$ and $\Box^{-1} F(t) = |\nabla|^{-1} \p_t ( \Box^{-1} F)(t)=0$ for $t \les 0$.  In view of Theorem \ref{thm:chara of S}, the required conclusion would follow provided that $\| \Box^{-1} F \pm i |\nabla|^{-1} \p_t( \Box^{-1} F) \|_{L^\infty_t L^2_x} < \infty$. To this end, we observe that
    \begin{align*}
      \| \Box^{-1} F \pm i &|\nabla|^{-1} \p_t \Box^{-1} F \|_{L^\infty_t L^2_x} \\
      &= \Big\| \ind_{[0, \infty)}(t) \int_0^t e^{ \mp i (t-s)|\nabla|} |\nabla|^{-1} F(s) ds \Big\|_{L^\infty_t L^2_x} \\
      &= \sup_{ \substack{\phi \in C^\infty_0 \\  \| \phi \|_{L^1_t L^2_x} \les 1 }} \Big| \int_\RR \int_0^t \lr{\phi(t), e^{ \pm i(t-s)|\nabla|} |\nabla|^{-1} F(s) } ds\, dt \Big| \\
      &= \sup_{ \substack{\phi \in C^\infty_0 \\ \| \phi \|_{L^1_t L^2_x} \les 1 }} \Big| \int_0^\infty \Big\langle \int_s^\infty  e^{ \pm i (s-t)|\nabla|} \phi(t) dt , |\nabla|^{-1} F(s)\Big \rangle ds \Big|.
    \end{align*}
If we let $\phi_\pm(s) =  e^{ \pm  i s |\nabla|} \int_s^\infty  e^{ \mp i t |\nabla|} \phi(t) dt$, then the definition of the $V^2_\pm$ norm gives
    $$ \| \phi_{\pm} \|_{V^2_\pm} \lesa \| \phi \|_{L^1_t L^2_x} \les 1$$
and hence we conclude that
    \begin{align*}
\| \Box^{-1} F \pm i |\nabla|^{-1} \p_t \Box^{-1} F \|_{L^\infty_t L^2_x} \lesa{}& \sup_{ \substack{ \phi \in C^\infty_0 \\ \| \phi \|_{V^2_\pm} \les 1 }} \Big| \int_0^\infty \lr{ \phi, |\nabla|^{-1} F }_{L^2} dt \Big|\\
 \lesa {}&  \sup_{ \substack{ \phi \in C^\infty_0 \\ \| \phi \|_{S_w} \les 1 }} \Big| \int_0^\infty  \lr{ \phi, |\nabla|^{-1} F }_{L^2} dt \Big|,
    \end{align*}
    where the last line follows from the definition of $\| \cdot \|_{S_w}$. Therefore,
\[\| \Box^{-1} F \pm i |\nabla|^{-1} \p_t( \Box^{-1} F) \|_{L^\infty_t L^2_x} < \infty,\] as required.
\end{proof}

\begin{remark}\label{rmk:cutoff}
It is possible to remove the time cutoff $\ind_{[0, \infty)}(t)$ in Corollary \ref{cor:energy ineq} by using an approximation argument. More precisely, provided that $F \in L^1_{t,loc} \dot{H}^{-1}$, we have
    $$ \sup_{\substack{\phi \in C^\infty_0 \\ \| \phi \|_{S_w} \les 1}} \Big| \int_0^\infty \lr{\phi, |\nabla|^{-1} F }_{L^2_x} dt \Big| \les \sup_{\substack{\phi \in C^\infty_0 \\ \| \phi \|_{S_w} \les 1}} \Big| \int_\RR \lr{\phi, |\nabla|^{-1} F }_{L^2_x} dt \Big|.$$
This follows by writing $\phi(t) = \chi(\epsilon^{-1} t) \phi(t) + (1-\chi(\epsilon^{-1} t)) \phi(t)$ with $\chi(t) \in C^\infty$, $\chi(t) =1$ for $t \g 1$, and $\chi(t)= 0 $ for $t<\frac{1}{2}$, and noting that $\chi(\epsilon^{-1} t) \phi(t) \in C^\infty_0(0, \infty)$ and
        $$ \Big| \int_0^\infty \lr{ (1 - \chi(\epsilon^{-1} t) ) \phi(t), |\nabla|^{-1} F}_{L^2_x} dt \Big| \to 0$$
as $\epsilon \to 0$ since $F \in L^1_{t, loc} \dot{H}^{-1}$.
\end{remark}

\subsection{Product estimates in far cone regions}\label{subsec:product}

In this section we give the key estimates to control the product of two functions in the spaces $S$ and $S_w$ in the special case where one of the functions has high modulation, or in other words, is far from the cone. This estimate is a consequence of a general product high-low type product estimate in adapted versions $U^p$ and $V^p$. To state the product estimate we require, we start by defining the (spatial) bilinear  Fourier multiplier
	\begin{equation}\label{eqn:spatial bilinear Fourier multiplier}
		 \mc{M}[f, g]( x) = \int_{\RR^n} \int_{\RR^n} \widehat{f}(\xi - \eta) \widehat{g}(\eta)  m(\xi - \eta, \eta) d\eta e^{ i x\cdot \xi} d\xi
	\end{equation}
where $m: \RR^n \times \RR^n \to \CC$.

\begin{theorem}[High-low product estimate: wave case]\label{thm:high-low U2 wave}
Let $d \in 2^\ZZ$. Suppose that $m:\RR^n \times \RR^n \to \CC$ with
		$$ \big|\pm_0 |\xi+\eta| - \pm_1 |\xi| - \pm_2 |\eta| \big| \les d $$
for $(\xi, \eta) \in \supp m$. If $u \in V^2_{\pm_1}$ with $\supp \widetilde{u} \subset \{ |\tau \pm_1 |\xi| | \g 4 d\}$ and $v \in V^2_{\pm_2}$ then  $\mc{M}[u,v]  \in V^2_{\pm_0}$ with
	$$ \| \mc{M}[u,v] \|_{V^2_{\pm_0}} \lesa \| m(\xi, \eta) \|_{L^\infty_\xi L^2_\eta + L^\infty_\eta L^2_\xi} \| u \|_{V^2_{\pm_1}} \| v \|_{V^2_{\pm_2}}.$$
If in addition we have $u \in U^2_{\pm_1}$ and $v \in U^2_{\pm_2}$, then $\mc{M}[u,v] \in U^2_{\pm_0}$ and
	$$ \| \mc{M}[u,v] \|_{U^2_{\pm_0}} \lesa \| m(\xi, \eta) \|_{L^\infty_\xi L^2_\eta + L^\infty_\eta L^2_\xi} \| u \|_{U^2_{\pm_1}} \| v \|_{U^2_{\pm_2}}.$$
\end{theorem}
\begin{proof}
This is a direct consequence of Theorem \ref{thm:general adapted high mod prod} and a rescaling argument. Roughly the point is that, after applying Plancherel, for the $V^2_{\pm}$ estimate, by definition, we need to prove that
	\begin{align*}
	 \Big\| \int_{\RR^n} e^{ i t ( \pm_0 |\xi| - \pm_1 |\xi-\eta| - \pm_2 |\eta|)} &m(\xi - \eta, \eta) \phi(t,\xi-\eta) \psi(t,\eta) d\eta \Big\|_{V^2}\\
		&\lesa \| m(\xi, \eta) \|_{L^\infty_\xi L^2_\eta + L^\infty_\eta L^2_\xi} \| \phi\|_{V^2} \| \psi \|_{V^2}.
	\end{align*}
However, the Fourier support assumptions imply that $\mc{F}_t( \phi) \subset \{ |\tau| \g 4d \}$ (and thus $\phi$ has high temporal frequency), and $e^{ i t ( \pm_0 |\xi| - \pm_1 |\xi-\eta| - \pm_2 |\eta|)}$ has low temporal frequency. Hence, using the standard heuristic that the derivative only falls on the high frequency term, we see that the $V^2$ norm only hits the product $\phi(t) \psi(t)$, and we can simply place the exponential factor in $L^\infty_t$. This argument is made precise in Subsection \ref{subsec:alg}.
\end{proof}

The high-low product estimate in Theorem \ref{thm:high-low U2 wave} is also true for general phases, see Theorem \ref{thm:stab with conv} and Theorem \ref{thm:general adapted high mod prod} below. In particular, Theorem \ref{thm:high-low U2 wave} is a consequence of a general property of $U^p$ and $V^p$ spaces, and the precise nature of the solution operator $e^{\pm i t|\nabla|}$ plays no role. \\

Adapting Theorem \ref{thm:high-low U2 wave} to the solution spaces $S$ and $S_w$ gives the following.

\begin{theorem}\label{thm:high-mod}
For all $\lambda_0,\lambda_1,\lambda_2\in 2^\ZZ$ we have
\begin{align}
\label{eq:high-mod1}
\|P_{\lambda_0}(C_{\gg \mu} u_{\lambda_1} v_{\lambda_2})\|_{S_w}\lesa{}& \mu^{\frac{n}{2}} \Big(\frac{\min\{\lambda_1,\lambda_2\}}{\mu}\Big)^{\frac12} \|u_{\lambda_1}\|_{S_w}\|v_{\lambda_2}\|_{S_w},\\
\label{eq:high-mod2}
\|P_{\lambda_0}(C_{\gg \mu} u_{\lambda_1} v_{\lambda_2})\|_{S}\lesa{}& \mu^{\frac{n}{2}} \Big(\frac{\min\{\lambda_1,\lambda_2\}}{\mu}\Big)^{\frac32} \|u_{\lambda_1}\|_{S}\|v_{\lambda_2}\|_{S},
\end{align}
where $\mu=\min\{\lambda_0,\lambda_1,\lambda_2\}$.
\end{theorem}
\begin{proof}
Let $\lambda_{\max}=\max\{\lambda_0,\lambda_1,\lambda_2\}$. By dropping $C_{\gg \mu}$  and supposing throughout that the Fourier support of $u_{\lambda_1}$ or $v_{\lambda_2}$ is contained in $\{||\tau|-|\xi||\gg \mu\}$, we may assume that $\lambda_1\geq \lambda_2$. We claim
\begin{align}
\|P_{\lambda_0}(u_{\lambda_1} v_{\lambda_2} )\|_{V^2_{\pm_0}}\lesa{}& \mu^{\frac{n}{2}}\Big(\frac{\lambda_{\max}}{\mu}\Big)^{\frac12}\|u_{\lambda_1}\|_{V^2_+ + V^2_-}\|v_{\lambda_2}\|_{V^2_+ + V^2_-},\label{eq:pv}\\
\|P_{\lambda_0}(u_{\lambda_1} v_{\lambda_2} )\|_{U^2_{\pm_0}}\lesa{}& \mu^{\frac{n}{2}}\Big(\frac{\lambda_{\max}}{\mu}\Big)^{\frac12}\|u_{\lambda_1}\|_{U^2_+ + U^2_-}\|v_{\lambda_2}\|_{U^2_+ + U^2_-}\label{eq:pu}.
\end{align}
In addition, if $\mathcal{M}(f,g)(x)$ is defined as in \eqref{eqn:spatial bilinear Fourier multiplier} with
	$$ m(\xi, \eta) = (|\xi+\eta|-|\xi|-|\eta|)\ind_{\{|\xi|\approx \lambda_1,|\eta|\approx \lambda_2,|\xi+\eta|\approx \lambda_0\}}(\xi,\eta)$$
we claim
\begin{equation}
\|\mathcal{M}(u_{\lambda_1}, v_{\lambda_2})\|_{U^2_{\pm_0}}\lesa{} \mu^{\frac{n}{2}}\Big(\frac{\lambda_{\max}}{\mu}\Big)^{\frac12} \min\{\lambda_1,\lambda_2\}\|u_{\lambda_1}\|_{U^2_+ + U^2_-}\|v_{\lambda_2}\|_{U^2_+ + U^2_-}.
\label{eq:mu}
\end{equation}
Assuming these claims for the moment, we now give the proof of the bounds \eqref{eq:high-mod1} and \eqref{eq:high-mod2}. Concerning \eqref{eq:high-mod1}, we observe that in the case $\lambda_1\sim \lambda_2$, \eqref{eq:high-mod1} boils down to \eqref{eq:pv}. On the other hand, if $\lambda_1\gg \lambda_2$, we can directly apply Theorem \ref{thm:high-low U2 wave} with symbol  $m(\xi,\eta)=\ind_{\{|\xi+\eta|\approx \lambda_0, |\xi|\approx \lambda_1, |\eta|\approx \lambda_2\}}$ and obtain
\[
\|P_{\lambda_0}(u_{\lambda_1} v_{\lambda_2})\|_{S_w}\lesa{}\|P_{\lambda_0}(u_{\lambda_1} v_{\lambda_2} )\|_{V^2_{\pm_1}} \lesa \mu^{\frac{n}{2}} \|u_{\lambda_1}\|_{V^2_{\pm_1}}\|v_{\lambda_2}\|_{V^2_{\pm_2}},
\]
since $\big|\pm_1|\xi|\pm_2|\eta|-\pm_1|\xi+\eta|\big|\lesa \mu$ for $(\xi,\eta)\in \supp m$ and $\|m\|_{L^\infty_\xi L^2_\eta+L^\infty_\eta L^2_\xi}\lesa \mu^{\frac{n}{2}}$.

For the proof of \eqref{eq:high-mod2},  it is enough to show
\[\|(1+i|\nabla|^{-1}\partial_t) P_{\lambda_0}(u_{\lambda_1} v_{\lambda_2})\|_{U^2_+}\lesa{} \mu^{\frac{n}{2}} \Big(\frac{\min\{\lambda_1,\lambda_2\}}{\mu}\Big)^{\frac32} \|u_{\lambda_1}\|_{S}\|v_{\lambda_2}\|_{S}.\]
We decompose the left hand side into
\begin{align*}
 (1+i |\nabla|^{-1}\partial_t)(u_{\lambda_1} v_{\lambda_2})={}& \big( (1+i |\nabla|^{-1}\partial_t)  u_{\lambda_1}\big) v_{\lambda_2}-i |\nabla|^{-1}\mathcal{M}( |\nabla|^{-1}\partial_tu_{\lambda_1},v_{\lambda_2} )\\
&{}+i|\nabla|^{-1}( u_{\lambda_1} \partial_t v_{\lambda_2})-i|\nabla|^{-1}( |\nabla|^{-1} \partial_t u_{\lambda_1} |\nabla| v_{\lambda_2}).
\end{align*}
For the first term we directly apply Theorem \ref{thm:high-low U2 wave} as above and, since $\lambda_1\geq \lambda_2$, obtain
\begin{align*}
\| \big( (1+i |\nabla|^{-1}\partial_t)  u_{\lambda_1}\big) v_{\lambda_2}\|_{U^2_+}\lesa{}& \mu^{\frac{n}{2}}\|(1+i |\nabla|^{-1}\partial_t)  u_{\lambda_1}\|_{U^2_+} \|v_{\lambda_2}\|_{U^2_-+U^2_+}\\
\lesa{}& \mu^{\frac{n}{2}}\| u_{\lambda_1}\|_{S} \|v_{\lambda_2}\|_{S}.
\end{align*}
For the second term we apply \eqref{eq:mu}  and obtain
\begin{align*}
\| |\nabla|^{-1}\mathcal{M}( |\nabla|^{-1}\partial_tu_{\lambda_1},v_{\lambda_2})\|_{U^2_+}\lesa{}& \mu^{\frac{n}{2}}\Big(\frac{\lambda_{\max}}{\mu}\Big)^{\frac12} \frac{\lambda_2}{\lambda_0} \| |\nabla|^{-1}\partial_tu_{\lambda_1}\|_{U^2_+ + U^2_-} \|v_{\lambda_2}\|_{U^2_+ + U^2_-}\\
\lesa{}&\mu^{\frac{n}{2}}\Big(\frac{\lambda_2}{\mu}\Big)^{\frac32}\| u_{\lambda_1}\|_{S} \|v_{\lambda_2}\|_{S}.
\end{align*}
For the terms in the second line, we apply \eqref{eq:pu} and obtain
\begin{align*}
\||\nabla|^{-1}( u_{\lambda_1} \partial_t v_{\lambda_2})\|_{U^2_+}\lesa&{}\frac{\lambda_2}{\lambda_0} \mu^{\frac{n}{2}}\Big(\frac{\lambda_{\max}}{\mu}\Big)^{\frac12}\|u_{\lambda_1}\|_{U^2_+ + U^2_-}\||\nabla|^{-1}\partial_t v_{\lambda_2}\|_{U^2_+ + U^2_-}
\end{align*}
and similarly
\begin{align*}
\||\nabla|^{-1}( |\nabla|^{-1}  \partial_t u_{\lambda_1} |\nabla| v_{\lambda_2})\|_{U^2_+}\lesa{}&\frac{\lambda_2}{\lambda_0} \mu^{\frac{n}{2}}\Big(\frac{\lambda_{\max}}{\mu}\Big)^{\frac12}\||\nabla|^{-1}\partial_t u_{\lambda_1}\|_{U^2_+ + U^2_-}\| v_{\lambda_2}\|_{U^2_+ + U^2_-},
\end{align*}
such that
\[
\||\nabla|^{-1}( u_{\lambda_1} \partial_t v_{\lambda_2})\|_{U^2_+}+\||\nabla|^{-1}( |\nabla|^{-1}  \partial_t u_{\lambda_1} |\nabla| v_{\lambda_2})\|_{U^2_+}\lesa \mu^{\frac{n}{2}}\Big(\frac{\lambda_2}{\mu}\Big)^{\frac32}\| u_{\lambda_1}\|_{S} \|v_{\lambda_2}\|_{S},
\]
so the proof of \eqref{eq:high-mod2} is complete.

It remains to prove \eqref{eq:pv}, \eqref{eq:pu} and \eqref{eq:mu}. We start by observing that the standard $U^2$ Besov embedding (see Theorem \ref{thm:besov embedding} below) gives
\[
\|C_{\lesa \lambda_{\max}}P_{\lambda_0}(u_{\lambda_1}v_{\lambda_2})\|_{U^2_{\pm_0}}\lesa \lambda_{\max}^{\frac12}\|P_{\lambda_0}(u_{\lambda_1}v_{\lambda_2})\|_{L^2_{t,x}}.
\]
Now, if $u_{\lambda_1}$ is away from the cone, we have by \eqref{itm:lem S prop:Xsb} in Lemma \ref{lem:S properties}
\[\|P_{\lambda_0}(u_{\lambda_1}v_{\lambda_2})\|_{L^2_{t,x}}\lesa \mu^{\frac{n}{2}} \|u_{\lambda_1}\|_{L^2_{t,x}}\|v_{\lambda_2}\|_{L^\infty_t L^2_{x}}\lesa \mu^{\frac{n-1}{2}} \|u_{\lambda_1}\|_{S_w}\|v_{\lambda_2}\|_{S_w}, \]
and similarly, if $v_{\lambda_2}$ is away from the cone, we have
\[\|P_{\lambda_0}(u_{\lambda_1}v_{\lambda_2})\|_{L^2_{t,x}}\lesa \mu^{\frac{n}{2}} \|u_{\lambda_1}\|_{L^\infty_t L^2_{x}}\|v_{\lambda_1}\|_{L^2_{t,x}} \lesa \mu^{\frac{n-1}{2}} \|u_{\lambda_1}\|_{S_w}\|v_{\lambda_2}\|_{S_w}. \]
In summary, we have
\begin{equation}\label{eq:close-cone-p}
\|C_{\lesa \lambda_{\max}}P_{\lambda_0}(u_{\lambda_1}v_{\lambda_2})\|_{U^2_{\pm_0}}\lesa \lambda_{\max}^{\frac12}\mu^{\frac{n-1}{2}} \|u_{\lambda_1}\|_{S_w}\|v_{\lambda_2}\|_{S_w}.
\end{equation}
On the other hand, we note that by Theorem  \ref{thm:high-low U2 wave} and the uniform disposability from \eqref{itm:lem S prop:dis} in Lemma \ref{lem:S properties}
\begin{align*}
&\|C_{\gg \lambda_{\max}}P_{\lambda_0}(u_{\lambda_1}v_{\lambda_2})\|_{V^2_{\pm_0}}\\
 \lesa{}& \|P_{\lambda_0}(C_{\gg \lambda_{\max}} u_{\lambda_1}v_{\lambda_2})\|_{V^2_{\pm_0}} + \|P_{\lambda_0}(C_{\lesa \lambda_{\max}} u_{\lambda_1} C_{\gg \lambda_{\max}} v_{\lambda_2})\|_{V^2_{\pm_0}}\\
\lesa{}&  \mu^{\frac{n}{2}} \|u_{\lambda_1}\|_{S_w}\|v_{\lambda_2}\|_{S_w}.
\end{align*}
This, together with the estimate \eqref{eq:close-cone-p} and the embedding $U^2_{\pm_0}\subset V^2_{\pm_0}$ finishes the proof of \eqref{eq:pv}.
The proof of \eqref{eq:pu} follows from the same argument, by using the $U^2$-estimate in Theorem \ref{thm:high-low U2 wave} and the embedding $U^2_++U^2_- \subset S_w$ instead.
This argument also proves \eqref{eq:mu}, because Theorem  \ref{thm:high-low U2 wave} allows for multipliers, and due to the obvious fixed-time muliplier bound in $L^2_x$ in \eqref{eq:close-cone-p}.
\end{proof}

\subsection{Bilinear  $L^2_{t,x}$ Estimates}\label{subsec:bil}
In this section we give the second key bilinear bound that is required for the proof of Theorem \ref{thm:div-prob}. Similar to the previous section, the key bilinear input is a special case of a \emph{general} bilinear restriction estimate in $L^2_{t,x}$, which holds not just for the wave equation, but also the case of general phases. On the other hand, in contrast to the previous section, the bilinear estimate we prove here is much more delicate, as it handles the case where both functions, as well as their product, is close to the light cone.

We start with some motivation. Let $\lambda_1\g 1$ and define
    $$ \Lambda_1= \{ |\xi - e_1| < \tfrac{1}{100} \}, \qquad \qquad \Lambda_2= \{ |\xi \mp \lambda e_2| < \tfrac{1}{100} \lambda \} $$
with $e_1 = (1, 0, \dots, 0)$ and $e_2 = (0, 1, 0, \dots, 0)$. For free solutions, an application of Plancheral followed by H\"older implies that if $\supp \widehat{f} \subset \Lambda_1$ and $\supp \widehat{g} \subset \Lambda_2$, then we have the bilinear estimate
    \begin{equation}\label{eqn:bilinear L2 free solns} \| e^{ - i t|\nabla|} f e^{ \mp i t |\nabla|} g \|_{L^2_{t,x}} \lesa  \| f\|_{L^2_x} \| g \|_{L^2_x}.
    \end{equation}
The atomic definition of $U^2$ then easily implies that
    $$ \|   u v \|_{L^2_{t,x}} \lesa \| u \|_{U^2_+} \|v \|_{U^2_\pm}$$
for $\supp \widehat{u} \subset \Lambda_1$ and $\supp \widehat{v} \subset \Lambda_2$. Although this estimate is potentially useful, the fact that both functions must be placed into $U^2$, means that it is far to weak to be able to deduce the estimates required to solve the division problem in Theorem \ref{thm:div-prob}. In fact this lack of a good $V^2$ replacement for \eqref{eqn:bilinear L2 free solns}, was a key motivation in developing the \emph{null frame spaces} of Tataru, which were constructed to solve this issue (among others).

Recently however, in work of the first author \cite{Candy2017b}, it was shown that it is possible to deduce a $U^2 \times V^2$ version of \eref{eqn:bilinear L2 free solns} provided that the low frequency term is placed in $U^2$. It turns out that the bilinear estimate for free solutions given in \eref{eqn:bilinear L2 free solns} is insufficient for this purpose, essentially since it does not exploit any dispersive properties of free waves. Instead, the argument given in \cite{Candy2017b}, shows that the $U^2 \times V^2$ estimate can be reduced to a deeper property of free waves, namely the bilinear estimates satisfied by \emph{wave tables}. The wave table construction efficiently exploits both transversality and curvature, and was introduced by Tao \cite{Tao2001b} in the proof of the endpoint bilinear restriction estimate for the cone. In the case of the wave equation, the conclusion is the following.

\begin{theorem}[{\cite[Theorem 1.7 and Theorem 1.10]{Candy2017b}}]\label{thm:bilinear restriction}
Let $2 \les a \les b < n+1$. Let $0< \lambda_1 \les \lambda_2$,  $0<\alpha \les 1$, and $\kappa, \kappa' \in \mc{C}_{\alpha}$ with $\angle(\pm \kappa, \kappa') \approx \alpha$. If $u \in U^a_{+}$ and $v \in U^b_{\pm}$ then
    $$ \| R_{\kappa} u_{\lambda_1}  R_{\kappa'} v_{\lambda_2} \|_{L^2_{t,x}}  \lesa \alpha^{\frac{n-3}{2}} \lambda_1^{ \frac{n-1}{2}} \Big( \frac{\lambda_2}{\lambda_1} \Big)^{(n+1)(\frac{1}{2} - \frac{1}{a})} \| R_{\kappa} u_{\lambda_1} \|_{U^a_+} \| R_{\kappa'}v_{\lambda_2} \|_{U^b_\pm}.$$
If $\alpha \approx 1$ and $\lambda_1 \approx \lambda_2$, for cubes $q, q' \in \mc{Q}_\mu$ with $0<\mu \lesa \lambda_1$ we have stronger bound
    $$ \| R_{\kappa} P_q u_{\lambda_1}  P_{q'} R_{\kappa'} v_{\lambda_2} \|_{L^2_{t,x}}  \lesa \mu^{\frac{n-1}{2}}  \| R_{\kappa} P_q  u_{\lambda_1} \|_{U^a_+} \| R_{\kappa'}P_{q'} v_{\lambda_2} \|_{U^b_\pm}.$$
\end{theorem}
\begin{proof}
In the special case $\alpha \approx 1$, the bounds are a consequence of the wave table construction introduced by Tao in \cite{Tao2001b}, together with an induction on scales argument. In Section \ref{sec:adapted bilinear restriction}, we give the details of this argument by following \cite{Candy2017b} in the case of the cone. In the case $0<\alpha \lesa 1$, the proof requires a slightly more general wave table decomposition, see the proof of \cite[Theorem 1.7 and Theorem 1.10]{Candy2017b} and Remark \ref{rmk:general-alpha}.
\end{proof}

\begin{remark}\label{rem:general bilinear restriction}
The argument developed in \cite{Candy2017b} in fact shows that the bound in Theorem \ref{thm:bilinear restriction} can be generalised to the full bilinear range, and moreover holds for general phases under suitable curvature and transversality assumptions. See Section \ref{sec:adapted bilinear restriction} for further discussion.
\end{remark}

After using the standard embedding $V^2 \subset U^b$, we see that
	$$ \| R_{\kappa} u_\mu  R_{\kappa'} v_\lambda \|_{L^2_{t,x}}  \lesa \alpha^{\frac{n-3}{2}} \mu^{\frac{n-1}{2}}  \| R_{\kappa} u_{\mu} \|_{U^2_\pm} \| R_{\kappa'}v_{\lambda} \|_{V^2_+}	. $$
In particular, we have a bilinear $L^2_{t,x}$ estimate with the high frequency term in $V^2$, \emph{without} any high frequency loss.

If we combine Theorem \ref{thm:bilinear restriction} with an analysis of the resonant set, we obtain the following key bilinear estimate.

\begin{theorem}[Main Bilinear $L^2_{t,x}$ bound]\label{thm:bilinear L2 bound}
Let $d, \lambda_0, \lambda_1, \lambda_2 \in 2^\ZZ$ and $\epsilon>0$. If $\lambda_1 \les \lambda_2$, $\mu = \min\{\lambda_0, \lambda_1\}$, and $d \lesssim \mu$ we have the bilinear estimates
	\begin{align}
	 \big\| C_d P_{\lambda_0} \big( C_{\ll d } u_{\lambda_1} C_{\ll d} v_{\lambda_2} \big) \big\|_{L^2_{t,x}}
	  		&\lesa d^{\frac{n-3}{4}} \mu^{\frac{n-1}{4}} \lambda_1^\frac{1}{2} \| u_{\lambda_1} \|_{S} \| v_{\lambda_2} \|_{S} \label{eqn:thm bilinear L2:U2U2}\\
	 \big\| C_d P_{\lambda_0} \big( C_{\ll d } u_{\lambda_1} C_{\ll d} v_{\lambda_2} \big) \big\|_{L^2_{t,x}}
	  		&\lesa d^{\frac{n-3}{4}} \mu^{\frac{n-1}{4}} \lambda_1^\frac{1}{2} \Big( \frac{\lambda_1^2}{ d \mu}\Big)^\epsilon \| u_{\lambda_1} \|_{S} \| v_{\lambda_2} \|_{S_w} \label{eqn:thm bilinear L2:U2V2}\\
	\big\| C_d P_{\lambda_0} \big( C_{\ll d } u_{\lambda_1} C_{\ll d} v_{\lambda_2} \big) \big\|_{L^2_{t,x}} 	
			&\lesa  d^{\frac{n-3}{4}} \mu^{\frac{n-1}{4}} \lambda_1^\frac{1}{2} \Big( \frac{\lambda_1 \lambda_2}{ d \mu}\Big)^\epsilon \| u_{\lambda_1} \|_{S_w} \| v_{\lambda_2} \|_{S_w}. \label{eqn:thm bilinear L2:V2V2}
	\end{align}
On the other hand, if $d \gg \mu$, we have
    \begin{equation}\label{eqn:thm bilinear L2:V2V2-high}
\big\| C_d P_{\lambda_0} \big( C_{\ll d } u_{\lambda_1} C_{\ll d} v_{\lambda_2} \big) \big\|_{L^2_{t,x}}\lesa \mu^{\frac{n-1}{2}} \Big(\frac{\lambda_1}{\mu}  \Big)^\epsilon \| u_{\lambda_1} \|_{S_w} \| v_{\lambda_2} \|_{S_w}.
\end{equation}
\end{theorem}
 The first bound in the previous theorem, \eqref{eqn:thm bilinear L2:U2U2}, is a direct consequence of the corresponding estimate for free solutions. In particular, it is sharp. On the other hand, due to the fact that we now only have $V^2_\pm$ control over $v_{\lambda_2}$, the bounds \eqref{eqn:thm bilinear L2:U2V2}, \eqref{eqn:thm bilinear L2:V2V2} and \eqref{eqn:thm bilinear L2:V2V2-high} require the bilinear restriction estimates contained in Theorem \ref{thm:bilinear restriction}. The key point is that  we have no high-frequency $\lambda_2$ loss, provided that we place the low frequency term $u_{\lambda_1}$ into $U^2_{\pm}$ (i.e. the $S$ norm). The only loss \eqref{eqn:thm bilinear L2:U2V2} appears when $\lambda_1 \approx \lambda_2 \gg \lambda_0$, which is in general an easier case to deal with. On the other hand, placing both $u_{\lambda_1}$ and $v_{\lambda_2}$ into $V^2$ causes an $\epsilon$ loss in the high frequency $\lambda_2$, and thus is only useful in certain special cases.

To reduce Theorem \ref{thm:bilinear L2 bound} to the bilinear restriction estimates in Theorem \ref{thm:bilinear restriction}, we need to show that the waves $u$ and $v$ are transverse. This is a consequence of the following.

\begin{lemma}[Resonance bound for full cone]\label{lem:resonance bound full cone}
Let $d, \lambda_0, \lambda_1, \lambda_2 \in 2^\ZZ$. Assume that $(\tau, \xi), (\tau', \eta) \in \RR^{1+n}$ satisfy
$|\xi| \approx \lambda_1$, $|\eta| \approx \lambda_2$, $|\xi+\eta| \approx \lambda_0$
and
$$
\big| |\tau| - |\xi| \big| \ll d, \qquad \big| |\tau'| - |\eta|\big| \ll d, \qquad
 \big| |\tau+\tau'| - |\xi+\eta| \big| \approx d. $$
If $d \lesa \min\{\lambda_0, \lambda_1, \lambda_2\}$, then
        \begin{align*}
\angle\big( \sgn(\tau) \xi, \sgn(\tau') \eta\big) \approx \Big( \frac{d \lambda_0}{\lambda_1 \lambda_2}\Big)^\frac{1}{2},\\
\angle\big( \sgn(\tau+\tau') (\xi+\eta), \sgn(\tau) \xi\big) \lesa \Big( \frac{d \lambda_2}{\lambda_0 \lambda_1} \Big)^\frac{1}{2},\\
\angle\big( \sgn(\tau+\tau') (\xi+\eta), \sgn(\tau') \eta \big) \lesa \Big( \frac{d \lambda_1}{\lambda_0 \lambda_2} \Big)^\frac{1}{2}.
        \end{align*}
        On the other hand, if $d \gg \min\{\lambda_0, \lambda_1, \lambda_2\}$ then in fact $\sgn(\tau) = \sgn(\tau')$ and
       $$ \angle\big(  \xi, \eta \big) \approx 1, \qquad d \approx \max\{\lambda_0, \lambda_1, \lambda_2\}, \qquad \lambda_0 \ll \lambda_1 \approx \lambda_2. $$
\begin{proof}
We first observe that since
    $$ \big| |\tau + \tau'| - | \sgn(\tau)|\xi| + \sgn(\tau') |\eta| | \big| \les \big| |\tau| - |\xi| \big| + \big| |\tau'|  -|\eta| \big| \ll d$$
and
    $$ \big| ( \sgn(\tau) |\xi| + \sgn(\tau')|\eta|)^2 - |\xi + \eta|^2 \big| \approx \lambda_1 \lambda_2 \angle\big( \sgn(\tau) \xi, \sgn(\tau')\xi'\big)$$
we have
    \begin{equation}\label{eqn:proof of lem resonance bound:main ident}
      d \approx \big| \big| \sgn(\tau) |\xi| + \sgn(\tau')|\eta|\big| - |\xi + \eta| \big|
        \approx \frac{\lambda_1 \lambda_2 \angle^2\big( \sgn(\tau) \xi, \sgn(\tau')\eta\big)}{\big| | \sgn(\tau) |\xi| + \sgn(\tau')|\eta|| + |\xi + \eta| \big|}.
    \end{equation}
If $\lambda_0 \approx \max\{\lambda_1, \lambda_2\}$, then \eref{eqn:proof of lem resonance bound:main ident} already gives $ d\lesa \min\{ \lambda_0, \lambda_1, \lambda_2\}$ and the claimed orthogonality bound, so it remains to consider the case $\lambda_0 \ll \lambda_1 \approx \lambda_2$. We first consider the interactions where $\sgn(\tau) = -\sgn(\tau')$ which implies that
    $$ \big| | \sgn(\tau) |\xi| + \sgn(\tau')|\eta|| + |\xi + \eta| \big| \approx \lambda_0, \qquad  \lambda_1 \lambda_2 \angle^2(\xi, -\eta) \approx |\xi + \eta|^2 - \big| |\xi| - |\eta| \big|^2 \lesa \lambda_0^2.$$
Consequently the claimed bounds again follow immediately from \eref{eqn:proof of lem resonance bound:main ident}. On the other hand, if $\sgn(\tau) = \sgn(\tau')$ then as $\angle(\xi, - \eta) \lesa \frac{\lambda_0}{\lambda_1} \ll 1$ we must have $\angle(\xi, \eta) \approx 1$ and hence \eref{eqn:proof of lem resonance bound:main ident} implies that
    $$ d \approx \frac{\lambda_1 \lambda_2}{\lambda_1} \angle^2(\xi, \eta) \approx \lambda_1$$
as required. It only remains to control the angle between $\xi+\xi'$ and $\xi$. To this end, an analogous computation to that used to deduce \eref{eqn:proof of lem resonance bound:main ident} gives
    $$ \lambda_0 \lambda_1 \angle^2\big( \sgn(\tau + \tau') (\xi+\eta), \xi\big) \lesa d \big( \big|  \sgn(\tau + \tau') |\xi + \eta| - \sgn(\tau) |\xi|\big| + |\eta| \big) $$
which suffices unless we have $\lambda_0\approx \lambda_1 \gg \lambda_2$ and $\sgn(\tau + \tau') = - \sgn(\tau)$. But this implies that
    $$ d \approx \big| |\tau + \tau'| - |\xi+\eta|\big| \approx \big| \sgn(\tau) |\xi| +\sgn(\tau') |\eta| - \sgn(\tau + \tau')|\xi + \eta| \big| \approx \lambda_0 $$
which contradicts the previous computation which showed that we must have $d \lesa \lambda_2$. Hence the case $\lambda_0\approx \lambda_1 \gg \lambda_2$ and $\sgn(\tau + \tau') = - \sgn(\tau)$ cannot occur, and consequently we deduce the correct angle bound between $\xi+\eta$ and $\eta$.
\end{proof}
\end{lemma}

\begin{lemma}[Lower bound on resonance]\label{lem:lb-resonance}
Let $d, \lambda_0, \lambda_1, \lambda_2 \in 2^\ZZ$. Assume that $(\tau, \xi), (\tau', \eta) \in \RR^{1+n}$ satisfy
    $$ |\xi| \approx \lambda_1, \qquad |\eta| \approx \lambda_2, \qquad |\xi+\eta| \approx \lambda_0$$
and
    $$ \big| |\tau| - |\xi| \big| \lesa d, \qquad \big| |\tau'| - |\eta|\big| \lesa d, \qquad \big| |\tau+\tau'| - |\xi+\eta| \big| \lesa d. $$
Then,
\begin{align*}
&\angle\big( \sgn(\tau) \xi, \sgn(\tau') \eta\big)+\angle\big( \sgn(\tau) \xi, \sgn(\tau+\tau') (\xi+\eta)\big)\\
{}&\qquad  +\angle\big( \sgn(\tau') \eta, \sgn(\tau+\tau') (\xi+\eta)\big)\lesa{} \Big( \frac{d}{\min\{\lambda_0, \lambda_1, \lambda_2\}}\Big)^{\frac{1}{2}}
\end{align*}
\end{lemma}
\begin{proof}
We start with the observation
\[
\big|  \sgn(\tau) |\xi| + \sgn(\tau')|\eta| - \sgn(\tau+\tau')|\xi + \eta| \big|\lesa d.
\]
The left hand side is bounded below by
\[
\big| | \sgn(\tau) |\xi| + \sgn(\tau')|\eta|| - |\xi + \eta| \big|\approx \frac{|\xi||\eta| \angle^2\big( \sgn(\tau) \xi, \sgn(\tau') \eta\big)}{| \sgn(\tau) |\xi| + \sgn(\tau')|\eta|| + |\xi + \eta|},
\]
which implies the bound on the first summand. The bounds on the other two summands follow similarly.
\end{proof}

We now give the proof of Theorem \ref{thm:bilinear L2 bound}.

\begin{proof}[Proof of Theorem \ref{thm:bilinear L2 bound}]
Let $\lambda_0, \lambda_1, \lambda_2 \in 2^\ZZ$ with $\lambda_1 \les \lambda_2$, and take $\mu = \min\{ \lambda_0, \lambda_1, \lambda_2\}$. It is enough to consider the case $u_{\lambda_1} \in U^2_+$ (or $V^2_+$), and $v_{\lambda_2}\in U^2_{\pm}$ (or $V^2_{\pm}$). Suppose that $d \lesa \mu$ and note that we can write
    $$C_{\ll d} u_{\lambda_1} = C_{\ll d}^+ u_{\lambda_1} + C_{\ll d}^- u_{\lambda_1} = C_{\ll d}^+ u_{\lambda_1} + C_{\ll d}^- C^+_{\approx \lambda_1} u_{\lambda_1}.$$
Applying H\"{o}lder's inequality we deduce that
     \begin{align}
        \big\| C_d P_{\lambda_0} \big( C_{\ll d }^- u_{\lambda_1} C_{\ll d} v_{\lambda_2} \big) \big\|_{L^2_{t,x}} &\lesa \mu^{\frac{n}{2}} \| C^+_{\approx \lambda_1} u_{\lambda_1} \|_{L^2_{t,x}} \| v_{\lambda_2} \|_{L^\infty_t L^2_x} \notag \\
            &\lesa \mu^{\frac{n}{2}} \lambda_1^{-\frac{1}{2}} \| u_{\lambda_1} \|_{U^4_+} \| v_{\lambda_2} \|_{U^4_\pm} \label{eqn:proof thm bilinear L2:far cone I}
     \end{align}
which suffices if $n=2, 3$ (clearly we can choose an exponent larger than $4$ if necessary). To obtain a slightly sharper bound, we can decompose into caps/cubes before applying H\"{o}lder, namely, letting $\alpha = (\frac{d \lambda_0}{\lambda_1 \lambda_2})^\frac{1}{2}$ and applying Lemma \ref{lem:resonance bound full cone}, we have
    \begin{align*}
        &\big\| C_d P_{\lambda_0} \big( C_{\ll d }^- u_{\lambda_1} C_{\ll d} v_{\lambda_2} \big) \big\|_{L^2_{t,x}} \\
\lesa{}& \Big( \sum_{\substack{\kappa, \kappa' \in \mc{C}_\alpha }}  \sum_{q, q'\in \mc{Q}_\mu} \big\| C_d P_{\lambda_0}\big( R_{\kappa} P_q C^-_{\ll d}u_{\lambda_1} R_{\kappa'} P_{q'} C_{\ll d} v_{\lambda_2} \big) \big\|_{L^2_{t,x}}^2 \Big)^\frac{1}{2} \\
       \lesa{}& \mu^{\frac{1}{2}} (\mu d)^{\frac{n-1}{4}} \Big( \sum_{\substack{\kappa, \kappa' \in \mc{C}_\alpha }}  \sum_{q, q'\in \mc{Q}_\mu} \| R_{\kappa} P_q C^+_{\approx \lambda_1} u_{\lambda_1} \|_{L^2_{t,x}}^2 \| R_{\kappa'} P_{q'} v_{\lambda_2} \|_{L^\infty_t L^2_x}^2 \Big)^\frac{1}{2} \\
            \lesa{}& \mu^{\frac{1}{2}} (\mu d)^{\frac{n-1}{4}} \lambda_1^{-\frac{1}{2}} \| u_{\lambda_1} \|_{U^2_+} \| v_{\lambda_2} \|_{U^2_\pm}
    \end{align*}
where we used the fact that since $\lambda_1 \les \lambda_2$ we have $\lambda_1 \alpha = ( d \mu)^{\frac{1}{2}}$. Interpolating with \eref{eqn:proof thm bilinear L2:far cone I}, we obtain an estimate which clearly suffices in higher dimensions as well. After noting the identity
        $$ C_{\ll d} v_{\lambda_2} = C_{\ll d}^\pm v_{\lambda_2} + C_{\approx \lambda_2}^{\pm} C_{\ll d}^{\mp} v_{\lambda_2} $$
a similar argument to the above reduces the problem to proving the bounds
        \begin{equation}\label{eqn:proof of thm bilinear L2 bound:initial reduction}
            \begin{split}
        \big\| C_d P_{\lambda_0} \big( C^+_{\ll d } u_{\lambda_1} C^\pm_{\ll d} v_{\lambda_2} \big) \big\|_{L^2_{t,x}} &\lesa d^{\frac{n-3}{4}} \mu^{\frac{n-1}{4}} \lambda_1^\frac{1}{2} \| u_{\lambda_1} \|_{U^2_+} \| v_{\lambda_2} \|_{U^2_{\pm}},\\
        \big\| C_d P_{\lambda_0} \big( C^+_{\ll d } u_{\lambda_1} C^\pm_{\ll d} v_{\lambda_2} \big) \big\|_{L^2_{t,x}} &\lesa d^{\frac{n-3}{4}} \mu^{\frac{n-1}{4}} \lambda_1^\frac{1}{2} \Big( \frac{\lambda_1^2}{\mu d }\Big)^\epsilon \| u_{\lambda_1} \|_{U^2_+} \| v_{\lambda_2} \|_{V^2_{\pm}}\\
        \big\| C_d P_{\lambda_0} \big( C^+_{\ll d } u_{\lambda_1} C^\pm_{\ll d} v_{\lambda_2} \big) \big\|_{L^2_{t,x}} &\lesa d^{\frac{n-3}{4}} \mu^{\frac{n-1}{4}} \lambda_1^\frac{1}{2} \Big( \frac{\lambda_1 \lambda_2}{\mu d }\Big)^\epsilon \| u_{\lambda_1} \|_{V^2_+} \| v_{\lambda_2} \|_{V^2_{\pm}}
            \end{split}
      \end{equation}
We now exploit Lemma \ref{lem:resonance bound full cone} and orthogonality. Let $\alpha = ( \frac{d \lambda_0}{\lambda_1 \lambda_2} )^\frac{1}{2}$, and $\beta = (\frac{d \lambda_1}{\lambda_0 \lambda_2})^\frac{1}{2}$. Note that since we assume $\lambda_1 \les \lambda_2$, we have $\lambda_1 \alpha = (d \mu)^\frac{1}{2}$ and $\lambda_0 \beta = (d \lambda_0)^\frac{1}{2}$. An application of Lemma \ref{lem:resonance bound full cone} and  orthogonality implies that after decomposing into caps, we have
   \begin{equation}\label{eqn:proof of bilinear L2 thm:L2 bdd by sqr sum}
        \begin{split}& \big\| C_d P_{\lambda_0} \big( C^+_{\ll d } u_{\lambda_1} C^\pm_{\ll d} v_{\lambda_2} \big) \big\|_{L^2_{t,x}}^2 \\
            \lesa{}& \sum_{\kappa'' \in \mc{C}_\beta} \sum_{q, q' \in \mc{Q}_\mu} \Big( \sum_{\substack{\kappa, \kappa'\in \mc{C}_\alpha \\ |\kappa \mp \kappa'| \approx \alpha}} \big\| C_d P_{\lambda_0} \big( C^+_{\ll d }R_{\kappa} P_q u_{\lambda_1} C^\pm_{\ll d} R_{\kappa'} R_{\kappa''} P_{q'} v_{\lambda_2} \big) \big\|_{L^2_{t,x}}\Big)^2.
        \end{split}
   \end{equation}
The standard bilinear $L^2_{t,x}$ bound for free solutions, together with the disposability of the $C_{\ll d}$ multipliers,  gives for any cubes $q, q' \in \mc{Q}_\mu$ and caps $\kappa, \kappa' \in \mc{C}_\alpha$ with $|\kappa \mp \kappa'| \approx \alpha$ the estimate
   \begin{align*}
        &\big\| C_d P_{\lambda_0} \big( C^+_{\ll d }R_{\kappa} P_q u_{\lambda_1} C^\pm_{\ll d} R_{\kappa'} R_{\kappa''} P_{q'} v_{\lambda_2} \big) \big\|_{L^2_{t,x}}\\
            \lesa{}& d^{\frac{n-3}{4}} \mu^{\frac{n-1}{4}} \lambda_1^\frac{1}{2}  \| R_{\kappa} P_q u_{\lambda_1} \|_{U^2_+} \| R_{\kappa'} R_{\kappa''} P_{q'} v_{\lambda_2} \|_{U^2_{\pm}}.
   \end{align*}
   Therefore, from \eref{eqn:proof of bilinear L2 thm:L2 bdd by sqr sum} and the $U^2$ square sum bound, we deduce that
	\begin{equation}\label{eqn:proof thm bilinear L2:bdd by U2}
		\big\| C_d P_{\lambda_0} \big( C^+_{\ll d } u_{\lambda_1} C^\pm_{\ll d} v_{\lambda_2} \big) \big\|_{L^2_{t,x}}
			\lesa d^{\frac{n-3}{4}} \mu^{\frac{n-1}{4}} \lambda_1^\frac{1}{2}  \| u_{\lambda_1} \|_{U^2_+} \| v_{\lambda_2} \|_{U^2_{\pm}}.
	\end{equation}
To replace $U^2$ with $V^2$, we apply Theorem \ref{thm:bilinear restriction}. We first suppose that $\lambda_1 \approx \lambda_2$. After applying Lemma \ref{lem:resonance bound full cone} and decomposing into caps of size $\alpha$, Theorem \ref{thm:bilinear restriction} implies that for any $2 \les a \les b <n+1$
    \begin{align}
        &\big\| C_d P_{\lambda_0} \big( C^+_{\ll d } u_{\lambda_1} C^\pm_{\ll d} v_{\lambda_2} \big) \big\|_{L^2_{t,x}}  \notag\\
\lesa{}& \sum_{\substack{\kappa, \kappa' \in \mc{C}_\alpha \\ |\kappa - \kappa'| \approx \alpha }} \big\|  C^+_{\ll d } R_\kappa u_{\lambda_1} C^\pm_{\ll d} R_{\kappa'} v_{\lambda_2} \big\|_{L^2_{t,x}} \notag \\
        \lesa{}& \alpha^{\frac{n-3}{2}} \lambda_1^{\frac{n-1}{2}} \sum_{\substack{\kappa, \kappa' \in \mc{C}_\alpha \\ |\kappa - \kappa'| \approx \alpha }} \big\| R_\kappa u_{\lambda_1}\|_{U^a_+} \| R_{\kappa'} v_{\lambda_2} \big\|_{U^b_{\pm}} \notag \\
        \lesa{}& \alpha^{\frac{n-3}{2} - (n-1)(1 - \frac{1}{a}-\frac{1}{b})} \lambda_1^{\frac{n-1}{2}} \Big( \sum_{\kappa  \in \mc{C}_\alpha} \| R_\kappa u_{\lambda_1}\|_{U^a_+}^a\Big)^\frac{1}{a} \Big( \sum_{\kappa' \in \mc{C}_\alpha}  \| R_{\kappa'} v_{\lambda_2} \|_{U^b_{\pm}}^b\Big)^\frac{1}{b} \notag \\
        \lesa{}& d^{\frac{n-3}{4}} \mu^{\frac{n-1}{4}} \lambda_1^\frac{1}{2} \Big(\frac{\lambda_1}{\mu}\Big)^\frac{1}{2} \Big( \frac{\lambda_1^2}{ d \mu} \Big)^{(n-1)(1-\frac{1}{a} - \frac{1}{b})}  \|u_{\lambda_1}\|_{U^a_+} \| v_{\lambda_2} \|_{U^a_\pm}
        \label{eqn:proof thm bilinear L2:high-high Ua}.
    \end{align}
Together \eref{eqn:proof thm bilinear L2:bdd by U2} and \eref{eqn:proof thm bilinear L2:high-high Ua} together with the standard $V^2$ interpolation argument give \eref{eqn:proof of lem resonance bound:main ident} in the case $\lambda_1 \approx \lambda_2$. On the other hand, if $\lambda_1 \ll \lambda_2$, we decompose into caps of size $\beta = (\frac{d}{\lambda_1})^\frac{1}{2}$ and again apply Lemma \ref{lem:resonance bound full cone} and Theorem \ref{thm:bilinear restriction} to deduce that for any $2\les a \les b < n+1$
     \begin{align*}
&        \big\| C_d P_{\lambda_0} \big( C^+_{\ll d } u_{\lambda_1} C^\pm_{\ll d} v_{\lambda_2} \big) \big\|_{L^2_{t,x}}\\
 \lesa{}& \Big( \sum_{\kappa \in \mc{C}_\beta} \sup_{\substack{\kappa' \in \mc{C}_\beta \\ |\kappa - \kappa'| \approx \beta}} \big\|  C^+_{\ll d } R_\kappa u_{\lambda_1} C^\pm_{\ll d} R_{\kappa'} v_{\lambda_2} \big\|_{L^2_{t,x}}^2 \Big)^{\frac{1}{2}}  \\
            \lesa{}& \beta^{\frac{n-3}{2}} \lambda_1^{\frac{n-1}{2}} \Big( \frac{\lambda_2}{\lambda_1}\Big)^{(n+1)(\frac{1}{2}-\frac{1}{a})} \Big( \sum_{\kappa \in \mc{C}_\beta} \| R_{\kappa} u_{\lambda_1} \|_{U^a_+}^a\Big)^\frac{1}{a} \| v_{\lambda_2} \|_{V^b_\pm}\\
            \lesa{}& \beta^{\frac{n-3}{2}- (n-1)(\frac{1}{2}-\frac{1}{a})} \lambda_1^{\frac{n-1}{2}} \Big( \frac{\lambda_2}{\lambda_1}\Big)^{(n+1)(\frac{1}{2}-\frac{1}{a})} \| u_{\lambda_1} \|_{U^a_+} \| v_{\lambda_2} \|_{V^2_\pm}.
     \end{align*}
Choosing $a=2$, we get the $U^2 \times V^2$ estimate. Taking $a$ sufficiently close to 2 gives the $V^2 \times V^2$ estimate.

It remains to consider the case $d \gg \mu$. In light of Lemma \ref{lem:resonance bound full cone}, the left hand side is only nonzero if $\lambda_0 \ll \lambda_1 \approx \lambda_2$ and $d \approx \lambda_1$, and moreover, we have the identity
	\begin{equation}\label{eq:dec-mod} \begin{split} &C_d P_{\lambda_0}( C_{\ll d} u_{\lambda_1} C_{\ll d } v_{\lambda_2} ) \\
={}& C_d P_{\lambda_0}( C^+_{\ll d} u_{\lambda_1} C^+_{\ll d } v_{\lambda_2} )  + C_d P_{\lambda_0}( C^-_{\ll d} u_{\lambda_1} C^-_{\ll d } v_{\lambda_2} ).
          \end{split}
        \end{equation}
        To estimate the $L^2_{t,x}$ norm of the first term in \eqref{eq:dec-mod}, if $\pm=+$, we decompose into caps and apply Theorem \ref{thm:bilinear restriction} which gives
	\begin{align*}
		\| C_d P_{\lambda_0}( C^+_{\ll d} u_{\lambda_1} C^+_{\ll d } v_{\lambda_2} ) \|_{L^2_{t,x}}
				&\lesa \sum_{\substack{ \kappa, \kappa' \in \mc{C}_\frac{1}{100} \\ \angle(\kappa, \kappa') \approx 1}}
		\sum_{\substack{ q,q' \in \mc{Q}_\mu \\ |q-q'|\approx \mu}} \| R_{\kappa} P_q u_{\lambda_1} R_{\kappa'} P_{q'} v_{\lambda_2} \|_{L^2_{t,x}} \\
				&\lesa \mu^{\frac{n-1}{2}} \sum_{\substack{ q,q' \in \mc{Q}_\mu \\ |q-q'|\approx \mu}} \| P_q u_{\lambda_1}\|_{U^a_+} \| P_{q'} v_{\lambda_2} \|_{U^b_+} \\
				&\lesa \mu^{\frac{n-1}{2}} \Big( \frac{\lambda_1}{\mu}\Big)^{n(1-\frac{1}{a} - \frac{1}{b})} \| u_{\lambda_1} \|_{U^a_+} \| v_{\lambda_2} \|_{U^b_+}.
	\end{align*}
If $\pm=-$, we have $C^+_{\ll d } v_{\lambda_2}=C^-_{\approx \lambda_2} C^+_{\ll d } v_{\lambda_2}$ and
\begin{align*}
\| C_d P_{\lambda_0}( C^+_{\ll d} u_{\lambda_1} C^+_{\ll d } v_{\lambda_2} ) \|_{L^2_{t,x}}
				&\lesa \mu^{\frac{n}{2}}\|u_{\lambda_1}\|_{L^\infty_t L^2_x}\|C^-_{\approx \lambda_2}  v_{\lambda_2}\|_{L^2_{t,x}}\\
&\lesa \mu^{\frac{n-1}{2}} \Big( \frac{\mu}{\lambda_2}\Big)^{\frac12}\|u_{\lambda_1}\|_{V^2_+}\| v_{\lambda_2}\|_{V^2_-}.
\end{align*}
To estimate the $L^2_{t,x}$ norm of the second term in \eqref{eq:dec-mod}, we use $C^-_{\ll d } u_{\lambda_1}=C^+_{\approx \lambda_1} C^-_{\ll d } u_{\lambda_1}$ and obtain
\begin{align*}
\| C_d P_{\lambda_0}( C^-_{\ll d} u_{\lambda_1} C^-_{\ll d } v_{\lambda_2} ) \|_{L^2_{t,x}}
				&\lesa \mu^{\frac{n}{2}}\|C^+_{\approx \lambda_1} u_{\lambda_1}\|_{L^2_{t,x}}\|  v_{\lambda_2}\|_{L^\infty_t L^2_x}\\
&\lesa \mu^{\frac{n-1}{2}} \Big( \frac{\mu}{\lambda_1}\Big)^{\frac12}\|u_{\lambda_1}\|_{V^2_+}\| v_{\lambda_2}\|_{V^2_\pm},
\end{align*}
which implies the claimed estimate as $\lambda_1\approx \lambda_2$ in this case.
\end{proof}

\subsection{Proof of Theorem \ref{thm:div-prob}}\label{subsec:non}

We now combine the high-low product estimates in Theorem \ref{thm:high-mod}, together with the bilinear $L^2_{t,x}$ estimate in Theorem \ref{thm:bilinear L2 bound}, and show that the space $S$ solves the division problem. To simplify the proof, we start by giving the following consequence of the bilinear $L^2_{t,x}$ bound, which is used to control the close cone interactions in Theorem \ref{thm:div-prob}.

\begin{lemma}\label{lem:tri-est}
Let $\epsilon>0$ and $\lambda_0, \lambda_1, \lambda_2 \in 2^\ZZ$ with $\mu = \min\{\lambda_0, \lambda_1, \lambda_2\}$. If $v_{\lambda_0}, u_{\lambda_1} \in S$ and $w_{\lambda_2} \in S_w$ then
\begin{equation}
\label{eq:tri-est1}\begin{split}
&
\Big| \int_{\RR^{1+n}}  C_{\lesa \mu} v_{\lambda_0} C_{\lesa \mu} u_{\lambda_1}\Box C_{\lesa \mu} w_{\lambda_2}dx\, dt \Big| \\
\lesa{}& \mu^{\frac{n-1}{2}}\lambda_2 (\min\{\lambda_0,\lambda_1\})^{\frac{1}{2}}  \| v_{\lambda_0} \|_{S} \| u_{\lambda_1} \|_S \| w_{\lambda_2} \|_{S_w}.
\end{split}
\end{equation}
Similarly, if $u_{\lambda_1} \in S$ and $v_{\lambda_0}, w_{\lambda_2} \in S_w$, then
\begin{equation}
\label{eq:tri-est2}
\begin{split}
&\Big| \int_{\RR^{1+n}}  C_{\lesa \mu} v_{\lambda_0}C_{\lesa \mu} u_{\lambda_1} \Box C_{\lesa \mu} w_{\lambda_2}
dx\, dt \Big| \\
\lesa{}& \mu^{\frac{n-1}{2}}\lambda_2 (\min\{\lambda_0,\lambda_1\})^{\frac{1}{2}} \Big( \frac{ \lambda_1 \min\{\lambda_0, \lambda_1\}}{ \mu^2 } \Big)^\epsilon \| v_{\lambda_0} \|_{S_w} \| u_{\lambda_1} \|_S \| w_{\lambda_2} \|_{S_w}.
\end{split}
\end{equation}
\end{lemma}

\begin{proof}
We start by proving the bounds
    \begin{align}\label{eq:claim1}
    \sum_{d \lesa \mu } \Big| \int_{\RR^{1+n}}  C_d v_{\lambda_0} C_{\lesa d}u_{\lambda_1} \Box C_{\lesa d}w_{\lambda_2} dx dt \Big|
    &\lesa  \mu^{\frac{n}{2}} \lambda_2  \| v_{\lambda_0} \|_{S_w} \| u_{\lambda_1} \|_S \| w_{\lambda_2} \|_{S_w}\\
    \label{eq:claim2}
    \sum_{d \lesa \mu } \Big| \int_{\RR^{1+n}} C_{\lesa d} v_{\lambda_0} C_d u_{\lambda_1} \Box C_{\lesa d} w_{\lambda_2}  dx dt \Big|
    &\lesa \mu^{\frac{n}{2}} \lambda_2 \| v_{\lambda_0} \|_{S_w} \| u_{\lambda_1} \|_S \| w_{\lambda_2} \|_{S_w}
    \end{align}
and, for every $\epsilon>0$,
    \begin{equation}\label{eq:claim3}
    \begin{split}
   &\sum_{d \lesa \mu } \Big| \int_{\RR^{1+n}} C_{\ll d} v_{\lambda_0} C_{\ll d}u_{\lambda_1} \Box C_d w_{\lambda_2}  dx dt \Big|\\
    \lesa{}& \mu^{\frac{n}{2}}\lambda_2 \Big(\frac{\min\{\lambda_0,\lambda_1\}}{\mu}\Big)^{\frac{1}{2}} \| v_{\lambda_0} \|_{S} \| u_{\lambda_1} \|_S \| w_{\lambda_2} \|_{S_w}, \\
    &\sum_{d \lesa \mu } \Big| \int_{\RR^{1+n}} C_{\ll d} v_{\lambda_0} C_{\ll d}u_{\lambda_1} \Box C_d w_{\lambda_2}  dx dt \Big|\\
    \lesa{}& \mu^{\frac{n}{2}}\lambda_2    \Big(\frac{\min\{\lambda_0,\lambda_1\}}{\mu}\Big)^{\frac{1}{2}+\epsilon} \Big( \frac{ \lambda_1 }{ \mu } \Big)^\epsilon \| v_{\lambda_0} \|_{S_w} \| u_{\lambda_1} \|_S \| w_{\lambda_2} \|_{S_w}.
    \end{split}
    \end{equation}
Let $\beta = (\frac{d}{\mu})^\frac{1}{2}$. The bound \eref{eq:claim1} follows by decomposing into caps of radius $\beta$, applying Lemma \ref{lem:lb-resonance} together with  Lemma \ref{lem:S properties}, and observing that
    \begin{align*}
  &    \int_{\RR^{1+n}}C_{d} v_{\lambda_0} C_{\lesa d} u_{\lambda_1} \Box C_{\lesa d} w_{\lambda_2}dx dt \\
\lesa{}&\beta^{\frac{n-1}{2}} \mu^{\frac{n}{2}} \sum_{\substack{\kappa, \kappa', \kappa'' \in \mc{C}_\beta \\ \angle_\pm (\kappa, \kappa'), \angle_\pm (\kappa ,\kappa'')\lesa \beta}} \| R_{\kappa} C_d v_{\lambda_0} \|_{ L^2_{t,x}} \| R_{\kappa'} C_{\lesa d} u_{\lambda_1} \|_{L^\infty_t L^2_{x}} \| \Box C_{\lesa d} R_{\kappa''} w_{\lambda_2} \|_{L^2_{t,x}} \\
      \lesa{}& \beta^{\frac{n-1}{2}} \mu^{\frac{n}{2}}  \|  C_d v_{\lambda_0} \|_{L^2_{t,x}} \sup_{\kappa' \in \mc{C}_\beta}  \| R_{\kappa'} C_{\lesa d}u_{\lambda_1} \|_{L^\infty_t L^2_x} \| \Box C_{\lesa d} w_{\lambda_2} \|_{L^2_{t,x}} \\
     \lesa{}& d^{\frac{n-1}{4}} \mu^{\frac{n+1}{4}} \lambda_2 \| v_{\lambda_0} \|_{S_w} \| u_{\lambda_1} \|_S \| w_{\lambda_2} \|_{S_w}.
    \end{align*}
Consequently summing up over modulation $d \lesa \mu $, this gives \eqref{eq:claim1}. The proof of \eref{eq:claim2} is similar. More precisely, again decomposing into caps of radius $\beta$  gives
    \begin{align*}
  &    \int_{\RR^{1+n}}C_{\lesa d} v_{\lambda_0} C_d u_{\lambda_1} \Box C_{\lesa d} w_{\lambda_2}dx dt \\
\lesa{}&\beta^{\frac{n-1}{2}} \mu^{\frac{n}{2}} \sum_{\substack{\kappa, \kappa', \kappa'' \in \mc{C}_\beta \\ \angle_\pm (\kappa, \kappa'), \angle_\pm (\kappa ,\kappa'')\lesa \beta}} \| R_{\kappa} v_{\lambda_0} \|_{L^\infty_t L^2_x} \| R_{\kappa'} C_d u_{\lambda_1} \|_{L^2_{t,x}} \| \Box C_{\lesa d} R_{\kappa''} w_{\lambda_2} \|_{L^2_{t,x}} \\
      \lesa{}& \beta^{\frac{n-1}{2}} \mu^{\frac{n}{2}} \sup_{\kappa \in \mc{C}_\beta}  \| R_\kappa  v_{\lambda_0} \|_{L^\infty_t L^2_x} \|  C_d u_{\lambda_1} \|_{L^2_{t,x}} \| \Box C_{\lesa d} w_{\lambda_2} \|_{L^2_{t,x}} \\
     \lesa{}& d^{\frac{n-1}{4}} \mu^{\frac{n+1}{4}} \lambda_2 \| v_{\lambda_0} \|_{S_w} \| u_{\lambda_1} \|_S \| w_{\lambda_2} \|_{S_w}.
    \end{align*}
Consequently summing up over modulation $d \lesa \mu $, we have \eqref{eq:claim2}. To prove \eref{eq:claim3}, we first observe that Theorem \ref{thm:bilinear L2 bound} implies that for every $\epsilon>0$
	\begin{align*} \big\| P_{\lambda_2} C_d \big( C_{\ll d} &v_{\lambda_0}   C_{\ll d} u_{\lambda_1} \big) \big\|_{L^2_{t,x}}\\
	& \lesa d^{\frac{n-3}{4}} \mu^{\frac{n-1}{4}} (\min\{\lambda_0, \lambda_1\})^\frac{1}{2} \Big( \frac{ \lambda_1 \min\{\lambda_0, \lambda_1\}}{ d \mu } \Big)^\epsilon \| v_{\lambda_0}\|_{S_w} \| u_{\lambda_1} \|_{S},
	\end{align*}
where we can put $\epsilon=0$ if $\| v_{\lambda_0} \|_{S_w} $ is replaced by $\| v_{\lambda_0} \|_{S}$. Hence an application of H\"older's inequality gives
	\begin{align*}
		&\int_{\RR^{1+3}} C_{\ll d} v_{\lambda_0} C_{\ll d}u_{\lambda_1} \Box C_d w_{\lambda_2}  dx dt \\
\lesa{}& \| C_{d} P_{\lambda_2} \big( C_{\ll d} v_{\lambda_0} C_{\ll d} u_{\lambda_1}\big) \|_{L^2_{t,x}} \| \Box C_d w_{\lambda_2} \|_{L^2_{t,x}} \\
		\lesa{}& d^{\frac{n-1}{4}} \mu^{\frac{n-1}{4}} \lambda_2 (\min\{\lambda_0, \lambda_1\})^\frac{1}{2} \Big( \frac{ \lambda_1 \min\{\lambda_0, \lambda_1\}}{ d \mu } \Big)^\epsilon \| v_{\lambda_0}\|_{S_w} \| u_{\lambda_1} \|_{S} \| w_{\lambda_2} \|_{S_w},
	\end{align*}
where, again, we can put $\epsilon=0$ if $\| v_{\lambda_0} \|_{S_w} $ is replaced by $\| v_{\lambda_0} \|_{S}$.
Summing up over modulation $d \lesa \mu $, this gives \eqref{eq:claim3}.

In order to finally prove \eqref{eq:tri-est1} and \eqref{eq:tri-est2}, we decompose the product into
  \begin{align*}
&  C_{\lesa \mu } v_{\lambda_0} C_{\lesa \mu }u_{\lambda_1} \Box C_{\lesa \mu }w_{\lambda_2} =
\sum_{d\lesa \mu } C_d v_{\lambda_0} C_{\leq d}u_{\lambda_1} \Box C_{\leq d}w_{\lambda_2} \\ & {}\qquad \qquad + \sum_{d\lesa \mu }C_{< d} v_{\lambda_0} C_d u_{\lambda_1} \Box C_{< d} w_{\lambda_2}+ \sum_{d\lesa \mu } C_{< d} v_{\lambda_0} C_{< d}u_{\lambda_1} \Box C_d w_{\lambda_2}.
  \end{align*}
 Estimate \eqref{eq:claim1} takes care of the first term and estimate \eqref{eq:claim2} yields the required bound for the second term. Further, we write
\begin{align*}
&\sum_{d\lesa \mu } C_{< d} v_{\lambda_0} C_{< d}u_{\lambda_1} \Box C_d w_{\lambda_2}
=\sum_{d\lesa \mu }C_{\ll d} v_{\lambda_0} C_{\ll d}u_{\lambda_1} \Box C_d w_{\lambda_2}\\
&{}\qquad \qquad+\sum_{d\lesa \mu }C_{< d}C_{\sim d} v_{\lambda_0} C_{<d}u_{\lambda_1} \Box C_d w_{\lambda_2}+\sum_{d\lesa \mu } C_{\ll d} v_{\lambda_0} C_{< d}C_{\sim d}u_{\lambda_1} \Box C_d w_{\lambda_2},
\end{align*}
and now estimate \eqref{eq:claim3} gives an acceptable bound for the first term. For the second and the third term we again use estimate \eqref{eq:claim1} and \eqref{eq:claim2}, in addition to the uniform disposablity of the modulation projections.
     \end{proof}

We now come to the proof of Theorem \ref{thm:div-prob}.

\begin{proof}[Proof of Theorem \ref{thm:div-prob}]
We start with the proof of the algebra property, \eref{eqn:thm div-prob:S-alg}. Let $\mu = \min\{\lambda_0, \lambda_1, \lambda_2\}$ and decompose the product into
   \begin{equation}\label{eqn:thm div-prob:alg est decomp} \begin{split}u_{\lambda_1} v_{\lambda_2} ={}& \big( u_{\lambda_1} v_{\lambda_2} - C_{\lesa \mu} u_{\lambda_1} C_{\lesa \mu} v_{\lambda_2}\big) \\
&{}+ C_{\gg \mu}\big( C_{\lesa \mu} u_{\lambda_1} C_{\lesa \mu } v_{\lambda_2} \big) + C_{\lesa \mu} \big( C_{\lesa \mu} u_{\lambda_1} C_{\lesa \mu} v_{\lambda_2}\big).
     \end{split}
   \end{equation}
For the first term, at least one of $u_{\lambda_1}$ or $v_{\lambda_2}$ must have modulation at distance $\gg \mu$ from the cone, hence an application of Theorem \ref{thm:high-mod} gives
   \begin{align*}
\big\| P_{\lambda_0} \big( u_{\lambda_1} v_{\lambda_2} - C_{\lesa \mu} u_{\lambda_1} C_{\lesa \mu} v_{\lambda_2}\big) \big\|_S \lesa{}& \mu^{\frac{n}{2}} \Big( \frac{\min\{\lambda_1, \lambda_2\}}{\mu}\Big)^\frac{3}{2} \| u_{\lambda_1} \|_S \| v_{\lambda_2} \|_S\\
 \lesa{}& \Big( \frac{\lambda_1 \lambda_2}{\lambda_0} \Big)^{\frac{n}{2}} \|u_{\lambda_1} \|_S \| v_{\lambda_2}\|_S.
   \end{align*}
   For the second term in \eref{eqn:thm div-prob:alg est decomp}, Lemma \ref{lem:resonance bound full cone} implies that we have the identity
    $$C_{\gg \mu}P_{\lambda_0}\big( C_{\lesa \mu} u_{\lambda_1} C_{\lesa \mu } v_{\lambda_2} \big) = C_{\approx \lambda_{max}}P_{\lambda_0}\big( C_{\lesa \mu} u_{\lambda_1} C_{\lesa \mu } v_{\lambda_2} \big) $$
where $\lambda_{max} = \max\{\lambda_0, \lambda_1, \lambda_2\}$. Hence Lemma \ref{lem:S properties} and Theorem \ref{thm:bilinear L2 bound} give
    \begin{align*}
      \big\| C_{\gg \mu}P_{\lambda_0}\big( C_{\lesa \mu} u_{\lambda_1} C_{\lesa \mu } v_{\lambda_2} \big) \big\|_S &\lesa \lambda_{max}^\frac{1}{2} \frac{\lambda_{max}}{\lambda_0} \big\| C_{\gg \mu}P_{\lambda_0}\big( C_{\lesa \mu} u_{\lambda_1} C_{\lesa \mu } v_{\lambda_2} \big)  \big\|_{L^2_{t,x}} \\
      &\lesa  \lambda_{max}^\frac{1}{2} \frac{\lambda_{max}}{\lambda_0} \mu^{\frac{n-1}{2}} \Big( \frac{\min\{\lambda_1, \lambda_2\}}{\mu} \Big)^\epsilon \| u_{\lambda_1} \|_S \| v_{\lambda_2} \|_S \\
      &\lesa \Big( \frac{\lambda_1 \lambda_2}{\lambda_0} \Big)^\frac{n}{2} \| u_{\lambda_1} \|_S \| v_{\lambda_2} \|_S
    \end{align*}
provided we take $\epsilon>0$ sufficiently small, and $n \g 2$. For the last term in \eref{eqn:thm div-prob:alg est decomp}, we apply \eref{eq:tri-est1} in Lemma \ref{lem:tri-est} together with the characterisation of $S$ in Lemma \ref{lem:S properties} to deduce that
    \begin{align*}
      \big\|  C_{\lesa \mu} \big( C_{\lesa \mu} u_{\lambda_1} C_{\lesa \mu} v_{\lambda_2}\big)\big\|_S &\lesa \sup_{\substack{\phi \in C^\infty_0 \\ \| \phi\|_{S_w} \les 1}} \Big| \int_\RR \lr{ |\nabla| \Box C_{\lesa \mu} \phi_{\lambda_0}, C_{\lesa \mu} u_{\lambda_1} C_{\lesa \mu} v_{\lambda_2} }_{L^2_x} dt \Big| \\
      &\lesa \mu^{\frac{n-1}{2}} (\min\{\lambda_1, \lambda_2\})^\frac{1}{2} \| u_{\lambda_1} \|_S \| v_{\lambda_2} \|_S.
    \end{align*}
Therefore \eref{eqn:thm div-prob:S-alg} follows.

We now turn to the proof of \eref{eqn:thm div-prob:S-nonlin}. The argument is in some a sense a dual version of the argument used to prove the algebra property \eref{eqn:thm div-prob:S-alg}. An application of the characterisation of $S$ in \eref{lem:S properties}, together with the invariance under complex conjugation of $S$ and $S_w$, reduces the problem to proving that
    \begin{equation}\label{eqn:thm div-prob:dual}
        \Big| \int_\RR \lr{ \phi_{\lambda_0} u_{\lambda_1}, \Box v_{\lambda_2}}_{L^2_x} dt \Big| \lesa  \Big( \frac{\lambda_1 \lambda_2}{\lambda_0} \Big)^\frac{n}{2} \lambda_0 \| \phi_{\lambda_0} \|_{S_w} \| u_{\lambda_1} \|_S \| v_{\lambda_2} \|_S
    \end{equation}
for all $\phi \in C^\infty_0$. The first step in the proof of \eref{eqn:thm div-prob:dual} is to decompose into the far cone and close cone regions
    \begin{equation}\label{eqn:thm div-prob:dual decomp}
\begin{split}
   \phi_{\lambda_0} u_{\lambda_1}=& \big( \phi_{\lambda_0} u_{\lambda_1} - C_{\lesa \mu} \phi_{\lambda_0}  C_{\lesa \mu} u_{\lambda_1}\big) + C_{\gg \mu}\big( C_{\lesa \mu} \phi_{\lambda_0} C_{\lesa \mu } v_{\lambda_2} \big) \\
&{}\quad + C_{\lesa \mu} \big( C_{\lesa \mu} \phi_{\lambda_0} C_{\lesa \mu} u_{\lambda_1} \big).
 \end{split}
\end{equation}
For the first term in \eref{eqn:thm div-prob:dual decomp}, at least one of $\phi_{\lambda_0}$ or $u_{\lambda_1}$ must have modulation $\gg \mu$. Hence Theorem \ref{thm:high-mod} and Lemma \ref{lem:S properties} gives
    \begin{align*}
      &\Big| \int_\RR \lr{\big( \phi_{\lambda_0} u_{\lambda_1} - C_{\lesa \mu} \phi_{\lambda_0}  C_{\lesa \mu} u_{\lambda_1}\big), \Box v_{\lambda_2}}_{L^2_x} dt \Big|\\
            \lesa{}& \lambda_2 \big\| P_{\lambda_2} \big( \phi_{\lambda_0} u_{\lambda_1} - C_{\lesa \mu} \phi_{\lambda_0}  C_{\lesa \mu} u_{\lambda_1}\big)\big\|_{S_w} \| v_{\lambda_2} \|_S \\
            \lesa{}& \mu^{\frac{n}{2}} \Big( \frac{\min\{\lambda_0, \lambda_1\}}{\mu}\Big)^\frac{1}{2} \lambda_2 \| \phi_{\lambda_0} \|_{S_w} \|u_{\lambda_1} \|_S \| v_{\lambda_2} \|_S \\
            \lesa{}& \Big( \frac{\lambda_1 \lambda_2}{\lambda_0} \Big)^\frac{n}{2} \lambda_0  \| \phi_{\lambda_0} \|_{S_w} \|u_{\lambda_1} \|_S \| v_{\lambda_2} \|_S.
    \end{align*}
On the other hand, to bound the second term in \eref{eqn:thm div-prob:dual decomp}, we note that Lemma \ref{lem:resonance bound full cone} gives the identity
     $$C_{\gg \mu}P_{\lambda_2}\big( C_{\lesa \mu} \phi_{\lambda_0} C_{\lesa \mu} u_{\lambda_1} \big) = C_{\approx \lambda_{max}}P_{\lambda_2}\big( C_{\lesa \mu} \phi_{\lambda_0} C_{\lesa \mu} u_{\lambda_1}\big) $$
and hence Theorem \ref{thm:bilinear L2 bound} together with Lemma \ref{lem:S properties} gives
     \begin{align*}
      &\Big| \int_\RR \lr{C_{\gg \mu}P_{\lambda_0}\big( C_{\lesa \mu} \phi_{\lambda_0} C_{\lesa \mu} u_{\lambda_1} \big), \Box v_{\lambda_2}}_{L^2_x} dt \Big|\\
            \lesa{}& \lambda_2 \lambda_{max}^{\frac{1}{2}} \big\| C_{\approx \lambda_{max}}P_{\lambda_2}\big( C_{\lesa \mu} \phi_{\lambda_0} C_{\lesa \mu} u_{\lambda_1}\big)\big\|_{S_w} \| v_{\lambda_2} \|_S \\
            \lesa{}& \mu^{\frac{n-1}{2}} \Big( \frac{\min\{\lambda_0, \lambda_1\}}{\mu} \Big)^\epsilon \lambda_2 \lambda_{\max}^\frac{1}{2} \| \phi_{\lambda_0} \|_{S_w} \|u_{\lambda_1} \|_S \| v_{\lambda_2} \|_S \\
            \lesa{}& \Big( \frac{\lambda_1 \lambda_2}{\lambda_0} \Big)^\frac{n}{2} \lambda_0  \| \phi_{\lambda_0} \|_{S_w} \|u_{\lambda_1} \|_S \| v_{\lambda_2} \|_S.
    \end{align*}
Finally, to bound the last term in \eref{eqn:thm div-prob:dual decomp}, we apply the close cone bound in Lemma \ref{lem:tri-est} and conclude that
    \begin{align*}
      &\Big| \int_\RR \lr{C_{\gg \mu}P_{\lambda_0}\big( C_{\lesa \mu} \phi_{\lambda_0} C_{\lesa \mu} u_{\lambda_1} \big), \Box v_{\lambda_2}}_{L^2_x} dt \Big| \\
\lesa{}& \mu^{\frac{n-1}{2}} \lambda_2 \Big( \frac{\min\{\lambda_0, \lambda_1\}}{\mu}\Big)^{\frac{1}{2} + \epsilon} \Big( \frac{\lambda_1}{\mu}\Big)^\epsilon \| \phi_{\lambda_0} \|_{S_w} \| u_{\lambda_1} \|_S \| v_{\lambda_2}\|_S \\
      \lesa{}& \Big( \frac{\lambda_1 \lambda_2}{\lambda_0} \Big)^\frac{n}{2} \lambda_0 \| \phi_{\lambda_0} \|_{S_w}  \| u_{\lambda_1} \|_S \| v_{\lambda_2} \|_S.
    \end{align*}
Therefore we obtain \eref{eqn:thm div-prob:dual}.
\end{proof}

\section{The spaces $U^p$ and $V^p$}\label{sec:Up and Vp}
In this section we briefly review some of the fundamental properties of the  spaces $U^p$ and $V^p$, in particular we discuss embedding properties, almost orthogonality principles, and the dual pairing. The material we present in this section can essentially be found in \cite{Hadac2009, Koch2005,Koch2016} and thus we shall be somewhat brief but self-contained (up to the proof of Theorem \ref{thm:vu-emb}).
Using the notation introduced in Subsection \ref{subsec:upvp}, a function $u:\RR \rightarrow L^2(\RR^n)$ is a \emph{step function with partition $\tau\in \mb{P}$} if for each $I \in \mc{I}_\tau$ there exists  $f_I \in L^2(\RR^n)$  such that
    $$ u(t) =  \sum_{I \in \mc{I}_\tau} \ind_{I}(t) f_I.$$
By definition, all step functions vanish for sufficiently negative $t$. Define $\mathfrak{S}$ to be the collection of all step functions with partitions $\tau\in\mb{P}$, and let $\mathfrak{S}_0$ denote those elements of $\mathfrak{S}$ with compact support. In other words, step functions in $\mathfrak{S}_0$ vanish on the final interval $[t_N, \infty)$.

The normalisation condition at $t = - \infty$ implies that $\mathfrak{S} \subset V^p$. However step functions are \emph{not} dense in $V^p$,
 this follows by considering an example of Young \cite{Young1936} that shows that $V^p \not \subset U^p$ \cite[Theorem B.22]{Koch2016}. On the other hand the set of step functions, $\mathfrak{S}$, is dense in $U^p$, this follows by noting that $U^p$ atoms are step functions, and if $u = \sum c_j u_j$ is an atomic decomposition of $u \in U^p$, then letting $\phi_N = \sum_{j\les N } c_j u_j$ we get a sequence of step functions such that
    $$ \| u - \phi_N \|_{U^p} \les \sum_{j\g N} |c_j|  \rightarrow 0$$
as $N \to \infty$.

Recall that
     $$ |v|_{V^p} = \sup_{(t_j)_{j=1}^{N} \in \mb{P}} \Big( \sum_{j=1}^{N-1} \| v(t_{j+1}) - v(t_j) \|_{L^2}^p \Big)^{\frac{1}{p}}.$$
A computation shows that for any $v:\RR \to L^2$ and $t_0 \in \RR$ we have
        $$ 2^{-1}(\| v(t_0)\|_{L^2}^p + |v|_{V^p}^p)^{\frac{1}{p}}\les \| v\|_{V^p} \les 2(\| v(t_0)\|_{L^2}^p + |v|_{V^p}^p)^{\frac{1}{p}}.$$
In particular, if $v \in V^p$, then letting $t_0\rightarrow -\infty$, we see that $|\cdot|_{V^p}$ and $\|\cdot \|_{V^p}$ are equivalent norms on $V^p$.
The definition of the spaces $U^p$ and $V^p$ implies that they are both Banach spaces and convergence with respect to $\| \cdot \|_{U^p}$ or $\| \cdot \|_{V^p}$ implies uniform convergence (i.e. in $\| \cdot\|_{L^\infty_t L^2_x}$).

\subsection{Embedding properties}\label{subsec:funda}
For $1 \les  q < r < \infty$ we have the continuous  embeddings  $U^q \subset V^q \subset V^r $. An example due to Young \cite{Young1936} (see \cite[Lemma 4.15]{Koch2014}) shows that $V^q \not \subset U^q$. On the other hand, for $p<q$ we do have $V^p \subset U^q$.

\begin{theorem}\label{thm:vu-emb}
Let $1\les p < q < \infty$ and $v \in V^p$. There exists a decomposition $v = \sum_{j=1}^\infty v_j$ such that $v_j \in U^q$ with $ \| v_j \|_{U^q} \lesa 2^{j( \frac{p}{q}-1)} \|v\|_{V^p}$. In particular, the embedding $V^p \subset U^q$ holds.
\end{theorem}
\begin{proof}
The proof can be found in \cite[pp. 255--256]{Koch2005}, \cite[p. 923]{Hadac2009} or \cite[Proof of Theorem B.18]{Koch2016}.
\end{proof}

\begin{remark}\label{rmk:u-vs-v}
To gain some intuition into the $U^p$ and $V^p$ spaces, we note the following properties:
\begin{enumerate}
	\item ($C^\infty_0 \subset U^p, V^p$) Clearly we have $C^\infty_0 \subset V^1$. Consequently an application of Theorem \ref{thm:vu-emb}  implies that for $p>1$ we have $C^\infty_0 \subset U^p$. More generally, if $\partial_t u \in L^1_t L^2_x$ and $\lim_{t\to -\infty} u(t)=0$, then $u\in V^1\subset U^p$ and
\begin{equation}\label{eq:w11}
\|u\|_{U^p}\lesa\|u\|_{V^1}\lesa \|\partial_t u\|_{ L^1_t L^2_x}
\end{equation}
by Theorem \ref{thm:vu-emb}.
	\item (Approximation of $C^\infty_0$ functions in $V^p$, $p>1$) The example of Young shows that $\mathfrak{S}$ is not dense in $V^p$. On the other hand, $\mathfrak{S}$ is dense in the closure of $V^p \cap C^\infty_0$ for $p>1$. This follows by observing that given $\tau \in \mb{P}$ and defining $u_\tau = \sum_{j=1}^N \ind_{[t_j, t_{j+1})} u(t_j)$ (with $t_{N+1} = \infty$) we have
		$$ |u-u_\tau|_{V^p}^p \lesa \Big(\sup_j \int_{t_j}^{t_{j+1}} \| \p_t u \|_{L^2}\Big)^{p-1} \int_\RR \| \p_t u \|_{L^2_x} dt. $$
	
	\item (Counterexample for $p=1$) If we suppose that $u(t) = \rho(t) f $ with $f\in L^2$ and $\rho$ monotone on an interval $(a,b)$, then $u$ cannot be approximated by step functions using the $V^1$ norm. As a consequence $u \not \in U^1$. In particular, if
			$$ \rho(t) = \begin{cases} t \qquad & t\in [0,1) \\
								0  & t \not \in [0,1)\end{cases}$$
then $u \in U^p$ but \emph{not} in $U^1$. In fact, a similar example shows that $C^\infty_0 \not \subset U^1$.	
\end{enumerate}
\end{remark}

\subsection{Almost orthogonality}\label{subsec:ao}
The definition of the spaces $U^p$ and $V^p$ implies that they satisfy a one sided version of the standard almost orthogonality property.

\begin{proposition}[Almost orthogonality in $U^p$ and $V^p$] \label{prop:orthog}
Let $M_1, M_2 \in [0, \infty]$. For $k\in \NN$, let $T_k:L^2(\RR^n)\to L^2(\RR^n)$ be a linear and bounded operator (acting spatially) such that
\[ M_1 \| f \|_{L^2} \les	\Big(\sum_{k \in \NN} \|T_k f\|_{L^2}^2\Big)^{\frac12}\les M_2 \|f\|_{L^2} \qquad  \text{ for all }f\in L^2.\]
If $1\leq p\leq 2$, then for all $u \in U^p$ we have the bound
        $$ \Big( \sum_{k\in \NN} \| T_k u \|_{U^p}^2 \Big)^\frac{1}{2} \les M_2 \| u \|_{U^p}.$$
On the other hand, if $p\geq 2$, then for all $v \in V^p$ we have the bound
        $$ \Big\|\sum_{k \in \NN} T_kv\Big\|_{V^p}\les M_1 \Big( \sum_{k\in \NN} \| T_k v \|_{V^p}^2 \Big)^\frac{1}{2}.
$$
\end{proposition}
\begin{proof}
We start with the $U^p$ bound. It is enough to consider the case where $u = \sum_{1\les m\les N} \ind_{[t_m, t_{m+1})}(t) f_k$ is a $U^p$-atom. Then by definition of the $U^p$ norm, together with the assumption $1\les p\les 2$, we have
    \begin{align*}
       \Big( \sum_{k\in \NN} \| T_k u \|_{U^p}^2 \Big)^\frac{1}{2} &\les \Big( \sum_{k\in \NN} \Big( \sum_{1\les m\les N} \|T_kf_m \|_{L^2}^p \Big)^\frac{2}{p} \Big)^\frac{1}{2}\\ &\les  \Big(\sum_{1\les m\les N}  \Big( \sum_{k\in \NN} \|T_kf_m\|^2 \Big)^\frac{p}{2} \Big)^\frac{1}{p}\les M_2,
    \end{align*}
as required.

Concerning the $V^p$ bound, we let $\tau=(t_j)_{j=1}^N\in \mathbf{P}$ be any partition. We compute
\begin{align*}
&\Big( \sum_{j=1}^{N-1}\Big\|\sum_{k \in \NN}T_kv(t_{j+1}) -\sum_{k \in \NN}T_kv(t_{j})\Big\|_{L^2}^p \Big)^\frac{1}{p}\\
\les{}& \Big(\sum_{j=1}^{N-1}\Big\|\sum_{k \in \NN}T_k(v(t_{j+1}) -v(t_{j}))\Big\|_{L^2}^p \Big)^\frac{1}{p} \\
 \les{}& M_1 \Big(\sum_{j=1}^{N-1}\Big(\sum_{k \in \NN}\Big\|  T_k(v(t_{j+1}) -v(t_{j}))\Big\|_{L^2}^2 \Big)^\frac{p}{2} \Big)^{\frac1p}\\
\les{}& M_1 \Big(\sum_{k \in \NN}\big\|  T_kv\big\|_{V^p}^2 \Big)^\frac{1}{2},
\end{align*}
as required.
\end{proof}

\subsection{The dual pairing}\label{subsec:dual}
Due to the atomic definition of $U^p$, it can be very difficult to estimate $\| u \|_{U^p}$ for $u\in U^p$. For instance, even the question of estimating the $U^p$ norm of a step function is a nontrivial problem. However, there is a duality argument that can reduce this problem to estimating a certain bilinear pairing between elements of $U^p$ and $V^p$. In the following we give two possible versions of this pairing, a discrete version and a continuous version. More precisely,  given a step function $w \in  \mathfrak{S}$ with partition $(t_j)_{j=1}^{N} \in \mb{P}$, and a function $u: \RR \to L^2$, we define the dual pairing
            $$ B(w, u) =  \lr{w(t_1), u(t_1)}_{L^2} + \sum_{j=2}^N \lr{w(t_j) - w(t_{j-1}), u(t_j)}_{L^2}.$$
The bilinear pairing $B$ is well defined, and sesquilinear in $w\in \mathfrak{S}$ and $u$. The linearity in $u$ is immediate, while the remaining properties follow by observing that if $w \in \mathfrak{S}$ is a step function with partition $\tau=(t_j) \in \mb{P}$, that is \emph{also} a step function with respect to another partition $\tau'=(t'_j) \in \mb{P}$, then a computation using the fact that $w(t) =0$ for $t<\max\{t_1, t_1'\}$ gives
        \begin{align*} \lr{w(t_1), u(t_1)}_{L^2} + &\sum_{j=2}^N \lr{w(t_j) - w(t_{j-1}), u(t_j)}_{L^2} \\
        &= \lr{w(t'_1), u(t'_1)}_{L^2} + \sum_{k=2}^{N'} \lr{w(t_j') - w(t_{j-1}'), u(t_j')}_{L^2}.
        \end{align*}
It is not so difficult to show that if $w \in \mf{S}$ and $u \in U^p$ then \[|B(w, u)| \les |v|_{V^q} \| u \|_{U^p},\] we give the details of this computation in Theorem \ref{thm:dual pairing} below. Thus $B(w,u)$ gives a way to pair (step) functions in $V^q$ with $U^p$ functions.
Alternative pairings are also possible, for instance we can also use the continuous pairing
    \begin{equation}\label{eqn:cts pairing} \int_\RR \lr{\p_t \phi, u}_{L^2} dt \end{equation}
for $\p_t \phi \in L^1_t L^2_x$ and $u \in U^p$. Clearly \eqref{eqn:cts pairing} and $B$ are closely related, since if $w\in \mf{S}$ and $u \in C^1$ we have
    $$B(w, u) =  \int_\RR \lr{v, \p_t u} dt$$
where the step function \[v(t) =  \ind_{(-\infty, t_1)}(t) w(t_N) - \sum_{j=1}^{N-1} \ind_{[t_j, t_{j+1})}(t) [w(t_{N}) - w(t_{j-1})].\] If we had $v \in C^1$, then integrating by parts would give \eqref{eqn:cts pairing}. More general pairings are also possible, this relates to the more general \emph{Stieltjes} integral, and using a limiting argument, the definition of $B(w,u)$ can be extended from $w\in \mf{S}$ to elements $w\in V^p$, see \cite{Hadac2009}. However for our purposes, it suffices to work with the pairings $B$ and \eqref{eqn:cts pairing} as defined above. The pairings are useful due to the following.

\begin{theorem}[Dual pairing for $U^p$]\label{thm:dual pairing}
Let $1< p, q < \infty$, $\frac{1}{p} + \frac{1}{q} = 1$. If  $u \in U^p$ and $\p_t \phi \in L^1_t L^2_x$ we have
         $$ \Big| \int_\RR \lr{\p_t \phi, u}_{L^2} dt \Big| \les |\phi|_{V^q} \| u \|_{U^p}$$
and
        $$  \| u \|_{U^p} =  \sup_{ \substack{w \in \mathfrak{S} \\ | w |_{V^q} \les 1} } |B(w, u)| = \sup_{\substack{\p_t \phi \in C^\infty_0 \\ |\phi|_{V^q} \les 1}} \Big| \int_\RR \lr{\p_t \phi, u}_{L^2} dt\Big|. $$

\end{theorem}
\begin{proof}
We start by proving that for every $u \in U^p$ we have
	\begin{equation}\label{eqn:thm dual pairing:upper bounds}
		 \sup_{\substack{\p_t \phi \in L^1_t L^2_x \\ |\phi|_{V^q} \les 1}} \Big| \int_\RR \lr{\p_t \phi, u}_{L^2} dt\Big| \les \sup_{ \substack{w \in \mathfrak{S} \\ | w |_{V^q} \les 1} } |B(w, u)| \les \| u \|_{U^p}.
	\end{equation}
By definition of the $U^p$ norm, it is enough to consider the case where $u \in \mathfrak{S}$ is a $U^p$-atom. The second inequality in \eqref{eqn:thm dual pairing:upper bounds} follows by observing that if $w \in \mf{S}$ with partition $\tau= (t_j)_{j=1}^N \in \mb{P}$, and we let $ E = \{ t_j \in \tau \mid u(t_j) = u(t_{j+1})\}$ and $\tau^* = (t_k^*)_{k=1}^{N^*} = \tau \setminus E$ (i.e. we remove points from $\tau$ where the step function $u$ is constant) then  by definition of the bilinear pairing we have
	\begin{align*}
		|B(w, u)| &=\Big| \lr{ w(t_1), u(t_1) }_{L^2} + \sum_{j=2}^N \lr{w(t_j) - w(t_{j-1}), u(t_j)}_{L^2 }\Big|\\
				&= \Big|\lr{ w(t_1^*), u(t_1^*) }_{L^2} + \sum_{j=2}^{N^*} \lr{w(t_j^*) - w(t_{j-1}^*), u(t_j^*)}_{L^2 }\Big|\\
				&\les \Big( \|w(t_1^*)\|_{L^2}^q + \sum_{j=2}^{N^*} \| w(t_j^*) - w(t_{j-1}^*)\|_{L^2}^q \Big)^{\frac{1}{q}} \Big( \sum_{j=1}^{N^*}\| u(t_j^*)\|_{L^2}^p\Big)^\frac{1}{p} \les |w|_{V^q}.
	\end{align*}
where the last line follows by noting that $w(t) = 0$ for $t<t_1^*$ and $u$ is a $U^p$ atom together with the definition of the partition $\tau^*$. On the other hand, to prove the first inequality in \eqref{eqn:thm dual pairing:upper bounds}, we observe that if $u = \ind_{[s_{N'},\infty)}(t) u(s_{N'}) + \sum_{k=1}^{N'-1} \ind_{[s_k, s_{k+1})}(s) u(s_k)$ and $T>s_{N'}$, we have
	\begin{align*}
		\int_{-\infty}^T \lr{\p_t \phi, u}_{L^2} dt &= \lr{\phi(T) - \phi(s_{N'}), u(s_{N'})}_{L^2} + \sum_{j=1}^{N'} \lr{ \phi(s_{j+1}) - \phi(s_j), u(s_j)}_{L^2} \\
							&= 	\lr{w(s_1), u(t_1)}_{L^2} + \sum_{j=2}^{N'} \lr{w(s_{j}) - w(s_{j-1}), u(t_j)}_{L^2}
	\end{align*}
where we define $w\in \mf{S}$ as
	$$w(t) = \sum_{j=1}^{N'-1} \ind_{[s_j, s_{j+1})}(t) \big( \phi(s_{j+1}) - \phi(s_1)\big) + \ind_{[s_{N'}, \infty)}(t) \big( \phi(T) - \phi(s_1)\big).$$
Since $|w|_{V^q} \les |\phi|_{V^q}$ we conclude that
	$$ \Big| \int_\RR \lr{ \p_t \phi, u}_{L^2} dt \Big| \les  \int_T^\infty \| \p_t \phi(t) \|_{L^2} dt \| u \|_{L^\infty_t L^2_x} + |\phi|_{V^q} \sup_{ \substack{w \in \mf{S} \\ |w|_{V^q}\les 1}} |B(w,u)|. $$
Hence letting $T \to \infty$ and using the assumption $\p_t \phi \in L^1_t L^2_x$ the bound \eqref{eqn:thm dual pairing:upper bounds} follows.

It remains to show that for $u \in U^p$ we have the bound
		\begin{equation}\label{eqn:thm dual pairing:Up bound}
			\| u \|_{U^p} \les  \sup_{\substack{\p_t \phi \in L^1_t L^2_x \\ |\phi|_{V^q} \les 1}} \Big| \int_\RR \lr{\p_t \phi, u}_{L^2} dt\Big|.
		\end{equation}
Since the set of step functions $\mf{S}\subset U^p$ is dense, it suffices to consider the case $u = \ind_{[t_N, \infty)}(t) f_N + \sum_{j=1}^{N-1} \ind_{[t_j, t_{j+1})}(t) f_j\in \mf{S}$. An application of the Hahn-Banach Theorem implies that there exists $L \in (U^p)^*$ such that
        \begin{equation}\label{eqn:thm dual pairing:prop of L} \| u \|_{U^p} = L(u), \qquad \qquad \sup_{\|h \|_{U^p}\les 1} |L(h)| = 1.\end{equation}
Note that given $f \in L^2$ and fixed $t \in \RR$, we have $\ind_{[t, \infty)} f \in U^p$ and $\| \ind_{[t, \infty)} f \|_{U^p} \les \| f \|_{L^2}$. In particular, the map $f \mapsto L(\ind_{[t, \infty)}f)$ is a linear functional on $L^2$. Consequently, by the Riesz Representation Theorem, there exists a function $\psi: \RR \to L^2$ such that for every $t\in \RR$ and $f \in L^2$ we have
    $$ L\big( \ind_{[t, \infty)} f\big) =  \lr{ \psi(t), f}_{L^2}. $$
By construction, we see that
	\begin{align*}
		\| u \|_{U^p} = T(u) &= \sum_{j=1}^{N-1} L( \ind_{[t_j, t_{j+1})} f_j ) + L(\ind_{[t_N, \infty)} f_N) \\
					&= \sum_{j=1}^{N-1} \lr{\psi(t_j) - \psi(t_{j+1}), f_j}_{L^2} + \lr{\psi(t_N), f_N}.
	\end{align*}
Let $\rho \in C^\infty_0(-1, 1)$ with $\int_\RR \rho = 1$, and define
	$$ \phi_\epsilon(t) =  - \int_\RR \frac{1}{\epsilon} \rho\Big( \frac{t-s}{\epsilon} + 1\Big) w(t) dt $$
with
	$$ w(t) = \ind_{(-\infty, t_1)}(t) \psi(t_1) + \sum_{j=1}^{N-1} \ind_{[t_j, t_{j+1})}(t) \psi(t_j). $$
Then provided we choose $\epsilon>0$ sufficiently small (depending on $(t_j) \in \mb{P}$), we have $\phi_\epsilon(t_j) = \psi(t_j)$ for $j=1, ..., N$, and $\phi_\epsilon(t) = 0$ for $t \g t_N + 2 \epsilon$ and consequently
	\begin{align*}
		 \| u \|_{U^p} &= \sum_{j=1}^{N-1} \lr{\psi(t_j) - \psi(t_{j+1}), f_j}_{L^2} + \lr{\psi(t_N), f_N} \\
		 			   &= \int_\RR \lr{ \p_t \phi_\epsilon, u}_{L^2} dt .
	\end{align*}
Therefore, since $\p_t \phi_\epsilon \in C^\infty_0$ and $|\phi_\epsilon|_{V^q} \les |w|_{V^q}$, it only remains to show that $ |w|_{V^q} \les 1$. To this end, we start by observing that
	\begin{equation}\label{eqn:thm dual pairing:w bound} |w|_{V^q} \les \sup_{(s_j) \in \mb{P}} \Big( \sum_{j=1}^{N'-1} \| \psi(s_{j+1}) - \psi(s_j) \|_{L^2}^q + \| \psi(s_{N'})\|_{L^2}^q \Big)^\frac{1}{q}.
	\end{equation}
Fix $(s_j)_{j=1}^{N'} \in \mb{P}$ and define $v \in \mf{S}$ as
	$$ v(t) = \ind_{[s_{N'}, \infty)}(t) \frac{ \alpha  \psi(s_{N'})}{\| \psi(s_{N'})\|_{L^2}^{2-q}} +  \sum_{j=1}^{N'-1}
	            \ind_{[s_j, s_{j+1})}(t) \frac{\alpha [\psi(s_{j}) - \psi(s_{j+1})]}{ \| \psi(s_{j+1}) - \psi(s_j)\|_{L^2}^{2-q}}$$
with
	$$ \alpha = \Big( \| \psi(s_{N'}) \|_{L^2}^q + \sum_{1\les j \les N'-1} \| \psi(s_{j+1}) - \psi(s_j) \|_{L^2}^q\Big)^{\frac{1}{q}-1}.$$
Then $v$ is a $U^p$ atom, and by construction, we have
	\begin{align*} L(v) &= \lr{ \psi(s_{N'}), v(s_{N'})}_{L^2} + \sum_{j=1}^{N'-1} \lr{ \psi(s_j) - \psi(s_{j+1}), v(s_j)}_{L^2} \\
	&= \Big( \| \psi(s_{N'}) \|_{L^2}^q + \sum_{j=1}^{N'-1} \| \psi(s_{j+1}) - \psi(s_j)\|_{L^2}^q \Big)^\frac{1}{q}.
    \end{align*}
Therefore, from \eqref{eqn:thm dual pairing:prop of L} and \eqref{eqn:thm dual pairing:w bound} we conclude that
	$$ |w|_{V^q} \les \sup_{ \substack{ v \in \mf{S} \\ \| v \|_{U^p} \les 1}} |L(v)| \les 1$$
as required.
\end{proof}

\begin{remark}\label{rem:compact support}
It is possible to replace the conditions $w \in \mf{S}$ and $\p_t \phi \in C^\infty_0$ in Theorem \ref{thm:dual pairing} with the compactly supported functions $w \in \mf{S}_0$ and $\phi \in C^\infty_0$. More precisely, provided that $u(t) \to 0$ as $t \to -\infty$, we have
	\begin{equation}\label{eqn:rem on comp:dis} \sup_{\substack{ w \in \mf{S} \\ |w|_{V^q}\les 1}} |B(w,u)| \les 2 \sup_{\substack{ w \in \mf{S}_0 \\ |w|_{V^q}\les 1}} |B(w,u)|
	\end{equation}
and
	\begin{equation}\label{eqn:rem on comp:cts}
		 \sup_{\substack{\p_t \phi \in C^\infty_0 \\ |\phi|_{V^q} \les 1}} \Big| \int_\RR \lr{ \p_t \phi, u}_{L^2} dt\Big|
		 	\les 2 \sup_{\substack{\phi \in C^\infty_0 \\ |\phi|_{V^q} \les 1}} \Big| \int_\RR \lr{ \p_t \phi, u}_{L^2} dt\Big|.
	\end{equation}
The factor of 2 arises as we potentially add another jump in the step function $w$ at $t=+\infty$ by imposing the compact support condition. To prove the discrete bound \eqref{eqn:rem on comp:dis}, we note that if $w = \ind_{[t_1, t_2)} f_1 + \dots +\ind_{[t_N, \infty)} f_N$ and we take $w_T =  - \ind_{[T, t_1)} f_N + \ind_{[t_1, t_2)} (f_1 - f_N) + \dots + \ind_{[t_{N-1}, t_N)} (f_{N-1} - f_N)$ then for any $u:\RR \to L^2$ we have
	$$ B(w, u) = \lr{f_N, u(T)}_{L^2} + B(w_T, u) .$$
Since $w_T \in \mf{S}_0$, and $|w_T|_{V^q} \les 2 |w|_{V^q}$, we conclude that if $u(t) \to 0$ as $t \to -\infty$, then
	$$ \sup_{\substack{ w \in \mf{S} \\ |w|_{V^q}\les 1}} |B(w, u)| \les 2 \sup_{ \substack{w \in \mf{G}_0 \\ |w|_{V^q} \les 1}} |B(w, u)| $$
which implies  \eqref{eqn:rem on comp:dis}. To prove \eqref{eqn:rem on comp:cts}, we first take a map $\rho:\RR \to \RR$ such that $\int_\RR |\p_t \rho| dt \les 1$, $\rho(t) = 1$ for $t>1$, and $\rho(t) = 0$ for $t<-1$, and let $\rho_T(t) = \rho(t-T)$. Given $\p_t \phi \in C^\infty_0$, we take
		$$ \phi_T(t)  = \rho_T(t) ( \phi(t) - \phi_\infty) $$
where we take $\phi_{\pm\infty} = \lim_{t\to \pm \infty} \phi(t)$, note that $\phi(\pm t) = \phi_{\pm\infty}$ for $t>0$ sufficiently large, since  $\p_t \phi \in C^\infty_0$. Then $\phi_T \in C^\infty_0$, $\p_t  \phi_T = \p_t \phi $ for $t> T+1$, and for all $T<0$ sufficiently negative we have
	\begin{align*}
	|\phi_T|_{V^q} &\les | (1 - \rho_T)( \phi - \phi_{\infty}) |_{V^q} + | \phi - \phi_\infty |_{V^q} \\
	&\les | (1 - \rho_T) (\phi_{-\infty}-\phi_\infty) |_{V^q} + | \phi |_{V^q} \les \| \phi_{-\infty}- \phi_\infty\|_{L^2} + |\phi|_{V^q}\les 2 |\phi|_{V^q}
	\end{align*}
where we used the bound
	$$ \Big( \sum_j | \rho_T(t_{j+1}) - \rho_T(t_j)|^q \Big)^\frac{1}{q} \les \int_\RR |\p_t \rho| dt = 1.$$
Therefore, provided that $\p_t \phi \in C^\infty_0$ and $|\phi|_{V^q} \les 1$, we have
	$$ \Big| \int_\RR \lr{\p_t \phi, u} dt \Big| \les \int_{-\infty}^{T+1} \| \p_t \phi - \p_t \phi_T\|_{L^2} \| u(t) \|_{L^2} dt + 2 \sup_{\substack{ \psi \in C^\infty_0 \\ |\psi|_{V^q } \les 1}} \Big| \int_\RR \lr{\p_t \psi, u}_{L^2} dt \Big| .$$
Consequently, since $u(t) \to 0 $ as $t \to -\infty$, \eqref{eqn:rem on comp:cts} follows by letting $T\to -\infty$.
\end{remark}

\section{Two Characterisations of $U^p$}\label{sec:characterisation}

In this section we consider the problem of determining if a general function $u:\RR \to L^2$ belongs to $U^p$. If we apply the definition of $U^p$, this requires finding an atomic decomposition of $u$, which is in general a highly non-trivial problem. An alternative approach is suggested by Theorem \ref{thm:dual pairing}. More precisely, since the two norms defined by the dual pairings in Theorem \ref{thm:dual pairing} are well-defined for general functions $u:\RR \to L^2$, we can try to use the finiteness of these quantities to characterise $U^p$. Recent work of Koch-Tataru \cite{Koch2016} shows that this is possible by using the discrete pairing $B(w,u)$. In this section we adapt the argument used in \cite{Koch2016}, and show that it is also possible to characterise $U^p$ using the continuous pairing $\int_\RR \lr{\p_t \phi, u}_{L^2} dt$.

Following \cite{Koch2016}, let $u: \RR \to L^2$ and define the semi-norms
    $$ \| u \|_{\widetilde{U}^p_{dis}} = \sup_{ \substack{ v \in \mathfrak{S} \\ | v |_{V^p} \les 1}} |B(v, u)|\in [0,\infty]$$
and its continuous counterpart
    $$ \| u \|_{\widetilde{U}^p_{cts}} = \sup_{\substack{ \p_t \phi \in C^\infty_0 \\ |\phi|_{V^p} \les 1}} \Big| \int_\RR \lr{ \p_t \phi, u}_{L^2} dt \Big| \in [0, \infty] .$$
Note that both $\|\cdot\|_{\widetilde{U}^p_{dis}}$ and $\| \cdot \|_{\widetilde{U}^p_{cts}}$ are only norms after imposing the normalisation $u(t)\to 0$ as $t \to -\infty$.
Our goal is to show that both $\| \cdot \|_{\widetilde{U}^p_{dis}}$ and $\| \cdot \|_{\widetilde{U}^p_{cts}}$ can be used to characterise $U^p$. To make this claim more precise, let $\widetilde{U}^p_{dis}$ be the collection of all right continuous functions $u:\RR \to L^2$ satisfying the normalising condition $u(t) \to 0$ (in $L^2_x$) as $t \to -\infty$, and the bound $\| u \|_{\widetilde{U}^p_{dis}}<\infty$. Similarly, we take $\widetilde{U}^p_{cts}$ to be the collection of all right continuous functions such that $u(t) \to 0$ as $t \to -\infty$ and $\| u \|_{\widetilde{U}^p_{cts}}<\infty$.

\begin{theorem}[Characterisation of $U^p$]\label{thm:characteristion of Up}
Let $1<p<\infty$. Then $U^p = \widetilde{U}^p_{dis} = \widetilde{U}^p_{cts}$.
\end{theorem}

We give the proof of Theorem \ref{thm:characteristion of Up} in Subsection \ref{subsec:proof of charac} below. Roughly, since Theorem \ref{thm:dual pairing} already shows that the norms are equivalent for step functions, and $\mf{S}$ is dense in $U^p$, it is enough to show that $\mf{S}$ is also dense in the spaces $\widetilde{U}^p_{cts}$ and $\widetilde{U}^p_{dis}$. The key step is a density argument which we give for a general bilinear pairing satisfying certain assumptions, this argument closely follows that given in \cite[Appendix B]{Koch2016} for the special case of the discrete pairing.

\subsection{A general density result}\label{subsec:gen-ver}

Let $X \subset V^q$ be a subspace and $B_{rc}\subset L^\infty_t L^2_x$ denote the set of bounded right-continuous ($L^2_x$ valued) functions $u:\RR \to L^2$. Let $\mathfrak{B}(v,u): X \times B_{rc} \to \CC$ be a sesquilinear form. For $u \in B_{rc}$,  define
	$$ \| u \|_{\widetilde{U}^p_{\mathfrak{B}}} = \sup_{\substack{ v \in X \\ | v |_{V^q}\les 1}} | \mathfrak{B}(v,u)| \in [0, \infty]$$
and for $-\infty \les a<b\les \infty$
	$$ \| u \|_{\widetilde{U}^p_{\mathfrak{B}}(a,b)} = \sup_{\substack{ v \in X \\ \supp v \subset (a,b) \\ | v |_{V^q}\les 1}} | \mathfrak{B}(v,u)| \in [0, \infty].$$
Note that, despite the suggestive notation, these quantities may not necessarily be norms, but, since $0 \in X$ they are always well defined. We now take
$\widetilde{U}^p_{\mathfrak{B}}$ to the collection of all $u \in B_{rc}$ such that $u(t) \to 0$ (in $L^2$) as $t \to -\infty$, and $\| u \|_{\widetilde{U}^p_{\mathfrak{B}}}<\infty$. We assume that we have the following properties:

\begin{itemize}
   \item[$\mb{(A1)}$] There exists $C>0$ such that for all $\tau = (t_j)_{j=1}^N \in \mb{P}$ and $u\in \widetilde{U}^p_{\mathfrak{B}}$, we have
	\begin{equation}\label{eqn:decomp prop for general norm}
		\| u - u_\tau \|_{\widetilde{U}^p_{\mathfrak{B}}} \les C \Big( \sum_{j=0}^N \| u \|_{\widetilde{U}^p_{\mathfrak{B}}(t_j, t_{j+1})}^p \Big)^\frac{1}{p},
	\end{equation}
where we take $t_0 = -\infty$, $t_{N+1} = \infty$, and define the step function $u_\tau = \sum_{j=1}^{N} \ind_{[t_j, t_{j+1})}(t) u(t_j) \in \mathfrak{S}$.

    \item[$\mb{(A2)}$] If $v \in X$ and $\epsilon>0$, there exists a step function $w \in \mathfrak{S}$ such that $| v - w |_{V^q} < \epsilon$. \\
\end{itemize}
Under the above assumptions, the set of step functions $\mf{S}$ is dense in $\widetilde{U}^p_{\mf{B}}$.

\begin{theorem}\label{thm:gen approx by step func}
Let $1<p<\infty$ and assume that $\mb{(A1)}$ and $\mb{(A2)}$ hold. Then for every $u \in \widetilde{U}^p_{\mathfrak{B}}$ and $\epsilon>0$, there exists $\tau \in \mb{P}$ such that
    $$ \| u - u_\tau \|_{\widetilde{U}^p_{\mathfrak{B}}}<\epsilon. $$
\end{theorem}

We start by proving two preliminary results.

\begin{lemma}\label{lem:gen approx smallness}
Let $1<p<\infty$ and $\| u \|_{\widetilde{U}^p_{\mathfrak{B}}} <\infty$. For every $\epsilon>0$ there exists a partition $ (t_k)_{k=1}^N \in \mb{P}$ such that
        $$ \sup_{0 \les k \les N} \| u \|_{\widetilde{U}^p_{\mathfrak{B}}(t_k, t_{k+1})} < \epsilon $$
where $t_0 = -\infty$ and $t_{N+1} = \infty$.
\end{lemma}
\begin{proof}
Let $\epsilon>0$. Since $\| \cdot \|_{\widetilde{U}^p_{\mathfrak{B}}(a, b)}$ decreases as $a \nearrow b$,
the compactness of closed and bounded intervals implies that is enough to prove that for every $-\infty < t^* \les \infty$ and $-\infty \les t_* < \infty $ we can find $t_1 < t^*$ and $t_2>t_*$ such that
            \begin{equation}\label{eqn:gen approx smallness:key bound} \| u \|_{\widetilde{U}_{\mathfrak{B}}^p(t_1, t^*)} + \| u \|_{\widetilde{U}_{\mathfrak{B}}^p(t_*, t_2)} \les \epsilon.\end{equation}
We only prove the first bound, as the second one is similar. Suppose that \eref{eqn:gen approx smallness:key bound} fails, then there exists $-\infty<t^* \les \infty$ and $\epsilon>0$ such that for every $t_1<t^*$ we have
            \begin{equation}\label{eqn:gen approx smallness:for contra}
                \| u \|_{\widetilde{U}_{\mathfrak{B}}^p(t_1, t^*)} \g \epsilon.
            \end{equation}
Let $T_1<t^*$. By definition,  \eref{eqn:gen approx smallness:for contra} together with the fact that $X$ is a \emph{subspace} of $V^q$, implies that there exists $w_1 \in X$ such that $\supp w_1 \subset (T_1, t^*)$, $|w_1|_{V^q}\les 1$ (with $\frac{1}{p}+ \frac{1}{q}=1$) and
        $$ \mathfrak{B}(w_1, u) \g \frac{\epsilon}{2}.$$
Let $T_1<T_2< t^*$ such that $\supp w_1 \subset (T_1, T_2)$. Another application of \eref{eqn:gen approx smallness:key bound} gives $w_2 \in X$ such that $\supp w_2 \subset (T_2, t^*)$, $|w_2|_{V^q} \les 1$, and
        $$ \mathfrak{B}(w_2, u) \g \frac{\epsilon}{2}.$$
Continuing in this manner, for every $N \in \NN$ we obtain a sequence $T_1<T_2< \dots < T_{N+1}<t^*$ and functions $w_j \in X$ such that
    $$ |w_j|_{V^q} \les 1, \qquad \supp w_j \subset (T_j, T_{j+1}), \qquad \mathfrak{B}(w_j, u) \g \frac{\epsilon}{2}. $$
If we now let $w = \sum_{j=1}^N w_j$, then using the disjointness of the supports of the $w_j$ we have $|w|_{V^q} \lesa (\sum_{1\les j \les N} |w_j|_{V^q}^q)^{\frac{1}{q}} \lesa N^\frac{1}{q}$ and hence
    $$ \frac{\epsilon}{2} N \les \sum_{j=1}^N \mathfrak{B}(w_j, u) = \mathfrak{B}(w, u) \les |w|_{V^q} \| u \|_{\widetilde{U}^p_{\mathfrak{B}}} \lesa N^{\frac{1}{q}} \| u \|_{\widetilde{U}^p_{\mathfrak{B}}}.$$
Letting $N\to \infty$ we obtain a contradiction. Therefore \eref{eqn:gen approx smallness:key bound} follows.
\end{proof}

The second result gives the existence of functions $w \in X$ satisfying a number of crucial properties.

\begin{lemma}\label{lem:gen approx construction of w}
Let $\frac{1}{p} + \frac{1}{q}=1$ with $1<p<\infty$ and assume that $\mb{(A1)}$ holds. Let $u \in \widetilde{U}^p_{\mathfrak{B}}$ and $\tau = (s_k)_{k=1}^N \in \mb{P}$ with $\| u - u_\tau\|_{\widetilde{U}^p_{\mathfrak{B}}}>0$. There exists $w \in X$ such that $|w|_{V^q} \les 1$, we have the $L^\infty$ bound
        $$ \| w \|_{L^\infty L^2} \lesa  \| u - u_\tau \|_{\widetilde{U}^p_{\mathfrak{B}}}^{1-p} \sup_{0\les k \les N} \| u \|_{\widetilde{U}^p_{\mathfrak{B}}(s_k, s_{k+1})}^{p-1},$$
and the lower bound
        $$ \mathfrak{B}(w, u) \gtrsim   \| u - u_\tau \|_{\widetilde{U}^p_{\mathfrak{B}}}. $$
\end{lemma}
\begin{proof} We begin by observing that by $\textbf{(A1)}$ and the assumption $\| u - u_\tau \|_{\widetilde{U}^p_{\mathfrak{B}}}>0$, we have
        \begin{equation}\label{eqn:gen approx construction of w:sum nonzero} \sum_{k=0}^N \| u \|_{\widetilde{U}^p_{\mathfrak{B}}(s_k, s_{k+1})}^p >0. \end{equation}
In particular, at least one of the terms $\| u \|_{\widetilde{U}^p_{\mathfrak{B}}(s_k, s_{k+1})} \not = 0$. For $k=0, \dots, N$, we define functions $w_k \in X$ as follows. If $\| u \|_{\widetilde{U}^p_{\mathfrak{B}}(s_k, s_{k+1})} = 0$, we take $w_k=0$. On the other hand, if $\| u \|_{\widetilde{U}^p_{\mathfrak{B}}(s_k, s_{k+1})}>0$, we take $w_k \in X$ such that $\supp w_k \subset (s_k, s_{k+1})$, $|w_k|_{V^q} \les 1$, and
        $$ \mathfrak{B}(w_k, u) \g \frac{1}{2} \| u \|_{\widetilde{U}^p_{\mathfrak{B}}(s_k, s_{k+1})}.$$
Such a function $w_k \in X$ exists by definition of $\|\cdot \|_{\widetilde{U}^p_{\mathfrak{B}}(s_k, s_{k+1})}$. We now let $w = \sum_{k=0}^N \alpha \| u \|_{\widetilde{U}^q(s_k, s_{k+1})}^{p-1} w_k$ with
    $$ \alpha = 2^{-\frac{1}{p}} \Big( \sum_{k=0}^N \| u \|_{\widetilde{U}^p_{\mathfrak{B}}(s_k, s_{k+1})}^p \Big)^{\frac{1}{p}-1}.$$
Note that \eref{eqn:gen approx construction of w:sum nonzero} implies that $\alpha <\infty$, and at least one of the functions $w_k$ is non-zero, thus $w \in X \setminus \{0\}$. By construction, the step functions $w_k$ have separated supports, and hence we have the $V^q$ bound
    $$ |w|_{V^q}^q \les  2^{q-1} \alpha^q \sum_{k=0}^N \| u \|_{\widetilde{U}^p_{\mathfrak{B}}(s_k, s_{k+1})}^{q(p-1)} |w_k|_{V^q}^q \les 2^{q-1} \alpha^q \sum_{k=0}^N \| u \|_{\widetilde{U}_{\mathfrak{B}}^p(s_k, s_{k+1})}^p = 1.$$
Moreover, via $\textbf{(A1)}$, we have the $L^\infty_t$ bound
    $$ \| w\|_{L^\infty L^2} \les \alpha \sup_{0\les k \les N} \| u \|_{\widetilde{U}_{\mathfrak{B}}^p(s_k, s_{k+1})}^{p-1} |w_k|_{V^q} \lesa \| u - u_\tau\|_{\widetilde{U}^p_{\mathfrak{B}}}^{1-p} \sup_{0\les k \les N} \| u \|_{\widetilde{U}^p_{\mathfrak{B}}(s_k, s_{k+1})}^{p-1} $$
and the lower bound
    \begin{align*} \mathfrak{B}(w, u) &= \sum_{k=0}^N \alpha \| u \|_{\widetilde{U}_{\mathfrak{B}}^{p}(s_k, s_{k+1})}^{p-1} \mathfrak{B}(w_k, u) \\
    &\g \frac{\alpha}{2}  \sum_{k=0}^N \| u \|_{\widetilde{U}_{\mathfrak{B}}^p(s_k, s_{k+1})}^p
    = 2^{-\frac{1+p}{p}}\Big( \sum_{k=0}^N \| u \|_{\widetilde{U}_{\mathfrak{B}}^p(s_k, s_{k+1})}^p\Big)^\frac{1}{p} \gtrsim \| u - u_\tau\|_{\widetilde{U}_{\mathfrak{B}}^p}.
    \end{align*}
\end{proof}

We now come to the proof of Theorem \ref{thm:gen approx by step func}.

\begin{proof}[Proof of Theorem \ref{thm:gen approx by step func}]
We argue by contradiction. Let $u\in \widetilde{U}^p_{\mathfrak{B}}$ and suppose there exists  $\epsilon_0>0$ such that for every $\tau \in \mb{P}$ we have
    \begin{equation}
        \label{eqn:gen approx by step func:contra assump} \| u - u_\tau \|_{\widetilde{U}^p_{\mathfrak{B}}} \g  \epsilon_0.
    \end{equation}
We claim that \eref{eqn:gen approx by step func:contra assump} implies that for each $N \in \NN$ there exists $w_N \in X$ such that
    \begin{equation}\label{eqn:gen approx by step func:properties of w_N}
        |w_N|_{V^q} \les (2N)^\frac{1}{q}, \qquad \mathfrak{B}(w_N, u) \gtrsim \frac{1}{4} N \epsilon_0 .
    \end{equation}
But, similar to Lemma \ref{lem:gen approx smallness}, this gives a contradiction as $N \to \infty$ since \eref{eqn:gen approx by step func:properties of w_N} implies that
    $$ \frac{1}{4} N \epsilon_0 \lesa \mathfrak{B}(w_N, u) \les |w_N|_{V^q} \| u \|_{\widetilde{U}^p_{\mathfrak{B}}} \les (2N)^\frac{1}{q} \| u \|_{\widetilde{U}^p_{\mathfrak{B}}}.$$
Thus it only remains to show that \eref{eqn:gen approx by step func:contra assump} implies \eref{eqn:gen approx by step func:properties of w_N}. To this end, note that \eref{eqn:gen approx by step func:contra assump} together with $\mb{(A1)}$, Lemma \ref{lem:gen approx smallness}, and Lemma \ref{lem:gen approx construction of w}, implies that for every $\delta >0$ there exists $w' \in X$ such that
    \begin{equation}\label{eqn:gen approx by step func:conseq of lem}
        |w'|_{V^q} \les 1, \qquad \| w'\|_{L^\infty L^2} \les \delta, \qquad B(w', u) \gtrsim  \epsilon_0.
    \end{equation}
Taking $\delta = 1$ gives $w_1$ satisfying \eref{eqn:gen approx by step func:properties of w_N}. Suppose we have $w_N$ satisfying \eref{eqn:gen approx by step func:properties of w_N}. Let $\epsilon^*>0$. By $\mb{(A2)}$, we have a step function $w_N' \in \mathfrak{S}$ such that $|w_N - w_N'|_{V^q} < \epsilon^*$. Choose $w' \in X$ such that \eref{eqn:gen approx by step func:conseq of lem} holds with
    $$ \delta < \epsilon^* \min_{\substack{ t_1<t_2 \\ w_N'(t_1) \not = w_N'(t_2)}} \| w_N'(t_1) - w_N'(t_2) \|_{L^2} $$
(this quantity is non-zero since $w'_N$ is a step function) and define $w_{N+1} = w_N + w' \in X$. To check the required $V^q$ bound, suppose that $\tau = (t_j) \in \mb{P}$ and take
    $$ \tau' = \{ t_j \in \tau \mid w_N'(t_j) = w_N'(t_{j+1}) \}, \qquad \tau'' = \{ t_j \in \tau \mid w_N'(t_j) \not = w_N'(t_{j+1} )\}.$$
Then
    \begin{align*}
      \sum_{j=1}^{N-1} &\| (w_N' + w')(t_{j+1}) - (w_N' + w')(t_j) \|_{L^2}^q \\
        &\les \sum_{t_j \in \tau'} \| w'(t_{j+1}) - w'(t_j) \|_{L^2}^q + (1+2\epsilon^*)^q \sum_{t_j \in \tau''} \| w_N'(t_{j+1}) - w_N'(t_j)\|_{L^2}^q \\
      &\les 1 + (1 + 2 \epsilon^*)^q 2N
    \end{align*}
and consequently, by choosing $\epsilon^*>0$ sufficiently small we have
    $$ | w_{N+1} |_{V^q} \les |w_{N} - w_N'|_{V^q} + |w_N' + w'|_{V^q} \les \epsilon^* + \big( 1 + 2 N (1 + 2 \epsilon^*)\big)^\frac{1}{q} \les \big(2 (N+1)\big)^\frac{1}{q}. $$
On the other hand, we have
    $$ \mf{B}(w_{N+1}, u) = \mf{B}(w_N, u) + \mf{B}(w', u) \gtrsim N \epsilon_0 + \frac{1}{4} \epsilon_0.$$
Consequently $w_{N+1}$ satisfies \eref{eqn:gen approx by step func:properties of w_N} as required.
\end{proof}

\subsection{Proof of Theorem \ref{thm:characteristion of Up}}\label{subsec:proof of charac} Let $1<p<\infty$. An application of Theorem \ref{thm:dual pairing} implies that for every $u\in \mf{S} \subset U^p$ we have
		$$ \| u \|_{U^p} = \| u \|_{\widetilde{U}^p_{dis}} = \| u \|_{\widetilde{U}^p_{cts}}.$$
Since $\mf{S}$ is a dense subset of $U^p$, to prove Theorem \ref{thm:characteristion of Up}, it suffices to show that the set of step functions $\mf{S}$ is also dense in $\widetilde{U}^p_{dis}$ and $\widetilde{U}^p_{cts}$. In view of Theorem \ref{thm:gen approx by step func}, the density of $\mf{S}$ in $\widetilde{U}^p_{dis}$ and $\widetilde{U}^p_{cts}$ would follow provided that the conditions $\mb{(A1)}$ and $\mb{(A2)}$ hold true. It is clear that $\mb{(A2)}$ holds in the discrete case by definition, in the continuous case we use Remark \ref{rmk:u-vs-v}. On the other hand, the proof of the localisation condition $\mb{(A1)}$ is more involved. Let $-\infty\les a < b \les \infty$ and define local versions of the norms $\| \cdot \|_{\widetilde{U}^p_{dis}}$ and $\| \cdot \|_{\widetilde{U}^p_{cts}}$ by taking
   $$ \| u \|_{\widetilde{U}^p_{dis}(a,b)} = \sup_{ \substack{ v \in \mathfrak{S}_0 \\ \supp v \subset (a,b)\\ | v |_{V^p} \les 1}} |B(v, u)|$$
and
    $$ \| u \|_{\widetilde{U}^p_{cts}(a,b)} = \sup_{\substack{ v \in C^\infty_0\\ \supp v \subset (a,b) \\ |v|_{V^p} \les 1}} \Big| \int_\RR \lr{ \p_t v, u}_{L^2} dt \Big|.$$
We have reduced the proof of Theorem \ref{thm:characteristion of Up} to the following.

\begin{lemma}\label{lem:local to global}
Let $1<p<\infty$ and $\tau \in \mb{P}$. If $u \in L^\infty_t L^2_x$ is right continuous with $u(t) \to 0$ as $t\to - \infty$, and we define $u_\tau = \sum_{j=1}^N \ind_{[t_j, t_{j+1})} u(t_j)$ and $t_0 = -\infty$, $t_{N+1} = \infty$, then
    \begin{equation}\label{eqn:lem local to global:dis bound} \| u - u_\tau \|_{\widetilde{U}^p_{dis}} \les 4 \Big( \sum_{j=0}^{N} \| u \|_{\widetilde{U}^p_{dis}(t_j, t_{j+1})}^p \Big)^\frac{1}{p},
    \end{equation}
and
	\begin{equation}\label{eqn:lem local to global:cts bound} \| u - u_\tau \|_{\widetilde{U}^p_{cts}} \les 4 \Big( \sum_{j=0}^N \| u \|_{\widetilde{U}^p_{cts}(t_j, t_{j+1})}^p \Big)^\frac{1}{p}.
	\end{equation}
\end{lemma}
\begin{proof} We start by proving something similar in the $V^q$ case. Suppose that $w \in V^q$ and $\tau=(t_j)_{j=1}^N \in \mb{P}$. Let $t_0 = -\infty$ and $t_{N+1}=\infty$. For $j=0, \dots, N$ let $I_j \subset (t_j, t_{j+1})$ be  a left closed and right open interval, and define $w_j = \ind_{I_j} ( w - w(a_j))$ where $a_j \in (t_j, t_{j+1}]$. We claim that
	\begin{equation}\label{eqn:lem local to global:Vq bound}
		\sum_{j=0}^N |w_j|_{V^q}^q \les 2^{q+1}  |w|_{V^q}.
	\end{equation}
To prove the claim, we observe that for every $\epsilon>0$ there exists $(s^{(j)}_k)_{k=1}^{N_j} \in \mb{P}$ such that
$ |w_j|_{V^q} \les \epsilon + \sum_{k=1}^{N_j-1} \| w_j(s^{(j)}_{k+1}) - w_j(s^{(j)}_k)\|_{L^2}^q$. Without loss of generality, we may assume that $s^{(j)}_1=t_j$, $s^{(j)}_{N_j} = t_{j+1}$ for $1\les j\les N-1$, and $s^{(0)}_{N_0} = t_1$, $s^{(N)}_1 = t_N$. The definition of $w_j$ then gives
	\begin{align*}
		\sum_{k=1}^{N_j-1} \| &w_j(s^{(j)}_{k+1}) - w_j(s^{(j)}_k) \|_{L^2}^q \\
		&\les 2^{q-1}  \sum_{k=1}^{N_j-1} \big\|\big( \ind_{I_j}(s^{(j)}_{k+1}) - \ind_{I_j}(s^{(j)}_k) \big)\big(w(s^{(j)}_k) - w(a_j)\big) \big\|_{L^2}^q\\
		&\qquad + 2^{q-1} \sum_{k=1}^{N_j-1} \| w(s^{(j)}_{k+1}) - w(s^{(j)}_k) \|_{L^2}^q\\
		&\les 2^{q-1} \Big( \| w(s^{(j)}_{k,min}) - w(a_j) \|_{L^2}^q +\| w(s^{(j)}_{k, max}) - w(a_j) \|_{L^2}^q\Big)\\
		&\qquad + 2^{q-1}\sum_{k=1}^{N_j-1} \| w(s^{(j)}_{k+1}) - w(s^{(j)}_k) \|_{L^2}^q
	\end{align*}
for some $t_j \les s^{(j)}_{k, min} < s^{(j)}_{k, max}\les t_{j+1}$. Summing up over $j$, and choosing $\epsilon>0$ sufficiently small, we then obtain \eqref{eqn:lem local to global:Vq bound}.

We now turn the proof of \eqref{eqn:lem local to global:dis bound}. Let $\tau = (t_j)_{j=1}^N \in \mb{P}$ and $v\in \mathfrak{S}$, $\epsilon>0$, and for $j=1, \dots, N-1$ define
	$$v_j(t) = \ind_{[t_j+\epsilon, t_{j+1}-\epsilon)}(t)\big( v(t) - v(t_{j+1}-\epsilon)\big),$$
and
	$$ v_0(t) = \ind_{[-\epsilon^{-1}, t_1)}(t)\big( v(t) - v(t_{1}-\epsilon)\big), \qquad v_N(t) = \ind_{[t_N+\epsilon, \infty)}(t) \big( v(t) - v(\infty)\big)$$
where we let $v(\infty) = \lim_{t\to \infty} v(t)$, note that $v(t) = v(\infty)$ for all sufficiently large $t$ since $ v\in \mf{S}$. By construction and the bound \eqref{eqn:lem local to global:Vq bound}, we have $v_j \in \mathfrak{S}$, $\supp v_j \subset (t_j, t_{j+1})$, and  $ \sum_{j=0}^N |v_j|_{V^q}^q \les 2^{q+1} |v|_{V^q}^q$. Suppose that we can show that for all $\delta>0$, by choosing $\epsilon= \epsilon(u,v,\delta, \tau)>0$ sufficiently small we have for all $j=0, \dots, N$
	\begin{equation}\label{eqn:lem local to global:dis error bound}
		|B( \ind_{[t_j, t_{j+1})} v- v_j, u - u_\tau) | \les \delta.
	\end{equation}
Then, by linearity and definition of the bilinear pairing $B$, we have $B(v_j, u_\tau)=0$ and hence
	\begin{align*}
		 |B(v,u-u_\tau)| &\les  (N+1) \sup_j \big| B\big( \ind_{[t_j, t_{j+1})} v - v_j, u-u_\tau\big)\big|+ \sum_{j=0}^N | B(v_j, u)|   \\
		 	&\les (N+1) \delta + \sum_{j=0}^{N} |v_j|_{V^q} \| u \|_{\widetilde{U}^p_{dis}(t_j, t_{j+1})} \\
		 	&\les (N+1) \delta + 4 |v|_{V^q} \Big( \sum_{j=0}^N \| u \|_{\widetilde{U}^p_{dis}(t_j, t_{j+1})}^p \Big)^\frac{1}{p}.
	\end{align*}
Since this holds for every $\delta>0$, \eqref{eqn:lem local to global:dis bound} follows. Thus it remains to prove \eqref{eqn:lem local to global:dis error bound}. Since $v \in \mf{S}$ is a step function with $v(-\infty) = 0$, provided that we choose $\epsilon>0$ sufficiently small, a computation gives for $j=1, \dots, N$ the identities
	$$ \ind_{[t_j, t_{j+1})}(t) v(t) - v_j(t) = \ind_{[t_j, t_j + \epsilon)}(t) v(t_j) + \ind_{[t_j + \epsilon, t_{j+1})}(t) v(t_{j+1} - \epsilon), $$
and
	$$ \ind_{(-\infty, t_1)}(t) v(t) - v_0(t) = \ind_{[-\epsilon^{-1}, t_1 )}(t) v(t_1 - \epsilon).$$
Hence by definition we have for $j=1, \dots, N$
	$$ B( \ind_{[t_j, t_{j+1})} v- v_j, u - u_\tau) = \lr{v(t_{j+1}-\epsilon) - v(t_j), u(t_j+\epsilon)-u(t_j)},$$
and
	$$B(\ind_{(-\infty, t_1)}v - v_0, u-u_\tau) = \lr{v(t_1 -\epsilon), u(-\epsilon^{-1})}.$$
Therefore \eqref{eqn:lem local to global:dis error bound} follows from the right continuity of $u$ and the normalisation condition on $u$ at $-\infty$. This completes the proof of \eqref{eqn:lem local to global:dis bound}.

We now turn to the proof of \eqref{eqn:lem local to global:cts bound}. Let $\p_t \phi \in C^\infty_0$, $\epsilon>0$, and define
	$$ \phi_j(t) = \int_\RR \frac{1}{\epsilon} \rho\Big(\frac{s}{\epsilon}\Big) \ind_{I_j}(t-s)ds \big( \phi(t) - \phi(t_{j+1}) \big)$$
with $I_j = [t_j + \epsilon, t_{j+1} - \epsilon)$ for $j=1, \dots, N-1$, $I_0 = [-\epsilon^{-1}, t_1-\epsilon)$, $I_N = [t_N + \epsilon, \epsilon^{-1})$, and we take $\rho \in C^\infty_0(-1, 1)$ with $\int_\RR \rho =1$. Then clearly $\phi_j \in C^\infty_0$, $\supp \phi_j \subset (t_j, t_{j+1})$, and a short computation using the bound \eqref{eqn:lem local to global:Vq bound} gives $\sum_j |\phi_j|_{V^q}^q \les 2^{q+1} |\phi|_{V^q}^q$. Consequently, similar to the discrete case  above, it is enough to show that for every $\delta>0$ we can find an $\epsilon = \epsilon(\phi, u, \tau, \delta) >0$ such that
	\begin{equation}\label{eqn:lem local to global:cts error bound}
		\Big| \int_{t_j}^{t_{j+1}} \lr{ \p_t \phi - \p_t \phi_j, u-u_\tau} dt \Big| \les \delta.
	\end{equation}	
This is a consequence of the right continuity of $u$. More precisely, if $j=1, \dots, N-1$, after writing
	\begin{align*} \p_t \phi(t) - \p_t \phi_j(t) &= \Big(1-\int_\RR\frac{1}{ \epsilon} \rho\Big(\frac{s}{\epsilon}\Big) \ind_{I_j}(t-s)ds\Big) \p_t \phi(t) \\
	&\qquad - \frac{\phi(t) - \phi(t_{j+1})}{\epsilon} \int_\RR \frac{1}{ \epsilon} \p_s\rho\Big(\frac{s}{\epsilon}\Big) \ind_{I_j}(t-s)ds
	\end{align*}
we see that
	\begin{align*}
		\Big| \int_{t_j}^{t_{j+1}} &\lr{ \p_t \phi - \p_t \phi_j, u-u_\tau} dt \Big|\\
			&\lesa \epsilon \| \p_t \phi \|_{L^\infty_t L^2_x} \| u \|_{L^\infty_t L^2_x} + \epsilon^{-1} \| \phi \|_{L^\infty_t L^2_x} \int_{t_j}^{t_{j}+3\epsilon}    \| u(t) - u(t_j)\|_{L^2} dt \\
			&\qquad \qquad \qquad \qquad \qquad \qquad + \epsilon^{-1} \| u \|_{L^\infty_t L^2_x} \int_{t_{j+1}-3\epsilon}^{t_{j+1}}    \| \phi(t) - \phi(t_j)\|_{L^2}  dt
	\end{align*}
and hence \eqref{eqn:lem local to global:cts error bound} follows by the right continuity of $u$ provided we choose $\epsilon$ sufficiently small. A similar argument proves the cases $j=0$ and $j=N$.
\end{proof}

\section{Convolution, multiplication, and the adapted function spaces}\label{sec:bound}

In this section we record a number of key properties of the $U^p$  and $V^p$ spaces that have been used frequently throughout this article. Namely, in Subsection \ref{subsec:conv} we prove a basic convolution estimate together with the standard Besov embedding \[\dot{B}^\frac{1}{p}_{p,1} \subset V^p \subset U^p \subset \dot{B}^{\frac{1}{p}}_{p, \infty},\] up to normalisation at $t=-\infty$. This is well-known, see \cite{Peetre1976,Koch2005,Koch2014}. In Subsection \ref{subsec:alg} we prove an important high-low product type estimate in $U^p$ and $V^p$, which gives as a special case the crucial product estimates used in Section \ref{sec:multi}. Finally, in Subsection \ref{subsec:adapted}, we consider the adapted functions spaces $U^p_\Phi$ and $V^p_\Phi$.

\subsection{Convolution and the Besov Embedding}\label{subsec:conv}
We use the notation
	$$ f *_\RR g(t) = \int_\RR f(s) g(t-s) ds $$
to signify the convolution in the $t$ variable.

\begin{lemma}\label{lem:conv oper on Up}
Let $1<p<\infty$ and $\phi(t) \in L^1_t(\RR)$. For all $u \in U^p$ and $v \in V^p$ we have $\phi \ast_\RR u \in U^p$, $\phi *_\RR v \in V^p$, and the bounds
	$$ \Big\| \phi *_\RR u \Big\|_{U^p} \les 2\| \phi\|_{L^1_t} \| u \|_{U^p}, \qquad \Big\| \phi *_\RR v \Big\|_{V^p} \les \| \phi\|_{L^1_t} \| v \|_{V^p}.$$
\end{lemma}
\begin{proof}
We first observe that since $u$ and $v$ are right continuous and decay to zero as  $t \to - \infty$, the convolutions $\phi*_\RR u$ and $\phi *_\RR v$ also satisfy these conditions. The proof of the $V^p$ bound is immediate, and hence $\phi *_\RR v \in V^p$. To show that $\phi *_\RR u \in U^p$, we apply Theorem \ref{thm:characteristion of Up} and observe that for any $\psi \in C^\infty_0$, we have
	\begin{align*}
		\Big| \int_\RR \lr{ \p_t \psi(t), \phi *_\RR u(t) }_{L^2_x} dt \Big| &= \Big| \int_\RR \Big\langle \int_\RR \p_t \psi(t+s) \phi(t) dt , u(s) \Big\rangle_{L^2_x} ds \Big| \\
		&\les \Big\| \int_\RR \psi(t+s) \phi(t) dt \Big\|_{V^p} \| u \|_{U^p} \\
		&\les \| \psi \|_{V^q} \| \phi \|_{L^1_t} \| u \|_{U^p}.
	\end{align*}
\end{proof}

A similar argument shows that the space-time convolution with an $L^1_{t,x}(\RR^{1+n})$ kernel is also bounded on $U^p$ and $V^p$.

The spaces $V^p$ and $Up$ are closely related.  In fact, for functions which have temporal Fourier support in an annulus centered at the origin, the $U^p$ and $V^p$ norms are equivalent.

\begin{theorem}\label{thm:besov embedding}
Let $1<p<\infty$ and $d>0$. If $u \in L^p_t L^2_x$ with $\supp \mc{F}_t u \subset \{ \frac{d}{100} \les |\tau| \les 100d\}$ then $u \in U^p$ and
	$$  \| u \|_{V^p} \approx \| u \|_{U^p} \approx d^\frac{1}{p} \| u \|_{L^p_t L^2_x}.$$
Conversely, if $u \in U^p$ and $\supp \mc{F}_t u \subset \{ \frac{d}{100} \les |\tau| \les 100d\}$, then $u \in L^p_t L^2_x$.
\end{theorem}

\begin{remark}\label{rmk:besov embedding}
Theorem \ref{thm:besov embedding} can also be stated in terms of the Besov spaces $\dot{B}^\frac{1}{p}_{p, \infty}$ and $\dot{B}^\frac{1}{p}_{p, 1}$. More precisely, let $\rho \in C^\infty_0( \{2^{-1} < \tau < 2\})$ such that $ \sum_{d \in 2^\ZZ} \rho(\tfrac{\tau}{d}) = 1 $
for $\tau > 0$ and define $P^{(t)}_d = \rho( \frac{ |- i \p_t|}{d})$. If $1<p<\infty$ and $v \in V^p$ then
            $$ \sup_{d \in 2^\ZZ} d^\frac{1}{p} \| P^{(t)}_d v \|_{L^p_t L^2_x} \lesa  | v |_{V^p}.$$
On the other hand, if $ u(t) \to 0$ in $L^2_x$ as $t \to -\infty$ and
    $$\sum_{d \in 2^\ZZ} d^\frac{1}{p} \| P^{(t)}_d u \|_{L^p_t L^2_x} <\infty,$$
then $u \in U^p$ and
    $$ \| u \|_{U^p} \lesa \sum_{d \in 2^\ZZ} d^{\frac{1}{p}} \| P^{(t)}_d u \|_{L^p_t L^2_x}. $$
The first claim follows directly from Theorem \ref{thm:besov embedding} and the disposability of the multipliers $P^{(t)}_d$ which is a  consequence of Lemma \ref{lem:conv oper on Up}. To prove the second claim, in view of Theorem \ref{thm:besov embedding}, the sum converges in the Banach space $U^p$, and so we have $\sum_d P^{(t)}_d u \in U^p$. Since $u(t)$ and $\sum_d P^{(t)}_d u(t)$ can only differ by a polynomial, and both vanish at $-\infty$, we conclude that $u(t) = \sum_d P^{(t)}_d u(t) \in U^p$ and the claimed bound.
\end{remark}

The proof of Theorem \ref{thm:besov embedding}, as well as the proof of the product estimates contained in the next section, requires the following observation.

\begin{lemma}\label{lem:key ineq for stab lem}
For any $s \in \RR$, $(t_j)_{j=1}^N \in \mb{P}$, and $g \in V^p$ we have
        $$ \sum_{j=1}^{N} m_j^p \big\| g(t_j) -  g(t_j-s)\big\|_{L^2}^p \les 2 (1+|s|) | g |_{V^p}^p$$
where $m_j = \min\{ t_{j+1} - t_j, 1\}$.
\end{lemma}
\begin{proof}
  It is enough to consider $s>0$. Let \[J_k=\big\{j \in \{1,\ldots N\}: t_j \in [sk,(k+1)s)\big\}.\]
Then,
\[\sum_{j=1}^{N} m_j^p \|g(t_{j})-g(t_j-s)\|_{L^2_x}^p =\sum_{\substack{k \in \ZZ \\ J_k \not = \varnothing}}\sum_{j \in J_k } m_j^p \|g(t_{j})-g(t_j-s)\|_{L^2_x}^p\]
and, as $m_j\les 1$,
\begin{align*}
\sum_{j \in J_k } m_j^p \|g(t_{j})-g(t_j-s)\|_{L^2_x}^p &\les \Big(\sum_{j \in J_k } m_j \Big) \max_{j \in J_k}\|g(t_{j})-g(t_j-s)\|_{L^2_x}^p\\
    &\les (1+s) \|g(t_{j_k})-g(t_{j_k}-s)\|_{L^2_x}^p
\end{align*}
where $t_{j_k}\in J_k$ is chosen such that \[\|g(t_{j_k})-g(t_{j_k}-s)\|_{L^2_x}=\max_{j \in J_k}\|g(t_{j})-g(t_j-s)\|_{L^2_x}.\]
Now,
\begin{align*}
&\sum_{\substack{k \in \ZZ \\ J_k \not = \varnothing}}\|g(t_{j_k})-g(t_{j_k}-s)\|_{L^2_x}^p\\
\les{}& \sum_{\substack{ k \text{ even} \\ J_k \not = \varnothing}}\|g(t_{j_k})-g(t_{j_k}-s)\|_{L^2_x}^p
 +\sum_{\substack{ k \text{ odd} \\ J_k \not = \varnothing}}\|g(t_{j_k})-g(t_{j_k}-s)\|_{L^2_x}^p
\les{} 2|v|_{V^p}^p,
\end{align*}
because for $k=2m$ even we have
\[
t_{j_{2m}}<t_{j_{2(m+1)}}-s<t_{j_{2(m+1)}}\; ,
\]
hence the above points form a partition, and a  similar argument applies to $k=2m+1$ odd. In summary, we have
\[
\sum_{j=1}^{N-1} m_j^p \|g(t_{j})-g(t_j-s)\|_{L^2_x}^p \les 2 (1+s) |g|_{V^p}^p.
\]
\end{proof}

We now come to the proof of Theorem \ref{thm:besov embedding}.

\begin{proof}[Proof of Theorem \ref{thm:besov embedding}]
After rescaling, it is enough to consider the case $d=1$. Let $\rho \in C^\infty(\RR)$ with $\supp \widehat{\rho} \subset \{ |\tau| \les \frac{1}{100}\}$, $\| (1+|s|) \rho(s)\|_{L^1_s} \lesa 1$ and $\mc{F}_t (\rho) (0) = 1$. The temporal Fourier support assumption implies that
    $$u(t) = \int_\RR \rho(s) [u(t) - u(t-s)]ds$$
and hence an application of Lemma \ref{lem:key ineq for stab lem} gives
	\begin{align*}
		\| u \|_{L^p_t L^2_x} &\les \int_\RR |\rho(s)| \| u(t-s) - u(t) \|_{L^p_t L^2_x} ds \\
							&\les 2 \int_\RR |\rho(s)| \Big( \sum_{j\in \NN} \| u(t_j-s) - u(t_j)\|_{L^2_x}^p \Big)^\frac{1}{p} ds \\
							&\les 4 \| (1+|s|) \rho(s)\|_{L^1_s} \|u\|_{V^p} \lesa \|u\|_{V^p}
	\end{align*}
where we choose $t_j=t_j(s) \in [j, j+1)$ such that \[\sup_{t\in [j, j+1)} \|u(t-s) - u(t) \|_{ L^2_x} \les 2 \| u(t_j-s) - u(t_j)\|_{L^2_x}.\]

It remains to show that $ u \in U^p$ and the bound $\| u \|_{U^p} \lesa \| u \|_{L^p_t L^2_x}$. The assumptions on $u$ imply that $u$ is right continuous and $\| u(t)\|_{L^2_x} \to 0$ as $t \to -\infty$. Thus we apply Theorem \ref{thm:characteristion of Up} and observe that, using the $V^p$ case proved above, we have for $\phi \in C^\infty_0$
    \begin{align*}
      \Big| \int_\RR \lr{ \p_t \phi, u }_{L^2} dt \Big| &\les \| \ind_{\{ 100^{-1}\les \tau \les 100\}}(-i \p_t) \p_t \phi \|_{L^q_t L^2_x} \| u \|_{L^p_t L^2_x} \\
      &\lesa \| \ind_{\{ 100^{-1}\les \tau \les 100\}}(-i \p_t)  \phi \|_{L^q_t L^2_x} \| u \|_{L^p_t L^2_x} \lesa |\phi|_{V^q} \| u \|_{L^p_t L^2_x}
    \end{align*}
where $\frac{1}{p} + \frac{1}{q} =1 $. Hence $u\in U^p$ and the claimed bounds follow.
\end{proof}

\subsection{Stability under multiplication}\label{subsec:alg}

In the following, given $u, v: \RR \to L^2_x$, we let $\widehat{u}(t)$ denote the Fourier transform in $x$ (i.e. with respect to the $L^2$ variable), and $\mc{F}_t u(\tau)$ be the Fourier transform in the $t$ variable. Our goal is to find conditions under which the product $uv$ belongs to either $U^p$ or $V^p$. One possibility is the following.

\begin{lemma}\label{lem:stab}
Let $1\leq p<\infty$. Let $f,g:\RR \to L^2_x$ be bounded and satisfy $\supp(\mc{F}_tf)\subset (-1,1)$ and $\supp(\mc{F}_t g)\subset \RR \setminus (-4,4)$. If $g \in V^p$, then $fg \in V^p$ and
\[
\|fg\|_{V^p}\lesa \|f\|_{L^\infty_{t,x}}\|g\|_{V^p}.
\]
On the other hand, if $g \in U^p$, then $fg \in U^p$ and
\[
\|fg\|_{U^p}\lesa \|f\|_{L^\infty_{t,x}}\|g\|_{U^p}.
\]

\end{lemma}

Under the support assumptions of the previous lemma, it is clear that for $s>0$ we have the Sobolev product inequality $\| fg \|_{\dot{H}^s}\lesa \|f \|_{L^\infty} \| g\|_{\dot{H}^s}$. In particular, the previous lemma should be thought of as the $U^p$ and $V^p$ version of the standard heuristic that derivatives can essentially always be taken to fall on the high frequency term.

We now turn to the proof of Lemma \ref{lem:stab}.

\begin{proof}[Proof of Lemma \ref{lem:stab}]
 We start with the $V^p$ bound under the weaker assumption that we only have $\supp \mc{F}_t g \subset \RR \setminus (-3, 3)$.  Clearly, it is enough to prove the bound for $|fg|_{V^p}$. Let $\tau=(t_j)_{j=1}^N\in \mb{P}$ be any partition. Then
\begin{align*}
\Big(\sum_{j=1}^{N-1}\|fg(t_{j+1})-fg(t_j)\|_{L^2_x}^p\Big)^{\frac1p}&\les \Big(\sum_{j=1}^{N-1}\| f(t_{j+1})(g(t_{j+1})-g(t_j))\|_{L^2_x}^p\Big)^{\frac1p} \\
    &\qquad \qquad+\Big(\sum_{j=1}^{N-1}\|(f(t_{j+1})-f(t_j))g(t_j)\|_{L^2_x}^p\Big)^{\frac1p},
\end{align*}
and due to
\[
\Big(\sum_{j=1}^{N-1}\| f(t_{j+1})(g(t_{j+1})-g(t_j))\|_{L^2_x}^p\Big)^{\frac1p}\les \|f\|_{L^\infty_{t,x}}\|g\|_{V^p}
\]
it is enough to prove
\begin{equation}\label{eqn:lem stab:reduced prob}\Big(\sum_{j=1}^{N-1}\|(f(t_{j+1})-f(t_j))g(t_j)\|_{L^2_x}^p\Big)^{\frac1p} \lesa \|f\|_{L^\infty_{t,x}}\|g\|_{V^p}.\end{equation}
The hypothesis on the Fourier-supports of $f$ and $g$ implies that there exists $\rho\in \mc{S}(\RR)$ such that
\[
f=\rho\ast f \; \text{ and }\; \rho\ast (f(\cdot-b)g)=0 \text{ for all }b \in \RR.
\]
Consequently, we have the identity
\[(f(t_{j+1})-f(t_j))g(t_j)=\int_\RR \rho(s) \big(f(t_{j+1}-s)-f(t_j-s)\big)\big(g(t_{j})-g(t_j-s)\big)ds.\]
Now, since
    \begin{align*} \|f(a)-f(b)\|_{L^\infty_{x}}&=\Big\|\int_{\RR} \big(\rho(a-s)-\rho(b-s)\big)f(s)ds\Big\|_{L^\infty_{x}}\\
    &=\Big\|\int_a^b \int_{\RR} \rho'(t-s)f(s)ds dt\Big\|_{L^\infty_{x}}\les |a-b|\|\rho'\|_{L^1}\|f\|_{L^\infty_{t,x}},
    \end{align*}
if we let $m_j = \min\{ t_{j+1} - t_j, 1\}$, an application of Lemma \ref{lem:key ineq for stab lem} gives
    \begin{align*}
   \Big(\sum_{j=1}^{N-1}\|&(f(t_{j+1})-f(t_j))g(t_j)\|_{L^2_x}^p\Big)^{\frac1p}\\
                        &\lesa \| f \|_{L^\infty_{t,x}} \int_\RR |\rho(s)| \Big(\sum_{j=1}^{N-1} m_j^p \|g(t_j) - g(t_j -s)\|_{L^2_x}^p\Big)^{\frac1p} ds \\
                        &\lesa \| f \|_{L^\infty_{t,x}} \| g\|_{V^p}  \int_\RR |\rho(s)| (1+|s|)^{\frac{1}{p}} ds \\
                        &\lesa \| f \|_{L^\infty_{t,x}} \| g \|_{V^p}
    \end{align*}
and hence \eref{eqn:lem stab:reduced prob} follows.

To prove the $U^p$ version, we observe that since $g \in U^p$ and $f$ is smooth and bounded, the product $fg$ is right continuous and satisfies the normalising condition $(fg)(t) \to 0$ in $L^2_x$ as $t \to -\infty$. Consequently, by Theorem \ref{thm:characteristion of Up}, it suffices to show that
    \begin{equation}\label{eqn:lem stab:Up temp} \Big| \int_\RR \lr{ \p_t \phi, fg}_{L^2_x} dt \Big| \lesa \| \phi\|_{V^q} \|f \|_{L^\infty} \| g \|_{U^p}
   	\end{equation}
where $\p_t \phi \in C^\infty_0$ and $ \frac{1}{q} + \frac{1}{p} = 1$. Applying the Fourier support assumption, together with the $V^p$ case of the product estimate proved above, we see that after writing $\p_t \phi f = \p_t (\phi f) - \p_t^{-1} \p_t( \phi \p_t f)$,
	\begin{align*}
		\Big| \int_\RR \lr{\p_t \phi, fg} dt \Big| &= \Big| \int_\RR \lr{\p_t P^{(t)}_{\g 3} \phi, fg} dt \Big|\\
				&\les \|P^{(t)}_{\g 3} \phi \overline{f} \|_{V^q} \| g\|_{U^p} + \|P^{(t)}_{\g 3} \phi \p_t \overline{f} \|_{V^q} \| \p_t^{-1} g\|_{U^p} \\
				&\lesa \| \phi\|_{V^q} \big( \| f\|_{L^\infty} \| g \|_{U^p} + \| \p_t f \|_{L^\infty} \| \p_t^{-1} g \|_{U^p}\big)
	\end{align*}
where $P^{(t)}_{\g 3}$ is a temporal Fourier projection to the set $\{|\tau| \g 3\}$, and we used Theorem \ref{thm:dual pairing}. The required bound \eqref{eqn:lem stab:Up temp} then follows by the boundedness of convolution operators on $U^p$ and $V^p$ together with the Fourier support assumptions on $f$ and $g$.
\end{proof}

In applications to PDE, in particular to the wave maps equation, we require a more general version of the previous lemma which includes a spatial multiplier. To this end, for $\phi: \RR \to L^\infty_x L^2_y + L^\infty_y L^2_x$ and $u, v \in L^\infty_t L^2_x$, we define
	$$ \mc{T}_\phi[u,v](t,x) = \int_{\RR^n}  \phi(t,x-y,y) u(t,x-y) v(t,y) dy. $$
An application of Fubini and H\"older shows that $\mc{T}_\phi(t,x,y): L^2 \times L^2 \to L^2$  with the fixed time bound
	$$ \| \mc{T}_\phi[u,v] \|_{L^\infty_t L^2_x} \les \Big(\sup_{t\in \RR} \| \phi(t,x,y) \|_{L^\infty_x L^2_y + L^\infty_y L^2_x}\Big) \| u \|_{L^\infty_t L^2_x} \| v \|_{L^\infty_t L^2_x}.$$
Adapting the proof of Lemma \ref{lem:stab}, under certain temporal Fourier support conditions on $\phi$ and  functions $u,v \in V^p$ (or $U^p$), we can show that $\mc{T}_\phi[u,v] \in V^p$ (or $U^p$).

\begin{theorem}\label{thm:stab with conv}
  Let $1<p<\infty$. Let $\phi(t,x,y): \RR \to L^\infty_x L^2_y + L^\infty_y L^2_x$ continuous and bounded with $\supp \mc{F}_t \phi \subset (-1, 1)$. If $u,v \in V^p$ and $\supp \mc{F}_t u \subset \RR \setminus (-4, 4)$, then we have
  	$$ \| \mc{T}_\phi[u,v] \|_{V^p} \lesa  \Big(\sup_{t\in \RR} \| \phi(t,x,y) \|_{L^\infty_x L^2_y + L^\infty_y L^2_x}\Big) \| u \|_{V^p} \| v \|_{V^p}.$$
 Moreover, if in addition we have $u,v \in U^p$, then $\mc{T}_\phi[u,v] \in U^p$ and
 		$$ \| \mc{T}_\phi[u,v] \|_{U^p} \lesa  \Big(\sup_{t\in \RR} \| \phi(t,x,y) \|_{L^\infty_x L^2_y + L^\infty_y L^2_x}\Big) \| u \|_{U^p} \| v \|_{U^p}.$$
\end{theorem}
\begin{proof}
Let $S:=\sup_{t\in \RR} \| \phi(t,x,y) \|_{L^\infty_x L^2_y + L^\infty_y L^2_x}$.
We start with the $V^p$ case. As in the proof of Lemma \ref{lem:stab}, we can reduce to considering the case when the difference falls on $\phi$, thus our goal is to show that for $(t_j)_1^N \in \mb{P}$, we have
	\begin{equation}\label{eqn:thm stab with conv:reduced prob I}
    \begin{split}
		\sum_{j=1}^{N-1}\Big\| \int_{\RR^n} \big(\phi(t_{j+1},x-y,y)& - \phi(t_j, x-y, y) \big) u(t_j,x-y) v(t_j,y)dy \Big\|_{L^2_x(\RR^n)}^p\\
                &\lesa S^p \| u \|_{V^p}^p\| v \|_{V^p}^p.
    \end{split}
	\end{equation}
Following the argument used in Lemma \ref{lem:stab}, we see that the temporal Fourier support assumption implies the identity
	\begin{align*}
&\big(\phi(t_{j+1},z,y) - \phi(t_j, z, y) \big) u(t_j, z) \\
={}& \int_\RR \rho(s) \big(\phi(t_{j+1}-s,z,y) - \phi(t_j-s, z, y) \big) \big( u(t_j, z) - u(t_j-s,z) \big) ds\\
={}& \int_\RR \int_\RR \rho(s) \big( \rho(t_{j+1}-s-s') - \rho(t_j-s-s')\big) \phi(s',z,y) \big( u(t_j, z) - u(t_j-s,z) \big) ds\, ds'
        \end{align*}
        for some $\rho \in \mc{S}(\RR)$ with $\mc{F}_t \rho = 1$ on $\supp \mc{F}_t \phi$. Consequently we see that
	\begin{align*}
		&\Big\| \int_{\RR^n} \big(\phi(t_{j+1},x-y,y) -  \phi(t_j, x-y, y) \big) u(t_j,x-y) v(t_j,y)dy \Big\|_{L^2_x}\\
		    \les{}& S \int_\RR \int_\RR  |\rho(s)| |\rho(t_{j+1} -s') - \rho(t_j -s')| \| u(t_j)- u(t_j-s)\|_{L^2}\|v(t_j)\|_{L^2}ds\, ds'\\
			\lesa{}& S \int_\RR |\rho(s)|  m_j \| u(t_j) - u(t_j-s) \|_{L^2}ds\, \|v\|_{L^\infty_t L^2_x}
	\end{align*}
with $m_j = \min\{ t_{j+1} - t_j, 1\}$. Hence \eqref{eqn:thm stab with conv:reduced prob I} follows from Minkowski's inequality and Lemma \ref{lem:key ineq for stab lem}.

We now turn to the proof of the $U^p$ case. Since $u,v\in U^p$ and $\phi$ is continuous, $\mathcal{T}_\phi(u,v)$ is right continuous and converges to zero as $t\to -\infty$. Consequently, applying the Besov embedding in Remark \ref{rmk:besov embedding}, we have
\begin{align*}
\|P^{(t)}_{\les 4} \mathcal{T}_\phi(u,v)\|_{U^p} &\lesa \sum_{d \lesa 4} d^{\frac{1}{p}} \|P^{(t)}_{d} \mathcal{T}_\phi(u,v)\|_{U^p} \\
&\lesa{} \|\mathcal{T}_\phi(u,v)\|_{L^p_tL^2_x} \lesa{}S \|u\|_{L^p_t L^2_x}\|v\|_{L^\infty_t L^2_x}\lesa{}S \|u\|_{U^p}\|v\|_{U^p} .
\end{align*}
For the remaining term, we note that an application of Theorem \ref{thm:characteristion of Up}, reduces the problem to proving the bound
    $$\Big| \int_\RR \lr{ \p_t \psi, \mc{T}_\phi[u,v]}_{L^2} dt \Big| \lesa  S \| \psi \|_{V^{p'}} \| u \|_{U^p} \| v\|_{U^p} $$
where $\psi, \mc{F}_t \psi \in C^\infty$ and $\supp \mc{F}_t \psi \subset \RR \setminus (-4, 4)$. We write the left hand side as
    \begin{align} \Big| \int_\RR \lr{ &\p_t \psi,  \mc{T}_\phi[u,v]}_{L^2} dt \Big|\notag \\
     &= \Big| \int_\RR \big\langle \p_t  \psi(t,x), \phi(t,x-y,y) u(t,x-y) v(t,y) \big\rangle_{L^2_{x,y}} dt \Big| \notag \\
    &\les \Big| \int_\RR \big\langle   \psi(t,x), \p_t \phi(t,x-y,y) u(t,x-y) v(t,y) \big\rangle_{L^2_{x,y}} dt \Big|\notag \\
    &\qquad \qquad
    +\Big| \int_\RR \big\langle \p_t \big( \psi(t,x) \overline{\phi}(t,x-y,y)\big), u(t,x-y) v(t,y) \big\rangle_{L^2_{x,y}} dt \Big|
    \label{eqn:thm stab with conv:decomp}
    \end{align}
To bound the first term in \eref{eqn:thm stab with conv:decomp}, we use the temporal support assumption on $\phi$ to write $\p_t \phi(t) = \int_\RR \p_t \rho(s) \phi(t-s) ds$ with $\| \p_t \rho \|_{L^1_s} \lesa 1$, hence again applying the Besov embedding in Remark \ref{rmk:besov embedding} we obtain
    \begin{align*}
      \Big| \int_\RR \big\langle  &\psi(t,x), \p_t \phi(t,x-y,y) u(t,x-y) v(t,y) \big\rangle_{L^2_{x,y}} dt \Big|\\
                &\les \| \psi \|_{L^{p'}_t L^2_x} \int_\RR |\p_s \rho(s)| \Big\| \int_{\RR^n}  \phi(t-s,x-y,y) u(t,x-y) v(t,y) dy \Big\|_{L^p_t L^2_x} ds \\
                &\lesa S \|  \psi \|_{V^{p'}} \|u\|_{L^p_t L^2_x} \|v\|_{L^\infty_t L^2_x} \lesa  S \|  \psi \|_{V^{p'}} \|u\|_{U^p} \|v\|_{U^p}
    \end{align*}
where we used the temporal Fourier support assumptions on $\psi$ and $u$.

On the other hand, to bound the second term in \eref{eqn:thm stab with conv:decomp}, we first observe that it is enough to consider the case where $u$ and $v$ are $U^p$ atoms with partitions $\tau$ and $\tau'$. Let $(t_j)_{j=1}^N = \tau \cup \tau'$. Computing the integral in time, we deduce that
    \begin{align*}  &\int_\RR \big\langle \p_t \big( \psi(t,x) \overline{\phi}(t,x-y,y)\big), u(t,x-y) v(t,y) \big\rangle_{L^2_{x,y}} dt  \\
    &= \sum_{j=1}^{N-1}  \big\langle   \psi(t_{j+1},x) \overline{\phi}(t_{j+1},x-y,y) - \psi(t_{j},x) \overline{\phi}(t_{j},x-y,y) , u(t_j,x-y) v(t_j,y) \big\rangle_{L^2_{x,y}} \\
    &\qquad \qquad \qquad -  \big\langle  \psi(t_{N}) , \mc{T}_\phi[u, v](t_N) \big\rangle_{L^2_{x}}  \\
    &=  \sum_{j=1}^{N-1}  \big\langle   \psi(t_{j+1},x) \big(\overline{\phi}(t_{j+1},x-y,y) - \overline{\phi}(t_{j},x-y,y)\big) , u(t_j,x-y) v(t_j,y) \big\rangle_{L^2_{x,y}}  \\
    &\qquad \qquad \qquad + \sum_{j=1}^{N-1}  \big\langle   \psi(t_{j+1}) - \psi(t_{j}), \mc{T}_\phi[u, v](t_j) \big\rangle_{L^2_{x,y}} -   \big\langle  \psi(t_{N}) , \mc{T}_\phi[u, v](t_N) \big\rangle_{L^2_{x}}.
    \end{align*}
Applying H\"{o}lder's inequality and the fixed time convolution bound, we have
    \begin{align*}
      \sum_{j=1}^{N-1} \big| \big\langle &  \psi(t_{j+1}) - \psi(t_{j}), \mc{T}_\phi[u, v](t_j) \big\rangle_{L^2_{x,y}}\big| + \big|   \big\langle  \psi(t_{N}) , \mc{T}_\phi[u, v](t_N) \big\rangle_{L^2_{x}} \big|\\
            &\lesa S \Big( \sum_{j=1}^{N-1} \| \psi(t_{j+1}) - \psi(t_j) \|_{L^2}^{p'} \Big)^{\frac{1}{p'}} \Big( \sum_{j=1}^{N-1} \| u(t_j)\|_{L^2}^p \| v(t_j) \|_{L^2}^p\Big)^\frac{1}{p} \\
            &\qquad \qquad +  S \| \psi \|_{L^\infty_t L^2_x} \| u \|_{L^\infty_t L^2_x} \| v \|_{L^\infty_t L^2_x} \\
            &\lesa S  \| \psi \|_{V^{p'}}
    \end{align*}
where we used the fact that $u$ and $v$ are $U^p$ atoms and $(t_j)_{j=1}^{N} = \tau \cup \tau'$. Consequently, it only remains to prove that
    \begin{equation}\label{eqn:thm stab with conv:final Up bound}
        \begin{split}
      \sum_{j=1}^{N-1} \big| \big\langle   \psi(t_{j+1},x) \big(\overline{\phi}(t_{j+1},x-y,y) - \overline{\phi}(t_{j},x-y,y)\big)& , u(t_j,x-y) v(t_j,y) \big\rangle_{L^2_{x,y}}\big|\\
            &\lesa S \| \psi \|_{V^{p'}}.
        \end{split}
    \end{equation}
This follows by adapting the proof of the $V^p$ case above, namely, we first use the temporal Fourier support assumption to obtain the identity
    \begin{align*}
&\psi(t_{j+1}, x)\big(\overline{\phi}(t_{j+1},z,y) - \overline{\phi}(t_j, z, y) \big) \\
={}& \int_\RR \rho(s) \big(\overline{\phi}(t_{j+1}-s,z,y) - \overline{\phi}(t_j-s, z, y) \big) \big( \psi(t_{j+1}, x) - \psi(t_{j+1}-s, x) \big) ds\\
={}& \int_\RR \int_\RR \rho(s) \big( \rho(t_{j+1}-s') - \rho(t_j-s')\big) \overline{\phi}(s'-s,z,y) \big( \psi(t_{j+1}, x) - \psi(t_{j+1}-s, x)\big) ds\, ds'
        \end{align*}
and hence
    \begin{align*}
      \big| \big\langle  & \psi(t_{j+1},x) \big(\overline{\phi}(t_{j+1},x-y,y) - \overline{\phi}(t_{j},x-y,y)\big) , u(t_j,x-y) v(t_j,y) \big\rangle_{L^2_{x,y}}\big| \\
      &\lesa S \| u(t_j) \|_{L^2} \| v(t_j) \|_{L^2} \int_\RR |\rho(s)| m_j \| \psi(t_j) - \psi(t_j - s) \|_{L^2} ds.
    \end{align*}
Summing up, applying H\"older's inequality together with Lemma \ref{lem:key ineq for stab lem}, we then deduce \eref{eqn:thm stab with conv:final Up bound}.
\end{proof}

\subsection{Adapted $U^p$ and $V^p$}\label{subsec:adapted}

Given a phase $\Phi: \RR^n \to \RR$ (measurable and of moderate growth), we define the adapted function spaces $U^p_\Phi$ and $V^p_\Phi$ as
    $$ U^p_\Phi = \big\{ u \, \big| \,\,e^{i t \Phi(-i\nabla)} u(t) \in U^p\,\big\}, \qquad \qquad V^p_\Phi = \big\{ v \, \big| \, \, e^{ i t \Phi(-i\nabla)} v(t) \in V^p \big\},$$
as in \cite{Koch2005,Hadac2009}.
As in the case of $U^p$ and $V^p$, elements of $U^p_\Phi$ and $V^p_\Phi$ are right continuous, and approach zero as $t\to -\infty$. With the norms
\[
\|u\|_{U^p_\Phi}=\|t \mapsto e^{i t \Phi(-i\nabla)}u(t)\|_{U^p}, \text{ resp. }\|v\|_{V^p_\Phi}=\|t \mapsto e^{i t \Phi(-i\nabla)}v(t)\|_{V^p}
\]
these spaces $U^p_\Phi$ and $V^p_\Phi$ are Banach spaces.
They are constructed to contain perturbations of solutions to the linear PDE
        $$ -i \p_t u + \Phi(-i\nabla) u = 0.$$
In particular, for any $T \in \RR$ and $f \in L^2_x$, we have $\ind_{[T, \infty)}(t) e^{-i t \Phi(-i\nabla)} f \in U^p_\Phi$ (and $V^p_\Phi$). Note that the cutoff $\ind_{[T, \infty)}$ is essential here due to the normalisation at $t=-\infty$ (of course one could also choose a smooth cutoff $\rho(t) \in C^\infty$ instead).

The results obtained above for the $U^p$ and $V^p$ spaces can all be translated to the setting of the adapted function spaces $U^p_{\Phi}$ and $V^p_{\Phi}$. For instance, left and right limits always exist in $U^p$ and $V^p$, an application of Theorem \ref{thm:vu-emb} implies that for $p<q$ we have the embedding $V^p_{\Phi} \subset U^q_{\Phi}$, and Theorem \ref{thm:besov embedding} gives the  embedding
        \begin{equation}\label{eqn:adapted Up vs Xsb}
        		\| u \|_{U^p_{\Phi}} \approx \| u \|_{V^p_{\Phi}} \approx d^\frac{1}{p} \| u \|_{L^p_t L^2_x}
        \end{equation}
for all $u \in L^p_t L^2_x$ with $\supp \widetilde{u} \subset \{ (\tau, \xi) \in \RR^{1+n}|\, | \tau + \Phi(\xi)| \approx d\}$. In particular, we have $u \in U^p_{\Phi}$. Similarly, Theorem \ref{thm:dual pairing} and Theorem \ref{thm:characteristion of Up}, imply that
	$$ \| u \|_{U^p_{\Phi}} =  \sup_{ \substack{ -i\p_t v + \Phi(-i\nabla) v \in L^1_t L^2_x \\ |v|_{V^q_{\Phi}}\les 1}} \Big| \int_\RR \lr{ - i \p_t v + \Phi(-i\nabla)v, u}_{L^2} dt \Big| $$
and if $ -i \p_t \psi + \Phi(-i\nabla) \psi = F$, then
	$$ \| \ind_{[0, \infty)}(t) \psi \|_{U^p_{\Phi}} \approx \| \psi(0) \|_{L^2_x} + \sup_{ \substack{ v \in L^1_t L^2_x \\ |v|_{V^q_{\Phi}}\les 1}} \Big| \int_0^\infty \lr{ v, F}_{L^2} dt \Big| $$
in the sense that if the right-hand side is finite, and $u$ is right continuous with $u(t) \to 0$ as $t \to -\infty$, then $u, \psi \in U^p_\Phi$. Clearly, in view of Remark \ref{rem:compact support}, for reasonable phases $\Phi$ we can take the sup over $v\in C^\infty_0$ with $|v|_{V^q_\Phi} \les 1$ instead.  These properties show why the adapted function spaces are well adapted to studying the PDE $ - i \p_t \psi + \Phi(-i\nabla) \psi = F$. Further properties of the adapted function spaces are also known, see for instance the interpolation estimates in \cite{Hadac2009, Koch2005,  Koch2016}, and the vector valued transference type arguments in \cite{Candy2017b, Candy2016}.

We now give two applications of the product estimates in Subsection \ref{subsec:alg}. The first is an  application of Lemma \ref{lem:stab}  to remove the solution operators $e^{-it \Phi(-i\nabla)}$ in certain high modulation regimes.

\begin{proposition} Let $1<p<\infty$, $\Omega \subset \RR^n$ measurable and $\Phi: \Omega \to \RR$ such that $|\Phi(\xi)| \les 1$ for $\xi \in \Omega$ and assume that $\supp \widetilde{u} \subset \{ |\tau| \g 5\} \times \Omega$. If $u \in V^p$ then
		$$ \| u \|_{V^p_{\Phi}} \approx \| u \|_{V^p}. $$
Similarly, if $u \in U^p$ we have
		$$\| u \|_{U^p_{\Phi}} \approx  \| u \|_{U^p}. $$

\end{proposition}
\begin{proof}
In the $V^p$ case, an application of Plancherel shows that it suffices to prove the inequality
    $\| e^{  i t\Phi(\xi)} \widehat{u} \|_{V^p} \lesa \| \widehat{u} \|_{V^p}$.
The support assumption on $\widehat{u}$, implies that we may write $e^{it \Phi(\xi)} \widehat{u}(t,\xi) = e^{it\Phi(\xi)}\ind_{\Omega}(\xi) \widehat{u}(t,\xi)$. Since $|\Phi(\xi)|\les 1$ for $\xi \in \Omega$, we conclude that $\supp \mc{F}_t[ e^{it \Phi(\xi)} \ind_{\Omega}(\xi)] \subset \{ |\tau| \les 1\}$. Thus an application of  Lemma \ref{lem:stab} gives
    $$ \| e^{  i t\Phi(\xi)} \widehat{u} \|_{V^p} = \| e^{i t \Phi(\xi)} \ind_\Omega \widehat{u} \|_{V^p} \lesa \| e^{it\Phi(\xi)} \ind_\Omega \|_{L^\infty_{t,\xi}} \| \widehat{u} \|_{V^p} \lesa \| \widehat{u} \|_{V^p}$$
as required. An identical argument gives the $U^p$ case.
\end{proof}

The second is a reformulation of Theorem \ref{thm:stab with conv}. Define the spatial bilinear Fourier multiplier
        $$ \mc{M}[u,v](x) = \int_{\RR^n} \int_{\RR^n} m(\xi-\eta, \eta) \widehat{u}(\xi-\eta) \widehat{v}(\eta) e^{i x \cdot \xi} d\eta d \xi.$$

\begin{theorem}\label{thm:general adapted high mod prod}
Let $1<p<\infty$. Let $m:\RR^n \times \RR^n \to \CC$ and $\Phi_j$, $j=0, 1, 2$ be real-valued phases such that
        $$| \Phi_0(\xi+\eta) - \Phi_1(\xi) - \Phi_2(\eta)| \les 1$$
for all $(\xi, \eta) \in \supp m$. If $u \in V^p_{\Phi_1}$ with $\supp \widetilde{u} \subset \{ |\tau + \Phi_1(\xi)| \g 4 \}$ and $v \in V^p_{\Phi_2}$, then $M[u,v]\in V^p_{\Phi_0}$ and
        $$ \| \mc{M}[u,v] \|_{V^p_{\Phi_0}} \lesa \| m(\xi, \eta)\|_{L^\infty_\xi L^2_\eta + L^\infty_\eta L^2_\xi} \| u \|_{V^p_{\Phi_1}} \| v \|_{V^p_{\Phi_2}}. $$
If in addition $u \in U^p_{\Phi_1}$ and $v \in U^p_{\Phi_2}$, then $\mc{M}[u,v] \in U^p_{\Phi_0}$ and 
       $$ \| \mc{M}[u,v] \|_{U^p_{\Phi_0}} \lesa \| m(\xi, \eta)\|_{L^\infty_\xi L^2_\eta + L^\infty_\eta L^2_\xi} \| u \|_{U^p_{\Phi_1}} \| v \|_{U^p_{\Phi_2}}. $$
\end{theorem}
\begin{proof}
We begin by observing that $\| \psi \|_{V^p} = \| \widehat{\psi} \|_{V^p}$ and, via Theorem \ref{thm:characteristion of Up}, $\| \psi \|_{U^p} \approx \|\widehat{\psi} \|_{U^p}$. Let $u_1(t,\xi) = e^{ -i t \Phi_1(\xi)} \widehat{u}(t,\xi)$ and $v_2(t,\xi) = e^{-it\Phi_2(\xi)} \widehat{v}(t,\xi)$, then $\supp \mc{F}_t u_1 \subset \{ |\tau| \g 4\}$, and it suffices to prove that
        \begin{align*}
         \Big\| \int_{\RR^n} e^{ i t( \Phi_0(\xi) - \Phi_1(\xi-\eta) - \Phi_2(\eta))}&m(\xi-\eta, \eta) u_1(\xi-\eta) v_2(\eta) d\eta \Big\|_{V^p} \\
         & \lesa \| m \|_{L^\infty_\xi L^2_\eta + L^\infty_\eta L^2_\xi} \| u_1 \|_{V^p} \| v_2 \|_{V^p}.
        \end{align*}
But this is a consequence of Theorem \ref{thm:stab with conv} after noting that
    $$\int_{\RR^n} e^{ i t( \Phi_0(\xi) - \Phi_1(\xi-\eta) - \Phi_2(\eta))}m(\xi-\eta, \eta) u_1(\xi-\eta) v_2(\eta) d\eta = \mc{T}_\phi[u_1, v_2]$$
with $\phi(t,\xi, \eta) = e^{ i t( \Phi_0(\xi+\eta) - \Phi_1(\xi) - \Phi_2(\eta))}m(\xi, \eta)$, and $\supp \mc{F}_t \phi \subset \{ |\tau| \les 1\}$.  The proof of the $U^p$ bound follows from an identical application of Theorem \ref{thm:stab with conv}.
\end{proof}

Typically the multiplier $m(\xi, \eta) = \ind_{\Omega_0}(\xi + \eta) \ind_{\Omega_1}(\xi) \ind_{\Omega_2}(\eta)$ is a cutoff to some frequency region, in which case we have
		$$ \| m \|_{L^\infty_\xi L^2_\eta + L^\infty_\eta L^2_\xi} \lesa (\min\{ |\Omega_0|, |\Omega_1|, |\Omega_2|\}\big)^\frac{1}{2}.$$
Clearly more involved examples are also possible.

\section{The bilinear restriction estimate in adapted function spaces}\label{sec:adapted bilinear restriction}

 Let $\lambda_1\g 1$ and define
    $$ \Lambda_1= \{ |\xi - e_1| < \tfrac{1}{100} \}, \qquad \qquad \Lambda_2= \{ |\xi \mp \lambda e_2| < \tfrac{1}{100} \lambda \} $$
with $e_1 = (1, 0, \dots, 0)$ and $e_2 = (0, 1, 0, \dots, 0)$. In this section, we give the proof of the following theorem.

\begin{theorem}\label{thm:L2 bilinear restriction}
Let $\frac{1}{n+1} < \frac{1}{b}\les \frac{1}{a} \les \frac{1}{2}$, $\frac{1}{a} + \frac{1}{b} \g \frac{1}{2}$ and $\lambda \g 1$. Assume that $u \in U^a_+$ and $v \in U^b_{\pm}$ with $ \supp \widehat{u} \subset \Lambda_1$ and $\supp \widehat{v} \subset \Lambda_2$. Then
    $$ \| u v\|_{L^2_{t,x}} \lesa \lambda^{(n+1)(\frac{1}{2} - \frac{1}{a})} \| u \|_{U^a_+} \| v \|_{U^b_{\pm}}. $$
\end{theorem}

This is the case $\alpha\approx 1$ of Theorem \ref{thm:bilinear restriction}. In fact, the general case can essentially be reduced to this case.

\begin{remark}[Proof of the general case of Theorem \ref{thm:bilinear restriction}]\label{rmk:general-alpha}
To obtain the $L^2_{t,x}$ estimate stated in Theorem \ref{thm:bilinear restriction}, we need a small angle version of Theorem \ref{thm:L2 bilinear restriction}. More precisely, we need to show that if $\kappa, \kappa' \in \mc{C}_\alpha$ with $\angle(\kappa, \pm \kappa') \approx \alpha$, then provided that $\supp \widehat{u} \subset \{ |\xi| \approx 1, \frac{\xi}{|\xi|} \in \kappa \} $ and $\supp \widehat{v} \subset \{ |\xi| \approx \lambda, \frac{\xi}{|\xi|}  \in \kappa' \}$ we have
        $$ \| u v \|_{L^2_{t,x}} \lesa \alpha^{\frac{n-3}{2}} \lambda^{(n-1)(\frac{1}{2} - \frac{1}{a})} \| u \|_{U^a_+} \| v \|_{U^b_\pm}.$$
As in  \cite[Section 2.3]{Candy2017b}, after rotating and rescaling, this bound is equivalent to showing that
        $$ \| u v \|_{L^2_{t,x}} \lesa \lambda^{(n-1)(\frac{1}{2} - \frac{1}{a})} \| u \|_{U^a_\Phi} \| v \|_{U^b_{\pm \Phi}}$$
where the phase becomes $\Phi(\xi) = \alpha^{-2} ( \xi_1^2 + \alpha^2 |\xi'|)^\frac{1}{2}$, and the functions $u$ and $v$ now have Fourier support in the rectangles $\{ \xi_1 \approx 1, \xi_2 \approx 1, |\xi''| \ll 1\}$ and $\{ \pm \xi_1 \approx \lambda, |\xi'| \ll \lambda \}$. It is easy to check that the phase $\Phi$ behaves essentially the same as $|\xi|$. In particular, an analogue of the wave table construction Tao, Theorem \ref{thm:wave tables}, holds with $|\nabla|$ replaced with $\Phi$, which together with the argument given below, gives the small angle case of the bilinear $L^2_{t,x}$ estimate. See \cite{Candy2017b} for the details.
\end{remark}

\begin{remark}[The general bilinear restriction estimate]
Theorem \ref{thm:L2 bilinear restriction} is a special case of a  bilinear restriction type estimate for general phases. More precisely, suppose that $1\les q, r \les 2$ and $a,b \g 2$ with
	$$\frac{1}{q}  + \frac{n+1}{2r}< \frac{n+1}{2}, \qquad \frac{2}{(n+1)q} < \frac{1}{b} \les \frac{1}{a} \les \frac{1}{2}, \qquad \frac{1}{\min\{q,r\}} \les \frac{1}{a} + \frac{1}{b}.  $$
Then provided that the phases $\Phi_j$ satisfy suitable curvature and transversality properties, we have
     \begin{equation}\label{eqn:gen bilinear restriction} 
        \|  uv \|_{L^q_t L^r_x} 
        \les \mb{C} \| u \|_{U^a_{\Phi_1}} \| v \|_{U^b_{\Phi_2}}
     \end{equation}
where the constant is given by 
    $$ \mb{C} \approx \mu^{n+1-\frac{n+1}{r} - \frac{2}{q}} \mc{V}_{max}^{\frac{1}{r}-1} \mc{H}_1^{1-\frac{1}{q} - \frac{1}{r}} \Big( \frac{\mc{H}_1}{\mc{H}_2} \Big)^{\frac{1}{q} - \frac{1}{2} + (n+1)(\frac{1}{2} - \frac{1}{a})} \Big( \frac{\mc{V}_{max}}{\mu \mc{H}_1}\Big)^{(1-\frac{1}{r}-\frac{1}{b})_+}$$
with $\mu = \min\{ \text{diam}(\widehat{u}), \text{diam}(\widehat{v})\}$, $\mc{H}_j = \|\nabla^2 \Phi_j\|_{L^\infty}$, $\mc{V}_{max} = \sup |\nabla \Phi_1(\xi) - \nabla \Phi_2(\eta)|$, and $\mc{H}_2 \les \mc{H}_1$. This estimate is essentially sharp, see \cite[Theorem 1.7]{Candy2017b} for a more precise statement. Note that Theorem \ref{thm:bilinear restriction} is a special case of the bilinear restriction estimate \eqref{eqn:gen bilinear restriction}, together with a rescaling argument, see the proof of Theorem 1.10 in \cite{Candy2017b}. 
\end{remark}

In the remainder of this section we give the proof Theorem \ref{thm:L2 bilinear restriction}, assuming only Tao's wave table construction for free waves \cite[Proposition 15.1]{Tao2001b}. More precisely, in Subsection \ref{subsec:wave table}, we introduce some additional notation, and state the wave table construction of Tao \cite{Tao2001b}, as well as an averaging of cubes lemma. In Subsection \ref{subsec:atomic wave tables}, following closely the arguments in \cite{Candy2017b}, we show how the wave table construction for free waves implies a key localised bilinear estimate for atoms. Finally in Subsection \ref{subsec:induction on scales}, we run the induction on scales argument, and complete the proof of Theorem \ref{thm:L2 bilinear restriction}.

\subsection{Notation and the wave table construction}\label{subsec:wave table}
We use same notation as in \cite{Candy2017b}.
Given a set $\Omega \subset \RR^n$, a vector $h \in \RR^n$, and a scalar $c>0$,  we let $\Omega + h = \{ x + h \mid x \in \Omega \}$ denote the translation of $\Omega$ by $h$, and define $\diam(\Omega)= \sup_{x, y \in \Omega} |x-y|$ and $\Omega + c = \{ x + y\mid x\in \Omega, \,|y|<c\}$ to  be the Minkowski sum of $\Omega$ and the ball $\{|x|< c\}$.

Let $R \g 1$ and $0<\epsilon\les 1$. The constant $R$ denotes the large space-time scale and $0<\epsilon\les 1$ is a small fixed parameter used to control the various error terms that arise. All cubes in this article are oriented parallel to the coordinate axis. Let $Q$ be a cube side length $R$, and take a subscale $0<r\les R$. We define $\mc{Q}_r(Q)$ to be a collection of disjoint subcubes of width $ 2^{-j_0} R$ which form a cover of $Q$, where $j_0$ is the unique integer such that $2^{-1-j_0} R< r \les 2^{-j_0} R$ (in other words we divide $Q$ up into smaller cubes of equal diameter $2^{-j_0} R$). Thus all cubes in $Q_{r}(Q)$ have side lengths $\approx r$, and moreover, if $r \les r'\les R$ and $q \in \mc{Q}_r(Q)$, $q' \in \mc{Q}_{r'}(Q)$ with $q\cap q' \not = \varnothing$, then $q \in \mc{Q}_{r}(q')$. To estimate various error terms which arise, we need to create some separation between cubes. To this end, following Tao, we introduce the following construction. Given $0<\epsilon \ll 1$ and a subscale $0<r\les R$, we let
        $$ I^{\epsilon, r}(Q) = \bigcup_{q \in \mc{Q}_r(Q)} (1-\epsilon) q. $$
Note that we have the crucial property, that if $q \in \mc{Q}_r(Q)$, and $(t,x) \not \in I^{\epsilon, r}(Q) \cap q$, then $ \dist\big( (t,x), q\big) \g \epsilon r$. Given sequences $(\epsilon_m)$ and $(r_m)$ with $\epsilon_m>0$ and $0<r_m \les R$, we define
        $$ X[Q] = \bigcap_{m=1}^M I^{\epsilon_m, r_m}. $$
Thus cubes inside $X[Q]$ are separated at multiple scales. An averaging argument allows us to move from a cube $Q_R$, to the set $X[Q]$ but with a larger cube $Q$.

\begin{lemma}[{\cite[Lemma 6.1]{Tao2001b}}]\label{lem:cube averaging}
Let $R>0$, $ \epsilon = \sum_{j=1}^M \epsilon_j \les 2^{-(n+2)}$, and $r_m \les R$. For every cube $Q_R$ of diameter $R$, there exists a cube $Q \subset 4Q_R$ of diameter $2R$ such that for every $F \in L^2_{t,x}(Q_R)$ we have
        $$ \| F \|_{L^2_{t,x}(Q_R)} \les ( 1+ 2^{n+2}\epsilon) \| F \|_{L^2_{t,x}(X[Q])}. $$
\end{lemma}

All functions in the following are vector valued, thus maps $u:\RR^{1+n} \to \ell^2_c(\ZZ)$ where $\ell^2_c(\ZZ)$ is the set of complex valued sequences with finitely non-zero components. The vector valued nature of the waves plays a key role in the induction on scales argument. A function $u \in L^\infty_t L^2_x$ is a $\pm$-\emph{wave}, if $u = e^{\mp it |\nabla|} f$ for some (vector valued) $f \in L^2_x$. An \emph{atomic $\pm$-wave} is a function $v \in L^\infty_t L^2_x$ such that $ v = \sum_I \ind_I v_I$ with the intervals $I$ forming a partition of $\RR$, and  each $v_I$ is a $\pm$-wave (i.e. we take $e^{\pm i t|\nabla|} v \in \mathfrak{S}$). Given an atomic $\pm$-wave $v = \sum_I \ind_I v_I $, we let
        $$ \| v \|_{\ell^a L^2} = \Big( \sum_I \| v_I \|_{L^2}^a \Big)^\frac{1}{a}.$$\\

For a subset $\Omega \subset \RR^{1+n}$ we let $\ind_{\Omega}$ denote the indicator function of $\Omega$. Let $\mc{E}$ be a finite collection of subsets of $\RR^{1+n}$, and suppose we have a collection of ($\ell^2_c$-valued) functions $(u^{(E)})_{E\in \mc{E}}$. We then define the associated \emph{quilt}
            $$ [u^{(\cdot)}](t,x) = \sum_{E \in \mc{E}} \ind_E(t,x) |u^{(E)}(t,x)|.$$
This notation was introduced by Tao \cite{Tao2001b}, and plays a key technical role in localising the product $uv$ into smaller scales.\\

Following the argument in \cite{Candy2017b}, which adapts the proof of Tao to the setting of atomic waves, our goal is to show that Theorem \ref{thm:L2 bilinear restriction} is a consequence of the following \emph{wave table} construction of Tao.

\begin{theorem}[Wave Tables {\cite[Proposition 15.1]{Tao2001b}}]\label{thm:wave tables}
Fix $n\g 2$. There exists a constant $C>0$, such that, for all $0<\epsilon \ll 1$, $\lambda \g 1$, $R \g 100 \lambda$, all cubes $Q$ of diameter $R$, and $+$-waves $F$, and $\pm$-waves $G$ with
    $$ \supp \widehat{F} \subset \Lambda_1 + \frac{1}{100}, \qquad \supp \widehat{G} \subset \Lambda_2 + \frac{1}{100},$$
there exists for each $B \in \mathcal{Q}_{\frac{R}{4}}(Q)$, a $+$-wave $\mc{W}^{(B)}_{1,\epsilon} = \mc{W}_{1,\epsilon}^{(B)}[F; G,Q]$, and a $\pm$-wave $\mc{W}^{(B)}_{2,\epsilon} = \mc{W}^{(B)}_{2,\epsilon}[G; F,Q]$ satisfying
    $$ F = \sum_{B \in \mc{Q}_{\frac{R}{4}}(Q)} \mc{W}_{1,\epsilon}^{(B)}, \qquad G = \sum_{B \in \mc{Q}_{\frac{R}{4}}(Q)} \mc{W}_{2, \epsilon}^{(B)},$$
the Fourier support condition
    $$ \supp \widehat{\mc{W}}^{(B)}_{1,\epsilon} \subset \supp \widehat{F} + R^{-\frac{1}{2}} , \qquad \supp \widehat{\mc{W}}^{(B)}_{2,\epsilon} \subset \supp \widehat{G} + (\lambda R)^{-\frac{1}{2}}, $$
the energy estimates
    \begin{align*} \Big( \sum_{B \in \mc{Q}_{\frac{R}{4}}(Q)} \| \mc{W}^{(B)}_{1,\epsilon} \|_{L^\infty_t L^2_x}^2 \Big)^\frac{1}{2} &\les (1 + C \epsilon) \| F \|_{L^\infty_t L^2_x},\\
     \Big( \sum_{B \in \mc{Q}_{\frac{R}{4}}(Q)} \| \mc{W}_{2,\epsilon}^{(B)} \|_{L^\infty_t L^2_x}^2 \Big)^\frac{1}{2} &\les (1 + C \epsilon) \| G \|_{L^\infty_t L^2_x}, 
    \end{align*}
and the bilinear estimates
    $$ \big\| \big( |F| - [\mc{W}^{(\cdot)}_{1,\epsilon}]\big) G \big\|_{L^2_{t,x}(I^{\epsilon, \frac{R}{4}}(Q))} \les C \epsilon^{-C} R^{-\frac{n-1}{4}} \| F \|_{L^\infty_t L^2_x} \| G \|_{L^\infty_t L^2_x},$$
and
    $$\big\| F \big( |G| - [\mc{W}^{(\cdot)}_{2,\epsilon}]\big)  \big\|_{L^2_{t,x}(I^{\epsilon, \frac{R}{4}}(Q))} \les C \epsilon^{-C} \Big( \frac{R}{\lambda} \Big)^{-\frac{n-1}{4}} \| F \|_{L^\infty_t L^2_x} \| G \|_{L^\infty_t L^2_x}. $$
\end{theorem}

\begin{remark} The notation used in Theorem \ref{thm:wave tables} differs somewhat from that used in \cite{Tao2001b}. For a proof in the general phase case, see \cite[Theorem 9.3]{Candy2017b}. It is important to note that Theorem \ref{thm:wave tables} is purely a statement about free waves, and does not involve any atomic structure.
\end{remark}

\begin{remark}
  The construction of the wave table $\mc{W}_{1,\epsilon}$ relies on a wave packet decomposition of $u$. Roughly speaking, we take
    $$ \mc{W}^{(B)}_{1,\epsilon} = \sum_{T} c_T F_T$$
  where $F_T$ is a wave packet concentrated in the tube $T$, and the (real-valued) coefficients $c_T = c_T(G,B)$ are given by
    $$ c_T \approx \Big( \frac{\| G \|_{L^2_{t,x}(B \cap T)}}{ \| G \|_{L^2_{t,x}(Q \cap T)}}\Big)^2.$$
  Thus $\mc{W}^{(B)}_{1,\epsilon}$ contains the wave packets $F_T$ of $F$, such that $G|_T$ is concentrated on the smaller cube $B$. Since $\mc{W}^{(B)}_{1,\epsilon}$ is a sum of wave packets of $F$, the support conclusion and fact that it is a $+$-wave essentially follow directly. Similarly, ignoring the constant and using the (almost) orthogonality of the wave packet decomposition, we have
    \begin{align*}  \sum_{B} \| \mc{W}^{(B)}_{1,\epsilon} \|_{L^\infty_t L^2_x}^2
            \lesa \sum_T \sum_B \| F_T \|_{L^\infty_t L^2_x}^2  \frac{\| G \|_{L^2_{t,x}(B \cap T)}^2}{ \| G \|_{L^2_{t,x}(Q \cap T)}^2}
            &\lesa \sum_T \| F_T \|_{L^\infty_t L^2_x}^2  \lesa \| F \|_{L^\infty_t L^2_x}^2.
    \end{align*}
  Improving the constant here to $1+C \epsilon$ requires an improved wave packet decomposition introduced by Tao. Finally, the bilinear estimate exploits the fact that $|u| - [\mc{W}^{(\cdot)}_{1,\epsilon}]$ on the cube $B$, only contains wave packets such that $G|_T$ is \emph{not} concentrated on $B$. This is essentially a non-pigeon holed version of the argument of Wolff \cite{Wolff2001}.
\end{remark}

\subsection{From wave tables to a bilinear estimate for atoms}\label{subsec:atomic wave tables}

In this section, we give the proof of the following consequence of Theorem \ref{thm:wave tables}.

\begin{theorem}\label{thm-main bilinear estimate Up case}
Let $\frac{1}{n+1}<\frac{1}{b}\les \frac{1}{a} \les \frac{1}{2}$, $0<\epsilon \ll 1$, and $Q_R$ be a cube of diameter $R \g 100 \lambda$. Then for any atomic $+$-wave $u = \sum_{I \in \mc{I}} \ind_I u_I$, and any atomic  $\pm$-wave $v = \sum_{J \in \mc{J}} \ind_J v_J$ with
		$$\supp \widehat{u} \subset \Lambda_1 + \frac{1}{100}, \qquad \supp \widehat{v} \subset \Lambda_2 + \frac{1}{100}$$
there exist a cube $Q$ of diameter $2R$ such that for each $I \in \mc{I}$ and $J \in \mc{J}$ we have a decomposition
	$$ u_I = \sum_{B \in \mc{Q}_{\frac{2R}{4^M}}(Q)} u_I^{(B)}, \qquad v_J = \sum_{B' \in \mc{Q}_{\frac{R}{2}}(Q)} v_J^{(B')}$$
where $M\in \NN$ with $4^{M-1} \les \lambda < 4^M $, and $u^{(B)}= \sum_{I \in \mc{I}} \ind_I u_I^{(B)}$ is an atomic $+$-wave, $v^{(B')}=\sum_{J \in \mc{J}} \ind_J v_J^{(B')}$ is an atomic $\pm$-wave, with the support properties
	$$ \supp \widehat{u}^{(B)} \subset \supp \widehat{u} + 2\Big( \frac{2 R}{\lambda} \Big)^{-\frac{1}{2}}, \qquad \supp \widehat{v}^{(B')} \subset \supp \widehat{v} +  2 \Big( \frac{2R}{\lambda} \Big)^{-\frac{1}{2}}.$$
Moreover, for any $a_0, b_0 \g 2$ we have the energy bounds
	$$ \Big(\sum_{B \in \mc{Q}_{\frac{2R}{4^M}}(Q)} \|u^{(B)}\|_{\ell^{a_0} L^2_x}^{a_0} \Big)^\frac{1}{a_0} \les (1 +C\epsilon) \| u \|_{\ell^{a_0} L^2_x}$$
	 $$\Big(\sum_{B' \in \mc{Q}_{\frac{R}{2}}(Q)} \|v^{(B')}\|_{\ell^{b_0} L^2_x}^{b_0} \Big)^\frac{1}{b_0} \les (1 + C \epsilon) \| v \|_{\ell^{b_0} L^2_x}$$
and the bilinear estimate
	\begin{align*}  \| u v \|_{L^2_{t,x}(Q_R)} \les (1+ &C\epsilon)  \big\| \big[u^{(\cdot)}\big] \big[v^{(\cdot)}\big] \big\|_{L^2_{t,x}(Q)} \\
	 &+ C \epsilon^{-C} \lambda^{(n+1)(\frac{1}{2} - \frac{1}{a})} \Big( \frac{R}{\lambda} \Big)^{\frac{n+1}{2}(\frac{1}{2} - \frac{1}{b})-\frac{n-1}{4}} \|u \|_{\ell^a L^2_x} \| v \|_{\ell^b L^2_x}
    \end{align*}
where the constant $C$ depends only on the dimension $n$, and the exponents $a,b$.
\end{theorem}
\begin{proof}
Let $C_0$ denote the constant appearing in Theorem \ref{thm:wave tables}. An application of Lemma \ref{lem:cube averaging} implies that there exists a cube $Q$ of radius $2R$ such that
        $$ \| uv \|_{L^q_t L^r_x(Q_R)} \les ( 1  + C \epsilon) \| uv \|_{L^q_t L^r_x(X[Q])}$$
where we take
	$$ X[Q] = \bigcap_{m=1,\dots, M} I^{\epsilon_m, 4^{-m}2R}(Q), \qquad \epsilon_m = 4^{\delta(m-M)} \epsilon$$
and $\delta>0$ is some fixed constant to be chosen later (which will depend only on the dimension $n$, the exponent $b$, and the constant $C_0$). Let $V=(v_J)_{J\in \mc{J}}$, perhaps after relabeling, we have  $V: \RR^{1+n} \to \ell^2_c(\ZZ)$. In particular,  $V$ is a $\pm$-wave such that $|v|\les |V|$ and $\| V \|_{L^\infty_t L^2_x} = \| v \|_{\ell^2 L^2}$. We now repeatedly apply the wave table construction in Theorem \ref{thm:wave tables} with $G=V$. More precisely, given $B_1 \in \mc{Q}_{\frac{R}{2}}(Q)$ we let
	$$ u^{(B_1)}_{I, 1} = \mc{W}^{(B_1)}_{1, \epsilon_1}(u_I; V, Q)$$
and assuming we have constructed $u_{I, m}^{(B_{m})}$ with $B_{m} \in \mc{Q}_{\frac{2R}{4^{m}}}(Q)$, we define for $B_{m+1} \in \mc{Q}_{\frac{2R}{4^{m+1}}}(B_m)$
		$$ u_{I, m+1}^{(B_{m+1})} = \mc{W}^{(B_{m+1})}_{1, \epsilon_{m+1}}\big( u^{(B_{m})}_{I, m}; V, B_{m}\big)$$
(thus we apply Theorem \ref{thm:wave tables} with $F=u^{(B_{m})}_{I, m}$, $G=V$, $\epsilon= \epsilon_m$, and $Q=B_m$). To extend this to the atomic waves, we simply take $u_m^{(B_m)} = \sum_{I \in \mc{I}} \ind_I(t) u_{I, m}^{(B_m)}$.
Finally, for $B \in \mc{Q}_{\frac{2R}{4^M}}(Q)$, we let $ u^{(B)}=u_M^{(B)}$. Clearly $u^{(B)}$ is again an atomic $+$-wave, and from an application of Theorem \ref{thm:wave tables} the Fourier supports satisfy
    \begin{align*}
     \supp u^{(B)} &\subset \supp u^{(B)}_{M-1} + \Big( \frac{2R}{4^{M-1}} \Big)^{-\frac{1}{2}} \\
     &\subset \supp \widehat{u} + \sum_{m=1}^{M-1} \Big( \frac{2R}{4^{m-1}}\Big)^{-\frac{1}{2} } \subset \supp \widehat{u} + 2 \Big( \frac{ 2R}{\lambda} \Big)^{-\frac{1}{2}}.
    \end{align*}
On the other hand, the energy inequality follows by exchanging the order of summation, using the fact that $a\g 2$, and repeatedly applying the energy estimate in Theorem \ref{thm:wave tables}
	\begin{align*}
  \bigg( \sum_{B \in \mc{Q}_{\frac{2R}{4^M}}(Q)} &\| u^{(B)} \|_{\ell^a L^2_x}^a \bigg)^\frac{1}{a}\\
            &\les \bigg( \sum_{I \in \mc{I}} \bigg( \sum_{B \in \mc{Q}_{\frac{2R}{4^M}}(Q)} \big\| u^{(B)}_{I} \big\|_{L^\infty_t L^2_x}^2 \bigg)^\frac{a}{2}\bigg)^{\frac{1}{a}} \\
            &\les (1 + C_0 \epsilon_{M}) \bigg( \sum_{I \in \mc{I}} \bigg( \sum_{B_{M-1} \in \mc{Q}_{\frac{2R}{4^{M-1}}}(Q)} \big\| u^{(B_{M-1})}_{I, M-1} \big\|_{L^\infty_t L^2_x}^2 \bigg)^\frac{a}{2}\bigg)^{\frac{1}{a}} \\
            &\les \Pi_{m=1}^M ( 1 + C_0 \epsilon_m) \bigg( \sum_{I \in \mc{I}} \| u_{I} \|_{L^\infty_t L^2_x}^a\bigg)^{\frac{1}{a}} \les ( 1 + C \epsilon) \| u \|_{\ell^a L^2}
\end{align*}
where $C$ depends only on $\delta$ and $C_0$. The next step is to decompose $v = \sum_{J \in \mc{J}} \ind_J(t) v_J$. Let $U = ( u_I^{(B)})_{I \in \mc{I}, B \in \mc{Q}_{\frac{2R}{4^M}}}$. Then again, perhaps after relabeling, $U:\RR^{1+n} \to \ell^2_c(\ZZ)$, and hence $U$ is a $+$-wave with the pointwise bound $[u^{(\cdot)}] \les |U|$ and the energy bound $\| U \|_{L^\infty_t L^2_x} \les (1 + C \epsilon) \| u \|_{\ell^2 L^2_x}$.
We now decompose each $v_J$ relative to $U$ and the cube $Q$, in other words we apply Theorem \ref{thm:wave tables} and take for every $B' \in \mc{Q}_{\frac{2R}{4}}(Q) $
		$$ v^{(B')}_J = \mc{W}^{(B')}_{2, \epsilon}\big( v_J; U, Q\big)$$
and finally define $v^{(B')} = \sum_{J \in \mc{J}} \ind_J(t) v_J^{(B')}$. It is clear that $v^{(B')}$ is an atomic $\pm$-wave and that $v^{(B')}$ satisfies the correct Fourier support conditions. Furthermore, by a similar argument to the $u^{(B)}$ case, the required energy inequality also holds.

We now turn to the proof of the bilinear estimate. After observing that
        \begin{align*} 
        \| uv \|_{L^2_{t,x}(X[Q])} \les \big\| [u^{(\cdot)}] &[v^{(\cdot)}] \big\|_{L^2_{t,x}(X[Q])} \\
        		&+ \big\| \big( |u| - [u^{(\cdot)}]\big) v \big\|_{L^2_{t,x}(X[Q])} + \big\| [u^{(\cdot)}] \big( |v| - [v^{(\cdot)}]\big) \big\|_{L^2_{t,x}(X[Q])}, 
        \end{align*}
an application of H\"older's inequality implies that it is enough to show that for $\frac{1}{n+1} < \frac{1}{b} \les \frac{1}{2}$, and $ \frac{1}{b} \les \frac{1}{a} \les \frac{1}{2}$  we have
    \begin{equation}\label{eqn-thm main bilinear Up est-temp bilinear est}
    		\begin{split}
        \big\| \big( |u| - [u^{(\cdot)}]\big) v \big\|_{L^2_{t,x}(X[Q])} +& \big\| [u^{(\cdot)}] \big( |v| - [v^{(\cdot)}]\big) \big\|_{L^2_{t_x}(X[Q])}\\
                &\lesa \epsilon^{-C}\Big( \frac{R}{\lambda} \Big)^{\frac{1}{2} - \frac{n+1}{2b}} \lambda^{(n+1)(\frac{1}{2} - \frac{1}{a})} \| u \|_{\ell^a L^2} \| v \|_{\ell^b L^2}.
         \end{split}
    \end{equation}
We start by estimating the first term. The point is to interpolate between the ``bilinear'' $L^2_{t,x}$ estimate given in Theorem \ref{thm:wave tables} which decays in $R$, and  a ``linear'' $L^2_{t,x}$ estimate which can lose powers of $R$, but gains in the summability of the intervals $I$ and $J$. We first observe that by construction, Theorem \ref{thm:wave tables} implies that
	\begin{align}
		&\Big\| \Big( \big[ u^{(\cdot)}_{m-1}\big]- \big[ u^{(\cdot)}_{m} \big] \Big) v \Big\|_{L^2_{t,x}(X[Q])}^2   \notag \\
		&\les \sum_{I \in \mc{I}} \sum_{B_{m-1} \in \mc{Q}_{\frac{2R}{4^{m-1}}}(Q)} \Big\| \Big( |u^{(B_{m-1})}_{I, m-1}| - \big[ \mc{W}^{(\cdot)}_{1, \epsilon_m}(u^{(B_{m-1})}_{I, m-1}; V, Q)\big] \Big) V \Big\|_{L^2_{t,x}(I^{\epsilon_m, \frac{2R}{4^{m}}}(B_{m-1}))}^2\notag \\
		&\les C_0^2 \epsilon_m^{-2C_0} \Big( \frac{4^{m-1}}{2R}\Big)^{\frac{n-1}{2}}  \sum_{I \in \mc{I}}\sum_{B_{m-1} \in \mc{Q}_{\frac{2R}{4^{m-1}}}(Q)} \| u^{(B_{m-1})}_{I, m-1} \|_{L^\infty_t L^2_x}^2 \| V \|_{L^\infty_t L^2_x}^2  \notag \\
		&\lesa \epsilon^{-2C_0}  4^{-2(M-m)(\frac{n-1}{4} - \delta C_0)} \Big( \frac{R}{\lambda} \Big)^{-\frac{n-1}{2}} \| u \|_{\ell^2 L^2}^2 \| v \|_{\ell^2 L^2}^2 \label{eqn-thm main bilinear Up-bilinear L2 estimate}
	\end{align}
where we used the definition of $\epsilon_m$ and $M$. On the other hand, to obtain the linear $L^2_{t,x}$ bound, we start by noting that for any $1\les m \les M$, and $a_0\g2$ we have
    \begin{align}
    \bigg( \sum_{B \in \mc{Q}_{\frac{2R}{4^m}}(Q) } \big\| u^{(B)}_{m} \big\|_{L^\infty_t L^2_x}^2 \bigg)^\frac{1}{2}
                &\lesa   \bigg( \sum_{B \in \mc{Q}_{\frac{2R}{4^m}}(Q) } \big\| u^{(B)}_{m} \big\|_{\ell^{a_0} L^2_x}^2 \bigg)^\frac{1}{2} \notag \\
                &\lesa   4^{m (\frac{n+1}{2} - \frac{n+1}{a_0})} \bigg( \sum_{B \in \mc{Q}_{\frac{2R}{4^m}}(Q) } \big\| u^{(B)}_{m} \big\|_{\ell^{a_0} L^2_x}^{a_0} \bigg)^\frac{1}{a_0}  \notag \\
                &\lesa \lambda^{ (n+1)(\frac{1}{2} - \frac{1}{a_0})}  \| u \|_{\ell^{a_0} L^2_x} \label{eqn-thm main bilinear U p-L2 linear bound for u}
    \end{align}
where we applied the energy inequality for $u_m^{(B)}$. Therefore, an application of H\"older's inequality gives for any $a_0 \g 2$
	\begin{align}
		\Big\| \Big( \big[ u^{(\cdot)}_{m-1}\big]& -\big[ u^{(\cdot)}_{m} \big]\Big) v \Big\|_{L^2_{t,x}(X[Q])}\notag\\
			&\lesa \Big( \frac{R}{4^m}\Big)^\frac{1}{2}  \Bigg(\sum_{B_{m-1}} \| u^{(B_{m-1})}_{m-1} v\|_{L^\infty_t L^2_x (B_{m-1})}^2 + \sum_{B_m} \| u^{(B_m)}_m v \|_{L^\infty_t L^2_x(B_m)}^2 \Bigg)^\frac{1}{2}   \notag \\
            &\lesa  \Bigg(\sum_{B_{m-1}} \| u^{(B_{m-1})}_{m-1}\|_{L^\infty_t L^2_x}^2 + \sum_{B_m} \| u^{(B_m)}_m \|_{L^\infty_t L^2_x}^2 \Bigg)^\frac{1}{2} \| v\|_{L^\infty_t L^2_x}   \notag \\
            &\lesa\Big( \frac{R}{\lambda} \big)^{\frac{1}{2}} \lambda^{ (n+1)(\frac{1}{2} - \frac{1}{a_0})}  4^{(M-m) \frac{1}{2}} \| u \|_{\ell^{a_0} L^2} \|v \|_{\ell^{\infty} L^2} \label{eqn-thm main bilinear Up-linear L2}
    \end{align}
where we used the fact that $v^{(B')}$ has Fourier support contained in a set of diameter $1$. Interpolating between (\ref{eqn-thm main bilinear Up-bilinear L2 estimate}) and (\ref{eqn-thm main bilinear Up-linear L2}) then gives for any  $\frac{1}{b}\les \frac{1}{a} \les \frac{1}{2}$,
	\begin{align*} \Big\| \Big( \big[ u^{(\cdot)}_{m-1}\big] &- \big[ u^{(\cdot)}_{m} \big]\Big) v \Big\|_{L^2_{t,x}(X[Q])}\\
			&\lesa \epsilon^{-C}  4^{-(M-m)\delta^*} \Big( \frac{R}{\lambda} \Big)^{\frac{1}{2} - \frac{n+1}{2b}} \lambda^{(n+1)(\frac{1}{2} - \frac{1}{a})} \| u \|_{\ell^a L^2_x} \| v \|_{\ell^b L^2_x}
	\end{align*}
where $\delta^*  = \frac{n+1}{2b} - \frac{1}{2} - 2 \delta C_0 \frac{1}{b}$. Consequently, provided that $\frac{1}{n+1}<\frac{1}{b} \les 1 - \frac{1}{r}$, and we choose $\delta$ sufficiently small depending only on $C_0$, $b$, and $n$, we have $\delta^*>0$. Thus by telescoping the sum over $m$ and letting $ u_0^{(Q)} = u$, we deduce that
	\begin{align*} \big\| \big( |u| - [u^{(\cdot)}]\big) v \big\|_{L^2_{t,x}(X[Q])} &\les \sum_{m=1}^M \big\| \big( [ u^{(\cdot)}_{m-1}] - [ u^{(\cdot)}_{m} ] \big) v \big\|_{L^2_{t,x}(X[Q])} \\
	&\lesa \epsilon^{-C}\Big( \frac{R}{\lambda} \Big)^{\frac{1}{2} - \frac{n+1}{2b}}  \lambda^{(n+1)(\frac{1}{2} - \frac{1}{a})} \| u \|_{\ell^a L^2_x} \| v \|_{\ell^b L^2_x}.
	\end{align*}
It only remains to estimate the second term on the left hand side of \eref{eqn-thm main bilinear Up est-temp bilinear est}. To this end, applying the definition of $v^{(B')}$ together with Theorem \ref{thm:wave tables}, we have
		\begin{align} \big\| [u^{(\cdot)}] \big( |v| - [v^{(\cdot)}] \big)\big\|_{L^2_{t,x} (X[Q])}^2 &\les\sum_{J \in \mc{J}} \big\| U \big( |v_J| - [v^{(\cdot)}_J] \big)\big\|_{L^2_{t,x} (I^{\epsilon_1, \frac{R}{2}}(Q))}^2  \notag \\
		&\lesa  \epsilon^{-2C_n}  \Big( \frac{R}{\lambda} \Big)^{-\frac{n-1}{2}} \| u \|_{\ell^2 L^2}^2 \| v \|_{\ell^2 L^2}^2. \label{eqn:thm main bilinear Up:bilinear v}
        \end{align}
On the other hand an application of H\"older's inequality together with the energy estimates and (\ref{eqn-thm main bilinear U p-L2 linear bound for u}) gives for any $a_0 \g 2$
	\begin{align}  \big\| [u^{(\cdot)}] &\big( |v| - [v^{(\cdot)}] \big)\big\|_{L^2_{t,x} (X[Q])}\notag \\
	&\lesa  \bigg( \sum_{B \in \mc{Q}_{\frac{2R}{4^M}}} \big\| u^{(B)} \|_{L^2_{t,x}(B)}^2 \bigg)^\frac{1}{2} \sup_{J \in \mc{J}} \Big( \|v_J \|_{L^\infty_t L^2_x} + \| [v^{(\cdot)}_J] \|_{L^\infty_t L^2_x}\Big) \notag \\
	 &\lesa \Big( \frac{R}{\lambda} \Big)^\frac{1}{2} \lambda^{ (n+1)(\frac{1}{2} - \frac{1}{a_0})}  \| u \|_{\ell^{a_0} L^2} \| v \|_{\ell^{\infty} L^2}.
	 \label{eqn-thm main bilinear Up est-v linear L2 bound}
	 \end{align}
Therefore, interpolating between \eref{eqn:thm main bilinear Up:bilinear v} and \eref{eqn-thm main bilinear Up est-v linear L2 bound}, we deduce that for $\frac{1}{b}\les \frac{1}{a} \les \frac{1}{2}$, we have
	$$ \big\| [u^{(\cdot)}] \big( |v| - [v^{(\cdot)}] \big)\big\|_{L^2_{t,x}(X[Q])} \lesa \epsilon^{-C} \Big( \frac{R}{\lambda} \Big)^{\frac{1}{2} - \frac{n+1}{2b}} \lambda^{(n+1)(\frac{1}{2} - \frac{1}{a})}   \| u \|_{\ell^a L^2} \| v \|_{\ell^b L^2}$$
and consequently (\ref{eqn-thm main bilinear Up est-temp bilinear est}) follows.
\end{proof}

\subsection{The Induction on Scales Argument}\label{subsec:induction on scales}

Here we apply Theorem \ref{thm-main bilinear estimate Up case} and give the proof of Theorem \ref{thm:L2 bilinear restriction}. We start with the following definition.

\begin{definition}
Given $R>0$, we let $A(R)>0$ denote the best constant such that for all cubes $Q$ of diameter $R\g 100 \lambda$, and all atomic $+$-waves $u$, and atomic $\pm$-waves $v$ such that
    $$\supp \widehat{u} \subset \Lambda_1 + 4\Big( \frac{R}{\lambda}\Big)^{-\frac{1}{2}}, \qquad \qquad \supp \widehat{v} \subset \Lambda_2 + 4\Big( \frac{R}{\lambda}\Big)^{-\frac{1}{2}}$$
we have
    $$ \| u v\|_{L^2_{t,x}(Q)} \les A(R) \| u \|_{\ell^a L^2_x} \| v \|_{\ell^b L^2}. $$
\end{definition}

It is clear that $A(R) \lesa R^\frac{1}{2}$. Our goal is to show that in fact we have $A(R) \lesa \lambda^{(n+1)(\frac{1}{2}- \frac{1}{a})}$ for all $R \g 100 \lambda$. This is a consequence of an induction on scales argument, using the following bounds.

\begin{proposition}[Induction Bounds]\label{prop:induction bounds}
 There exists $C>0$ such that for all $R \g 100 \lambda$ and $0<\epsilon < \frac{1}{100}$ we have
    \begin{equation}\label{eqn:init step}
        A(R) \les C \lambda^{(n+1)(\frac{1}{2}- \frac{1}{a})} \Big( \frac{R}{\lambda} \Big)^{10}
    \end{equation}
and
    \begin{equation}\label{eqn:induc step}
        A(2R) \les ( 1 + C \epsilon) A(R) + C \epsilon^{-C} \lambda^{(n+1)(\frac{1}{2}- \frac{1}{a})} \Big( \frac{R}{\lambda } \Big)^{(\frac{1}{2} - \frac{1}{b}) \frac{n+1}{2} - \frac{n-1}{4} }.
    \end{equation}
\end{proposition}
\begin{proof}
  Let $Q$ be a cube of diameter $2R$, and let $u = \sum_{I \in \mc{I}} \ind_I(t) u_I$ be an atomic $+$-wave, and $v = \sum_{J \in \mc{J}} \ind_J(t) v_J$ be an atomic $\pm$-wave satisfying the support conditions
	$$ \supp \widehat{u} \subset \Lambda_1 + 4\Big( \frac{2R}{\lambda}\Big)^{-\frac{1}{2}}, \qquad \supp \widehat{v}  \subset \Lambda_2 + 4\Big( \frac{2R}{\lambda}\Big)^{-\frac{1}{2}} .$$
An application of Theorem \ref{thm-main bilinear estimate Up case} gives a cube $Q'$ of diameter $4R$, and atomic waves $(u^{(B)})_{B \in \mc{Q}_{\frac{R}{4^{M-1}}}(Q)}$, $(v^{(B')})_{B'\in \mc{Q}_{R}(Q)}$ such that
	\begin{equation}\label{eqn-prop induction bound Up version-bounded by quilt}
        \begin{split} \| u v\|_{L^2_{t,x}(Q)} \les (1 + &C\epsilon) \big\| [u^{(\cdot)}] [v^{(\cdot)}] \big\|_{L^2_{t,x}(Q')} \\
        &+  C \epsilon^{-C} \lambda^{(n+1)(\frac{1}{2} - \frac{1}{a})} \Big( \frac{R}{\lambda} \Big)^{\frac{n+1}{2}(\frac{1}{2} - \frac{1}{b})-\frac{n-1}{4} } \|u \|_{\ell^a L^2} \| v \|_{\ell^b L^2}
        \end{split}
\end{equation}
and the support properties
	$$ \supp \widehat{u}^{(B)} \subset \supp \widehat{u} + 2 \Big( \frac{4R}{\lambda}\Big)^{-\frac{1}{2}} \subset \Lambda_1 + 4 \Big( \frac{R}{\lambda} \Big)^{-\frac{1}{2}}$$
and similarly $\supp \widehat{v} \subset \Lambda_2 + 4 (\frac{R}{\lambda})^{-\frac{1}{2}}$.

To prove \eref{eqn:induc step}, we let $B' \in \mc{Q}_{R}(Q')$ and define the atomic $+$-wave $U^{(B')}= \sum_{I\in \mc{I}} \ind_I(t) U^{(B')}_I$ with $U^{(B')}_I = ( u^{(B)}_I )_{B \in \mc{Q}_{\frac{R}{4^{M-1}}}(B')}$. Then for every $B'\in \mc{Q}_{R}(Q)$ we have an atomic $+$-wave $U^{(B')}$ and an atomic $\pm$-wave $v^{(B')}$ satisfying the correct support assumptions to apply the definition of $A(R)$. Thus
	\begin{align*}
	   \big\| [ u^{(\cdot)} ] [ v^{(\cdot)}] \big\|_{L^2_{t,x}(Q')} &\les \bigg( \sum_{B' \in \mc{Q}_R(Q')} \| U^{(B')} v^{(B')} \|_{L^2_{t,x}}^2 \bigg)^{\frac{1}{2}} \\
	   &\les A(R) \Big( \sum_{B' \in \mc{Q}_R(Q')} \|U^{(B')}\|_{\ell^a L^2_x}^{a} \Big)^\frac{1}{a}\Big( \sum_{B' \in \mc{Q}_R(Q')} \|v^{(B')}\|_{\ell^b L^2_x}^{b} \Big)^\frac{1}{b} \\
	   &\les (1+C \epsilon) A(R)
	\end{align*}
where the second line used the assumption $\frac{1}{a} + \frac{1}{b} \g  \frac{1}{2}$ and the last applied the energy inequalities in Theorem \ref{thm-main bilinear estimate Up case}. Therefore the induction bound \eref{eqn:induc step} follows from an application of (\ref{eqn-prop induction bound Up version-bounded by quilt}).

We now turn to the proof of \eref{eqn:init step}. We begin by observing that again using the bound (\ref{eqn-prop induction bound Up version-bounded by quilt}) with $\epsilon\approx 1$ it is enough to prove that for every $\frac{1}{b} \les \frac{1}{a} \les \frac{1}{2}$  we have the quilt bound
     \begin{equation}\label{eqn-prop induction bound Up version-initial quilt bound}\big\| [u^{(\cdot)}] [v^{(\cdot)}] \big\|_{L^2_{t,x}(Q')} \lesa  \lambda^{(n+1)(\frac{1}{2} - \frac{1}{\lambda})} \Big( \frac{R}{\lambda}\Big)^{\frac{1}{2} - \frac{1}{b}} \| u \|_{\ell^a L^2} \| v \|_{\ell^b L^2}.\end{equation}
But this follows by observing that since $[u^{(\cdot)}] = \sum_B \ind_B |u^{(B)}|$ is localised to cubes of diameter $\frac{R}{\lambda}$, an application of H\"older's inequality together with the energy estimates implies that
    \begin{align*}
      \big\| [u^{(\cdot)}] [v^{(\cdot)}] \big\|_{L^2_{t,x}(Q')} &\lesa \lambda^{(n+1)(\frac{1}{2}- \frac{1}{a})} \Big( \frac{R}{\lambda} \Big)^\frac{1}{2} \bigg( \sum_{B} \| u^{(B)} \|_{L^\infty_t L^2_x}^a \bigg)^\frac{1}{a} \sup_{B'} \|v^{(B')}\|_{L^\infty_t L^2_x} \\
      &\lesa \lambda^{(n+1)(\frac{1}{2}- \frac{1}{a})} \Big( \frac{R}{\lambda} \Big)^\frac{1}{2} \| u \|_{\ell^a L^2_x} \|v \|_{\ell^\infty L^2_x}.
    \end{align*}
\end{proof}

We now come to the proof of Theorem \ref{thm:L2 bilinear restriction}.

\begin{proof}[Proof of Theorem \ref{thm:L2 bilinear restriction}]
Let $C$ denote the constant in Proposition \ref{prop:induction bounds}. Let $R= 2^k 100 \lambda$ and  $\epsilon_k = 2^{-\delta k}$ with $0<\delta < \frac{1}{C} [\frac{n-1}{4} - (\frac{1}{2} - \frac{1}{b}) \frac{n+1}{2}]$. Then an application of \eref{eqn:induc step} gives
    $$ A( 2^{k+1} 100 \lambda) \les ( 1 + C 2^{-k\delta}) A(2^k 100 \lambda) +  C \lambda^{(n+1)(\frac{1}{2}- \frac{1}{a})} 2^{ [(\frac{1}{2} - \frac{1}{b}) \frac{n+1}{2} - \frac{n-1}{4} + \delta C] k }. $$
Since both exponents decay in $k$, after $k$ applications, we deduce that
    $$ A(2^{k+1} 100 \lambda) \lesa A( 100 \lambda) +  \lambda^{(n+1)(\frac{1}{2}- \frac{1}{a})}  \lesa \lambda^{(n+1)(\frac{1}{2}- \frac{1}{a})}$$
where we used the initial induction bound \eref{eqn:init step}. Hence Theorem \ref{thm:L2 bilinear restriction} follows.
\end{proof}

\section*{Acknowledgement}
The authors thank Kenji Nakanishi and Daniel Tataru for
sharing their part in the story of the division problem with us.
Also, the authors thank Daniel Tataru for providing a preliminary version of \cite{Tataru2001}.

Financial support by the
  DFG through the CRC 1283 ``Taming uncertainty and profiting from
  randomness and low regularity in analysis, stochastics and their
  applications'' is acknowledged.
\bibliographystyle{amsplain}
\bibliography{wave_maps}

\providecommand{\bysame}{\leavevmode\hbox to3em{\hrulefill}\thinspace}
\providecommand{\MR}{\relax\ifhmode\unskip\space\fi MR }
\providecommand{\MRhref}[2]{%
  \href{http://www.ams.org/mathscinet-getitem?mr=#1}{#2}
}
\providecommand{\href}[2]{#2}
\begin{thebibliography}{10}

\bibitem{Bahouri2011}
Hajer Bahouri, Jean-Yves Chemin, and Rapha\"el Danchin, \emph{Fourier analysis
  and nonlinear partial differential equations}, Grundlehren der Mathematischen
  Wissenschaften [Fundamental Principles of Mathematical Sciences], vol. 343,
  Springer, Heidelberg, 2011. \MR{2768550}

\bibitem{Bejenaru2017}
Ioan Bejenaru, \emph{Optimal {B}ilinear {R}estriction {E}stimates for {G}eneral
  {H}ypersurfaces and the {R}ole of the {S}hape {O}perator}, Int. Math. Res.
  Not. IMRN (2017), no.~23, 7109----7147. \MR{3801419}

\bibitem{Bejenaru2015}
Ioan Bejenaru and Sebastian Herr, \emph{The cubic {D}irac equation: small
  initial data in {$H^1(\Bbb{R}^3)$}}, Comm. Math. Phys. \textbf{335} (2015),
  no.~1, 43--82. \MR{3314499}

\bibitem{Bejenaru2016}
\bysame, \emph{The cubic {D}irac equation: small initial data in {$H^{\frac
  12}(\Bbb R^2)$}}, Comm. Math. Phys. \textbf{343} (2016), no.~2, 515--562.
  \MR{3477346}

\bibitem{Bejenaru2011}
Ioan Bejenaru, Alexandru~D. Ionescu, Carlos~E. Kenig, and Daniel Tataru,
  \emph{Global {S}chr\"odinger maps in dimensions {$d\geq 2$}: small data in
  the critical {S}obolev spaces}, Ann. of Math. (2) \textbf{173} (2011), no.~3,
  1443--1506. \MR{2800718}

\bibitem{Bournaveas2016}
Nikolaos Bournaveas and Timothy Candy, \emph{Global well-posedness for the
  massless cubic {D}irac equation}, Int. Math. Res. Not. IMRN (2016), no.~22,
  6735--6828. \MR{3632067}

\bibitem{Candy2017b}
Timothy Candy, \emph{Multi-scale bilinear restriction estimates for general
  phases}, preprint: \texttt{arXiv:1707.08944[math.CA]} (2017).

\bibitem{Candy2016}
Timothy Candy and Sebastian Herr, \emph{Transference of bilinear restriction
  estimates to quadratic variation norms and the {D}irac-{K}lein-{G}ordon
  system}, preprint: \texttt{arXiv:1605.04882[math.AP]} (2016).

\bibitem{Geba2017}
Dan-Andrei Geba and Manoussos~G. Grillakis, \emph{An introduction to the theory
  of wave maps and related geometric problems}, World Scientific Publishing Co.
  Pte. Ltd., Hackensack, NJ, 2017. \MR{3585834}

\bibitem{Hadac2009}
Martin Hadac, Sebastian Herr, and Herbert Koch, \emph{Well-posedness and
  scattering for the {KP}-{II} equation in a critical space}, Ann. Inst. H.
  Poincar\'e Anal. Non Lin\'eaire \textbf{26} (2009), no.~3, 917--941.
  \MR{2526409 (2010d:35301)}

\bibitem{Klainerman1980}
Sergiu Klainerman, \emph{Global existence for nonlinear wave equations}, Comm.
  Pure Appl. Math. \textbf{33} (1980), no.~1, 43--101. \MR{544044}

\bibitem{Klainerman2001}
Sergiu Klainerman and Igor Rodnianski, \emph{On the global regularity of wave
  maps in the critical {S}obolev norm}, Internat. Math. Res. Notices (2001),
  no.~13, 655--677. \MR{1843256}

\bibitem{Klainerman2002}
Sergiu Klainerman and Sigmund Selberg, \emph{Bilinear estimates and
  applications to nonlinear wave equations}, Commun. Contemp. Math. \textbf{4}
  (2002), no.~2, 223--295. \MR{1901147}

\bibitem{Koch2005}
Herbert Koch and Daniel Tataru, \emph{Dispersive estimates for principally
  normal pseudodifferential operators}, Comm. Pure Appl. Math. \textbf{58}
  (2005), no.~2, 217--284. \MR{2094851 (2005m:35323)}

\bibitem{Koch2016}
Herbert Koch and Daniel Tataru, \emph{Conserved energies for the cubic {NLS} in
  1-d}, preprint: \texttt{arXiv:1607.02534[math.AP]} (2016).

\bibitem{Koch2014}
Herbert Koch, Daniel Tataru, and Monica Visan, \emph{Dispersive equations and
  nonlinear waves: Generalized {K}orteweg--de {V}ries, nonlinear
  {S}chr{\"o}dinger, wave and {S}chr{\"o}dinger maps}, vol.~45, Springer Basel,
  2014.

\bibitem{Krieger2003}
Joachim Krieger, \emph{Null-form estimates and nonlinear waves}, Adv.
  Differential Equations \textbf{8} (2003), no.~10, 1193--1236. \MR{2016680}

\bibitem{Krieger2004}
\bysame, \emph{Global regularity of wave maps from {$\bold R^{2+1}$} to
  {$H^2$}. {S}mall energy}, Comm. Math. Phys. \textbf{250} (2004), no.~3,
  507--580. \MR{2094472}

\bibitem{Krieger2012}
Joachim Krieger and Wilhelm Schlag, \emph{Concentration compactness for
  critical wave maps}, EMS Monographs in Mathematics, European Mathematical
  Society (EMS), Z\"urich, 2012. \MR{2895939}

\bibitem{Krieger2017}
Joachim Krieger and Daniel Tataru, \emph{Global well-posedness for the
  {Y}ang-{M}ills equation in {$4+1$} dimensions. {S}mall energy}, Ann. of Math.
  (2) \textbf{185} (2017), no.~3, 831--893. \MR{3664812}

\bibitem{Lee2010}
Sanghyuk Lee and Ana Vargas, \emph{Restriction estimates for some surfaces with
  vanishing curvatures}, J. Funct. Anal. \textbf{258} (2010), no.~9,
  2884--2909. \MR{2595728}

\bibitem{Oh2016}
Sung-Jin Oh and Daniel Tataru, \emph{Global well-posedness and scattering of
  the {$(4+1)$}-dimensional {M}axwell-{K}lein-{G}ordon equation}, Invent. Math.
  \textbf{205} (2016), no.~3, 781--877. \MR{3539926}

\bibitem{Peetre1976}
Jaak Peetre, \emph{New thoughts on {B}esov spaces}, Mathematics Department,
  Duke University, Durham, N.C., 1976, Duke University Mathematics Series, No.
  1. \MR{0461123}

\bibitem{Pisier1987}
Gilles Pisier and Quan~Hua Xu, \emph{Random series in the real interpolation
  spaces between the spaces {$v_p$}}, Geometrical aspects of functional
  analysis (1985/86), Lecture Notes in Math., vol. 1267, Springer, Berlin,
  1987, pp.~185--209. \MR{907695}

\bibitem{Shatah1998}
Jalal Shatah and Michael Struwe, \emph{Geometric wave equations}, Courant
  Lecture Notes in Mathematics, vol.~2, New York University, Courant Institute
  of Mathematical Sciences, New York; American Mathematical Society,
  Providence, RI, 1998. \MR{1674843}

\bibitem{Sterbenz2010b}
Jacob Sterbenz and Daniel Tataru, \emph{Energy dispersed large data wave maps
  in {$2+1$} dimensions}, Comm. Math. Phys. \textbf{298} (2010), no.~1,
  139--230. \MR{2657817}

\bibitem{Sterbenz2010a}
\bysame, \emph{Regularity of wave-maps in dimension {$2+1$}}, Comm. Math. Phys.
  \textbf{298} (2010), no.~1, 231--264. \MR{2657818}

\bibitem{Tao2001b}
Terence Tao, \emph{Endpoint bilinear restriction theorems for the cone, and
  some sharp null form estimates}, Math. Z. \textbf{238} (2001), no.~2,
  215--268. \MR{1865417 (2003a:42010)}

\bibitem{Tao2001}
\bysame, \emph{Global regularity of wave maps ii. small energy in two
  dimensions}, Communications in Mathematical Physics \textbf{224} (2001),
  no.~2, 443--544.

\bibitem{Tao2006}
\bysame, \emph{Nonlinear dispersive equations}, CBMS Regional Conference Series
  in Mathematics, vol. 106, Published for the Conference Board of the
  Mathematical Sciences, Washington, DC; by the American Mathematical Society,
  Providence, RI, 2006, Local and global analysis. \MR{2233925}

\bibitem{Tataru2001}
Daniel Tataru, \emph{On global existence and scattering for the wave maps
  equation}, Amer. J. Math. \textbf{123} (2001), no.~1, 37--77. \MR{1827277
  (2002c:58045)}

\bibitem{Tataru2004}
\bysame, \emph{The wave maps equation}, Bull. Amer. Math. Soc. (N.S.)
  \textbf{41} (2004), no.~2, 185--204. \MR{2043751}

\bibitem{Tataru2005}
\bysame, \emph{Rough solutions for the wave maps equation}, Amer. J. Math.
  \textbf{127} (2005), no.~2, 293--377. \MR{2130618 (2006a:58034)}

\bibitem{Wiener1924}
Norbert Wiener, \emph{The quadratic variation of a function and its {F}ourier
  coefficients}, Journal of Mathematics and Physics \textbf{3} (1924), no.~2,
  72--94.

\bibitem{Wolff2001}
Thomas Wolff, \emph{A sharp bilinear cone restriction estimate}, Ann. of Math.
  (2) \textbf{153} (2001), no.~3, 661--698. \MR{1836285}

\bibitem{Young1936}
Laurence~C. Young, \emph{An inequality of the {H}\"older type, connected with
  {S}tieltjes integration}, Acta Math. \textbf{67} (1936), no.~1, 251--282.
  \MR{1555421}

\end{thebibliography}
\end{document}